\documentclass[a4paper,reqno,11pt]{amsart}
\usepackage{amsfonts,amsmath,amsthm,amssymb,stmaryrd}
\usepackage{mathrsfs}

\usepackage[left=1 in, right=1 in,top=1 in, bottom=1 in]{geometry}

\allowdisplaybreaks

\numberwithin{equation}{section}
\usepackage{color}

\newtheorem{theorem}{Theorem}[section]
\newtheorem{lemma}[theorem]{Lemma}
\newtheorem{proposition}[theorem]{Proposition}

\newtheorem{remark}[theorem]{Remark}

\theoremstyle{definition}

\newcommand{\bR}{\mathbb R}

\newcommand{\Div}{\operatorname{div}}
\newcommand{\curl}{\operatorname{curl}}

\renewcommand{\epsilon}{\varepsilon}

\providecommand{\abs}[1]{\left\vert#1\right\vert}
\providecommand{\norm}[1]{\left\Vert#1\right\Vert}

\DeclareMathOperator{\diverge}{div}

\providecommand{\norm}[1]{\left\Vert#1\right\Vert}

\def\dt{\partial_t}

\def\pa{\partial}

\def\RRvert2{\right \vert\! \right\vert}
\def\Lvert3{\left \vert\!\left\vert\!\left\vert}
\def\Rvert3{\right \vert\!\right\vert\!\right\vert}

\def\nab{\nabla}

\def\dt{\partial_t}

\def\hal{\frac{1}{2}}
\def\ep{\varepsilon}

\def\ls{\lesssim}

\def\naba{\nab_{\mathcal{A}}}
\def\nabak{\nab_{\mathcal{A}^\kappa}}
\def\diva{\diverge_{\mathcal{A}}}
\def\divak{\diverge_{\mathcal{A}^\kappa}}

\def\V{\tilde{\mathcal{V}}}

\def\a{\mathcal{A}}

\def\i{\mathcal{I}}
\def\j{\mathcal{J}}

\def\fj1{\mathcal{J}^{-1}}

\def\bp{\bar\partial}

\def\nak{\nabla_{\mathcal{A}^\kappa}}

\def\ak{{\mathcal{A}^\kappa}}

\def\fk{{\psi^\kappa}}

\makeindex

\title[Free-surface incompressible ideal MHD]{On the construction of solutions to the free-surface incompressible ideal magnetohydrodynamic equations}

\author{Xumin Gu}
\address[X. Gu]{School of Mathematics\\ Shanghai University of Finance and Economics\\
Shanghai 200433, China
\newline \indent and
\newline \indent The Institute of Mathematical Sciences\\
The Chinese University of Hong Kong\\
Shatin, NT, Hong Kong}
\email{gu.xumin@shufe.edu.cn}

\author{Yanjin Wang}
\address[Y. J. Wang]{School of Mathematical Sciences\\
Xiamen University\\
Xiamen, Fujian 361005, China
\newline \indent and
\newline \indent The Institute of Mathematical Sciences\\
The Chinese University of Hong Kong\\
Shatin, NT, Hong Kong}
\email{yanjin$\_$wang@xmu.edu.cn}
\thanks{X. Gu was supported by the National Natural Science Foundation of China (No.~11601305).}

\thanks{Y. J. Wang was supported by the Fujian Province Natural Science Funds for Distinguished Young Scholar (No.~2015J06001), the National Natural Science Foundation of China (No.~11531010), and a Foundation for the Author of National Excellent Doctoral Dissertation of PR China (No.~201418).} %Xiamen University


\begin{document}

%\tableofcontents

\begin{abstract}
We consider a free boundary problem for the incompressible ideal magnetohydrodynamic equations that describes the motion of the plasma in vacuum. The magnetic field is tangent and the total pressure vanishes along the plasma-vacuum interface. Under the Taylor sign condition of the total pressure on the free surface, we prove the local well-posedness of the problem in Sobolev spaces.
\end{abstract}

\maketitle

%%%%%%%%%%%%%%%%%%%%%%%%%%%%%%%%%%%%%%%%%%%%%%%%%%%%%%%%%%%%%%%%%%%%%%%%
\section{Introduction}
%%%%%%%%%%%%%%%%%%%%%%%%%%%%%%%%%%%%%%%%%%%%%%%%%%%%%%%%%%%%%%%%%%%%%%%%

%%%%%%%%%%%%%%%%%%%%%%%%%%%%%%%%%%%%%%%%%%%%%%%%%%%%%%%%%%%%%%%%%%%%%%%%
\subsection{Eulerian formulation}
%%%%%%%%%%%%%%%%%%%%%%%%%%%%%%%%%%%%%%%%%%%%%%%%%%%%%%%%%%%%%%%%%%%%%%%%

In this paper we construct local solutions to the free boundary problem for the incompressible ideal magnetohydrodynamic equations (MHD):
\begin{equation}
  \label{mhd}
\begin{cases}
\dt u  + u \cdot \nabla u + \nabla \big(p+\dfrac{1}{2}|B|^2\big) = B\cdot\nabla B &\text{in } \Omega (t),\\
\diverge u =0 &\text{in } \Omega (t),\\
\dt B + u \cdot \nabla B= B\cdot\nabla u &\text{in } \Omega (t),\\
\diverge B =0 &\text{in } \Omega (t).
\end{cases}
\end{equation}
In the equations \eqref{mhd}, $u$ is the velocity field, $B$ is the magnetic field, and $p$ is the pressure function of the fluid which occupies the moving bounded domain $\Omega(t)$. We require the following boundary conditions on the free surface $\Gamma(t):=\pa \Omega (t)$:
\begin{equation} \label{mhd1}
V(\Gamma(t)) = u\cdot n \quad\text{on } \Gamma(t)
\end{equation}
and
\begin{equation} \label{mhd2}
  p+\dfrac{1}{2}|B|^2=0 ,\quad
  B\cdot n=0 \quad\text{on }\Gamma(t).
\end{equation}
The equation \eqref{mhd1} is called the kinematic boundary condition which states that the free surface $\Gamma(t)$ moves with the velocity of the fluid, where $V(\Gamma(t))$ denote the normal velocity of $\Gamma(t)$
and $n$ denotes the outward unit normal of $\Gamma(t)$. Finally, we impose the initial condition
\begin{equation} \label{mhd3}
(u, B)=(u_0,B_0) \text{ on }\Omega(0),\text{ and }\Omega(0)=\Omega.
\end{equation}

The system \eqref{mhd}--\eqref{mhd3} can be used to model the motion of the plasma in vacuum, which is a special case of the classical plasma-vacuum interface problem \cite{Go_04}. In the general setting, the plasma region $\Omega(t)$ is surrounded by the vacuum region $\Omega_v(t)$, and the moving plasma-vacuum interface $\Gamma(t)=\partial\Omega(t)$ does not intersect with the outer wall $\partial\Omega_w$, where $\Omega_w=\Omega(t)\cup \Gamma(t)\cup \Omega_v(t)$ is a fixed domain. When the characteristic plasma velocity is very small compared to the speed of sound, the motion of the plasma is governed by the incompressible ideal MHD in $\Omega(t)$, i.e., \eqref{mhd}. In the vacuum region $\Omega_v(t)$, we neglect the displacement current in the Maxwell equations as usual in the non-relativistic MHD and assume the pre-Maxwell equations:
\begin{equation}
\label{premax}
\begin{cases}
\curl \mathcal{B} =0,\quad \Div \mathcal{B} =0 &\text{in }\Omega_v(t),\\
\curl \mathcal{E} =-\dt \mathcal{B},\quad \Div \mathcal{E} =0&\text{in }\Omega_v(t).
\end{cases}
\end{equation}
In the equations \eqref{premax}, $\mathcal{B}$ and $\mathcal{E}$ denotes the magnetic and electric fields in vacuum, respectively. The motion of the plasma is connected with the vacuum through the jump condition on the interface $\Gamma(t)$:
\begin{equation}
  \left(\big(p  + \dfrac{1}{2}|B|^2\big)I  - B\otimes B\right)n=\left(  \dfrac{1}{2}|\mathcal{B}|^2I  -\mathcal{B}\otimes\mathcal{ B}\right)n\quad\text{on } \Gamma(t)
\end{equation}
and
\begin{equation}
 (B-\mathcal{B})\cdot n  = 0,\quad   (E-  \mathcal{E}) \times n=  -(u \cdot n) (B- \mathcal{B}) \quad\text{on } \Gamma(t).
\end{equation}
Here  $E$ is the electric field of the plasma, i.e., $E=-u\times B$. We also require the jump condition on $\partial\Omega_w$:
\begin{equation}
 (\mathcal{B}-\hat{\mathcal{B}})\cdot \nu  = 0,\quad (\mathcal E-\hat{\mathcal E})\times \nu=0 \quad\text{on } \partial\Omega_w.
\end{equation}
Here $\hat{\mathcal{B}}$ and $\hat{\mathcal E}$ denotes the magnetic and electric fields outside the wall $\partial\Omega_w$ and $\nu$ is the outward unit normal of $\partial\Omega_w$.

The well-posedness of the general plasma-vacuum interface problem is still an open question. In this paper we restrict to a special case when $\Omega_w$ is a bounded simply connected domain and $\pa\Omega_w$ is a perfectly conducting wall and when the plasma is a perfect conductor \cite{Go_04}. In this setting, we have the following boundary conditions
\begin{equation}
\label{B1}
\mathcal{B}\cdot \nu=0,\quad \mathcal{E}\times\nu=0, \quad\text{on } \partial\Omega_w
\end{equation}
and
\begin{equation}
\label{B2}
B\cdot n=\mathcal{B}\cdot n=0\quad \text{on } \Gamma(t).
\end{equation}
We thus find that $\mathcal{B}=0$ (and then $\mathcal{E}=0$) and hence the plasma-vacuum interface problem reduces to the free boundary problem \eqref{mhd}--\eqref{mhd3}. However, even for this special case, up to our knowledge, there does not exist complete well-posedness theory. Our purpose of this paper is to establishing the local well-posedness for this problem.

%%%%%%%%%%%%%%%%%%%%%%%%%%%%%%%%%%%%%%%%%%%%%%%%%%%%%%%%%%%%%%%%%%%%%%%
\subsection{Lagrangian reformulation}\label{lagrangian}
%%%%%%%%%%%%%%%%%%%%%%%%%%%%%%%%%%%%%%%%%%%%%%%%%%%%%%%%%%%%%%%%%%%%%%%

We transform the Eulerian problem \eqref{mhd}--\eqref{mhd3} on the moving domain $\Omega(t)$ to be one on the fixed domain $\Omega$ by the use of Lagrangian coordinates. Let $\eta(x,t)\in\Omega(t)$ denote the ``position" of the fluid particle $x$ at time $t$, i.e.,
\begin{equation}
\begin{cases}
\partial_t\eta(x,t)=u(\eta(x,t),t),\quad t>0,
\\ \eta(x,0)=x.
\end{cases}
\end{equation}
We assume that $\eta(\cdot,t)$ is invertible and define the Lagrangian unknowns on $\Omega$:
\begin{equation}
v(x,t)=u(\eta(x,t),t),\ b(x,t)=B(\eta(x,t),t) \text{ and } q(x,t)=p(\eta(x,t),t)+\frac{1}{2}|b(x,t)|^2.
\end{equation}
Then in Lagrangian coordinates, the problem \eqref{mhd}--\eqref{mhd3} becomes the following:
\begin{equation}\label{eq:mhdo}
\begin{cases}
\partial_t\eta =v &\text{in } \Omega,\\
\partial_tv  +\naba q =b\cdot\naba b  &\text{in } \Omega,\\
 \diva v = 0 &\text{in  }\Omega,\\
\partial_tb =b\cdot\naba v &\text{in } \Omega,\\
\diva b =0 &\text{in  }\Omega,\\
q=0&\text{on }\Gamma:=\partial\Omega,
\\ b \cdot n  =0 &\text{on }\Gamma,\\
 (\eta,v, b)\mid_{t=0} =(\text{Id}, v_0, b_0).
 \end{cases}
\end{equation}
Here the matrix $\mathcal A=\mathcal A(\eta):=(\nabla\eta)^{-T}$, the differential operators $\naba:=(\partial_1^{\a}, \partial_2^{\a}, \partial_3^{\a})$ with $\partial_i^{\a}=\a_{ij}\partial_j$ and $\diva g=\a_{ij}\partial_jg_i$. $n=\mathcal{A}N/\abs{\a N}$, where $N$ is the outward unit normal of $\Gamma$. Note that the kinematic boundary condition \eqref{mhd1} is automatically satisfied by the first equation in \eqref{eq:mhdo}.

We shall find out the conserved quantities for the system \eqref{eq:mhdo}. These quantities will help us reformulate the system in a proper way, and the reformulation will be more suitable for our construction of solutions. To begin with, we denote $J = {\rm det}(\nabla\eta)$, the Jacobian of the coordinate
transformation. Then we have $\dt J=0$, which implies $J=1$. Next, applying $\a^T$ to the fourth equation in \eqref{eq:mhdo}, we obtain
\begin{equation}
\a_{ji}\partial_t b_j=\a_{ji}b_m\a_{m\ell} \partial_\ell v_j =\a_{ji}b_m\a_{m\ell}\partial_t(\partial_\ell\eta_j)
=- \partial_t\a_{ji}b_m\a_{m\ell}\partial_\ell\eta_j=-\partial_t\a_{ji} b_j,
\end{equation}
which implies that
\begin{equation}\label{re00}
\partial_t (\a^T b)=0.
\end{equation}
Since $\eta(x,0)=x$ and $\a=(\nabla\eta)^{-T}$, we then have
\begin{equation}\label{hformula}
b=b_0\cdot\nabla \eta.
\end{equation}
It can be directly verified that $\diva b =0$ in $\Omega$ and $b \cdot n  =0 $ on $\Gamma$ if  $\Div b_0=0$ in $\Omega$ and $b_0\cdot N=0$ on $\Gamma$. That is, we must regard these two conditions as the restriction on the initial data.

Consequently, \eqref{hformula} motivates us to eliminate the magnetic field $b$ from the system \eqref{eq:mhdo}. Indeed, we may rewrite the Lorentz force term:
\begin{equation}
b\cdot\nabla_\a b=b_j\a_{j\ell}\partial_\ell b=b_{0\ell}\partial_\ell(b_{0m}\partial_m\eta)
 :=(b_0\cdot\nabla)^2\eta.
\end{equation}
Then the system \eqref{eq:mhdo} can be reformulated equivalently as a free-surface incompressible Euler system with a forcing term induced by the flow map:
\begin{equation}
\label{eq:mhd}
\begin{cases}
\partial_t\eta =v &\text{in } \Omega,\\
\partial_tv  +\naba q =(b_0\cdot\nabla)^2\eta  &\text{in } \Omega,\\
 \diva v = 0 &\text{in  }\Omega,\\
 q=0 & \text{on  }\Gamma,\\
 (\eta,v)\mid_{t=0} =(\text{Id}, v_0).
 \end{cases}
\end{equation}
In the system \eqref{eq:mhd}, the initial magnetic field $b_0$ can be regarded as a parameter vector that satisfies
\begin{equation}
\label{bcond}
\diverge b_0=0 \text{ in }\Omega, \text{ and }b_0\cdot N=0\text{ on }\Gamma.
\end{equation}

\begin{remark}
It is possible to derive the a priori estimates for the original formulation \eqref{eq:mhdo}, however, as one will see, it is crucial for us to work with the equivalent reformulation \eqref{eq:mhd}, which enables us to construct approximate solutions that are asymptotically consistent with the a priori estimates.
\end{remark}

\subsection{Previous works}
 Free boundary problems in fluid mechanics have important physical background and have been studied intensively in the mathematical community. There are a huge amount of mathematical works, and we only mention briefly some of them below that are closely related to the present work,  that is, those of the incompressible Euler equations and the related ideal MHD models.

 For the incompressible Euler equations, the early works were focused on the irrotational fluids, which began
with the pioneering work of Nalimov \cite{N} of the local well-posedness for the small initial data and was generalized to the
general initial data by the breakthrough of Wu \cite{Wu1,Wu2} (see also Lannes \cite{Lannes}). For the irrotational inviscid fluids, certain dispersive effects can be used to establish the global well-posedness for the small initial data; we refer to  Wu \cite{Wu3,Wu4}, Germain, Masmoudi and Shatah \cite{GMS1, GMS2}, Ionescu and Pusateri \cite{IP,IP2} and Alazard
and Delort \cite{AD}. For the general incompressible Euler equations, the first local well-posedness in 3D was obtained by Lindblad \cite{Lindblad05} for the case without surface tension (see Christodoulou and Lindblad \cite{CL_00} for the a priori estimates) and
by Coutand and Shkoller \cite{CS07} for the case with (and without) surface tension. We also refer to the results of Shatah and Zeng \cite{SZ} and Zhang and Zhang
\cite{ZZ}. Recently, the well-posedness in conormal Sobolev spaces can be found by the the inviscid limit of the free-surface incompressible Navier-Stokes equations, see Masmoudi and Rousset \cite{MasRou} and  Wang and Xin \cite{Wang_15}.

%For compressible fluids, in \cite{Makino_1986}, Makino proved the local-in-time existence of solutions with boundary condition $\rho=0$ for some non-physical restrictions on the initial data. Lindblad \cite{Lind_2005} considered the compressible liquid for general case of initial data. The local existence theory of classical solutions featuring the physical vacuum boundary was established by Coutand-Shkoller \cite{DS_2009, DS_2010} and Jang-Masmoudi \cite{MJ_2009, MJ_2010} independently. We also refer the readers to \cite{Gu_2011, Gu_2016, J_13, MJ_12, Xin_2013} and references therein for more works for the compressible fluid system, such as the Euler-Poisson equations.

The study of free boundary problems for the ideal MHD models seems far from being complete; it attracts many research interests, but up to now only few well-posedness theory for the nonlinear problem could be found. For the general plasma-vacuum interface model, the well-posedness
of the nonlinear compressible problem was recently proved in Secchi and Trakhinin \cite{Secchi_14} by
the Nash-Moser iteration based on the previous results on the linearized problem \cite{Trak_10,Secchi_13}. The well-posedness of the linearized incompressible problem was proved by Morando, Trakhinin and Trebeschi \cite{Mo_14}, but so far no well-posedness for the nonlinear problem can be found in the literature.  For the problem \eqref{mhd}--\eqref{mhd3}, under the assumption that the strength of the magnetic field is constant on the
free surface Hao and Luo \cite{Hao_13} proved the a priori estimates by adopting a geometrical point of view \cite{CL_00}, and this assumption was later removed by Hao \cite{Hao_16}. Finally, we also mention some works about the current-vortex sheet problem, which describes a velocity and magnet field discontinuity in two ideal MHD flows. The nonlinear stability of compressible current-vortex sheets was solved independently by Chen and Wang \cite{Chen_08} and Trakinin \cite{Trak_09} by using the Nash-Moser iteration. For incompressible current-vortex sheets, Coulombel, Morando, Secchi and Trebeschi \cite{CMST} proved an a priori estimate for the nonlinear problem under a strong stability condition, and  Sun, Wang and Zhang \cite{Sun_15} solved the nonlinear stability very recently.

\section{Main results}

\subsection{Statement of the result}
To avoid the use of local coordinate charts necessary for arbitrary geometries, for simplicity, we assume that the initial domain is given by
\begin{equation}
\label{domain}
\Omega=\mathbb{T}^2\times(0,1),
\end{equation}
where $\mathbb{T}^2$ denotes the 2-torus. This permits the use of one global Cartesian coordinate system. The boundary of $\Omega$ is then given by the horizontally flat bottom and top:
\begin{equation}
\Gamma =\mathbb{T}^2 \times (\{0\}\cup \{1\}).
\end{equation}
We denote $N$ by the outward unit normal vector of $\Gamma$:
\begin{equation}
N=e_3 \text{ when }x_3=1, \text{ and }N=-e_3\text{ when }x_3=0.
\end{equation}

Before stating our results of this paper, we may refer the readers to our notations and conveniences in Section \ref{nota sec}. As for the free-surface incompressible Euler equations, a Taylor sign condition is needed
to get the local well-posedness for \eqref{eq:mhd}.
Given the data $(v_0, b_0)\in H^4(\Omega)$ with $\Div v_0=\Div b_0=0$, we define the initial pressure function $q_0$ as the solution to the elliptic problem
\begin{equation}
\begin{cases}
-\Delta q_0=\partial_j v_{0i}\partial_i {v_0}_j-\partial_j b_{0i}\partial_i b_{0j}&\text{in }\Omega,\\
q_0=0&\text{on }\Gamma.
\end{cases}
\end{equation}
 Hence, we assume initially
\begin{equation}\label{taylor}
-\nabla q_0\cdot N\geq \lambda>0 \text{ on }\Gamma.
\end{equation}

We define the higher order energy functional
\begin{equation}
\label{edef}
\mathfrak{E}(t)=\norm{v}_4^2+\norm{\eta}_4^2+\norm{b_0\cdot\nabla \eta}_4^2+\abs{\bar\partial^4\eta\cdot n }_0^2.
\end{equation}
Then the main result in this paper is stated as follows.
\begin{theorem}\label{mainthm}
Suppose that the initial data $ v_0 \in H^4(\Omega)$ with $\Div v_0=0$ and $b_0 \in H^4(\Omega)$ satisfies \eqref{bcond} and that the Taylor sign condition \eqref{taylor} holds initially. Then there exists a $T_0>0$ and a unique solution $(v, q, \eta)$ to  \eqref{eq:mhd} on the time interval $[0, T_0]$ which satisfies
\begin{equation}\label{enesti}
\sup_{t \in [0,T_0]} \mathfrak E(t) \leq P\left(\norm{v_0}_4^2+\norm{b_0}_4^2\right),
\end{equation}
where $P$ is a generic polynomial.

\end{theorem}
\subsection{Strategy of the proof}

The strategy of proving the local well-posedness for the inviscid free boundary problems consists of three main parts: the a priori estimates in certain energy functional spaces, a suitable approximate problem which is asymptotically consistent with the a priori estimates, and the construction of solutions to the approximate problem. Each part can be highly nontrivial, especially the second part. Our approach of constructing solutions to the incompressible MHD equations \eqref{eq:mhd} is motivated by the one for the incompressible Euler equations without surface tension; however, we will develop new ideas to overcome several new difficulties caused by the presence of the magnetic field as illustrated below.

In the usual derivation of the a priori tangential energy estimates of \eqref{eq:mhd} in the $H^4$ setting, one deduces
\begin{equation}
\begin{split}
&\hal\dfrac{d}{dt} \int_{\Omega}\abs{\bar\partial^4 v}^2+\abs{\bar\partial^4 ( b_0\cdot\nabla\eta)}^2+\int_{\Gamma}(-\nabla q\cdot N)\bar\partial^4{\eta}_j{\mathcal A}_{j3} q\bar\partial^4 v_i{\mathcal A}_{i3}
\\&\quad \approx \underbrace{\int_\Omega \bar\partial^4 \nabla\eta \bar\partial^4  \eta+ \bar\partial^4 \nabla\eta\bar\partial^4  q}_{\mathcal{R}_Q}+l.o.t.,
\end{split}
\end{equation}
where $l.o.t.$  denotes integrals consisting of lower-order terms. Since $\dt \eta=v$, one has
\begin{equation}\label{taylores}
\begin{split}
\int_{\Gamma}(-\nabla q\cdot N)\bar\partial^4{\eta}_j{\mathcal A}_{j3}\bar\partial^4 v_i{\mathcal A}_{i3}&=\dfrac{1}{2}\dfrac{d}{dt}\int_{\Gamma}(-\nabla q\cdot N)\abs{\bar\partial^4 \eta_i{\mathcal{A}}_{i3}}^2
\\&\quad+\underbrace{\int_{\Gamma}  (-\nabla q\cdot N) \a_{j3}\bp^4   \eta_j   \a_{im}\pa_m   v_\ell \a_{\ell3} \bar\partial^4  \eta_i}_{\widetilde{\mathcal{R}}_Q}+l.o.t..
\end{split}
\end{equation}
Under the Taylor sign condition, one gets that $\bar\partial^4 \eta\cdot n\in L^2(\Gamma)$ which is $1/2$ higher regular than $v$. For the incompressible Euler equations, through a careful study of the vorticity equation, it can be shown that $\curl\eta\in H^{4-1/2}(\Omega)$ which leads to $\eta\in H^{4+1/2}(\Omega)$ (and hence $q\in H^{4+1/2}(\Omega)$), and thus $\mathcal{R}_Q$ and $\widetilde{\mathcal{R}}_Q$ can be controlled, which in turn closes the a priori estimates in the energy functional of $\norm{v}_{4}^2+\norm{\eta}_{4+1/2}^2$, see \cite{CS07,DS_10}. However, for the incompressible MHD equations, we cannot prove that $\curl\eta\in H^{4-1/2}(\Omega)$ due to the presence of the magnetic field. Indeed, in the study of the vorticity equation,  if one wishes to show $\curl\eta\in H^{4-1/2}(\Omega)$, then one needs to estimate $\norm{b_0\cdot\nabla\eta}_{4+1/2}^2$ which is out of control since one does not have $\bar\partial^4 (b_0\cdot\nabla\eta)\cdot n\in L^2(\Gamma)$ unless $b_0=0$ on $\Gamma$. As a consequence, we can only hope to close the a priori estimates in the energy functional $\mathfrak{E}(t)$ defined by \eqref{edef}. However, this yields a loss of derivatives in estimating $\mathcal{R}_Q$ and $\widetilde{\mathcal{R}}_Q$. Our idea to overcome this difficulty is, motivated by \cite{MasRou,Wang_15}, to use Alinac's good unknowns $
\mathcal{V} =\bar\partial^4 v- \bp^4\eta \cdot\nabla_\a v$ and $ \mathcal{Q} =\bar\partial^4 q- \bp^4\eta \cdot\nabla_\a q$, which derives a crucial cancellation observed by Alinhac \cite{Alinhac}, i.e., when considering the equations for $\mathcal{V} $ and $Q $, the terms $\mathcal{R}_Q$ and $\widetilde{\mathcal{R}}_Q$ disappear. This allows us to close the a priori estimates of \eqref{eq:mhd} in $\mathfrak{E}(t)$.

It should be noted that \eqref{taylores} relies heavily on the geometric transport-type structure of the nonlinear problem, which is lost in the linearized problem (c.f. $\bp^4{\eta}$ stems from $\bp^4 \a$) and then causes losing derivatives on the boundary estimates in the construction of approximate solutions by the linearization iteration \cite{CL_00,Lindblad05}. Indeed, the construction of a sequence of approximate solutions which is asymptotically consistent with a priori estimates is usually highly nontrivial for inviscid free boundary problems. For the incompressible Euler equations, the authors \cite{CS07,DS_10} introduced the horizontal convolution mollifier $\Lambda_{\kappa}$ to construct the nonlinear $\kappa$-approximate problem by approximating $\a=\a(\eta)$ with $\a^\kappa=\a(\Lambda_\kappa^2\eta)$. Note that in this approximation the type of estimates \eqref{taylores} is retained, recalling our use of good unknowns,
\begin{align} \label{ine}
\int_{\Gamma}(-\nabla q\cdot N)\bar\partial^4{\Lambda_\kappa^2\eta}_j{\mathcal A}^\kappa_{j3}  \mathcal{V}_i{\mathcal A}^\kappa_{i3}&=\dfrac{1}{2}\dfrac{d}{dt}\int_{\Gamma}(-\nabla q\cdot N)\abs{\bar\partial^4 \Lambda_{\kappa} \eta_i\mathcal A^{\kappa}_{i3}}^2 \\&\quad+\int_{\Gamma}  (-\nabla q\cdot N) \a^\kappa_{j3}\bp^4 \Lambda_\kappa \eta_j   \a^{\kappa}_{i\alpha}\pa_\alpha \Lambda_{\kappa}^2 v_\ell \a^\kappa_{\ell3} \bar\partial^4  \Lambda_\kappa \eta_i \nonumber
\\&\quad+
\int_{\Gamma}   \nabla q\cdot N  \a^\kappa_{j3}\bp^4 \Lambda_\kappa^2 \eta_j   \a^{\kappa}_{i\alpha}\pa_\alpha   v_\ell \a^\kappa_{\ell3} \bar\partial^4  \Lambda_\kappa^2 \eta_i+l.o.t.. \nonumber
\end{align}
It should be pointed out that in the a priori estimates of the original problem \eqref{eq:mhd} (i.e., $\kappa=0$) the last two terms in \eqref{ine} are cancelled out, while when $\kappa>0$ we do not have the cancellation. And for the incompressible MHD equations, neither of these two terms can be controlled independently of $\kappa$. Our idea to overcome this difficulty is to approximate the flow map as \begin{equation}
\partial_t\eta=v+\fk,
\end{equation}
where $\fk$ is the {\it modification term}, which is defined by \eqref{etaaa} and tends to zero as $\kappa\rightarrow0$. The crucial point is  that in the estimates of \eqref{ine} we then get an additional term
\begin{equation}
\int_{\Gamma}(-\nabla q\cdot N)\bar\partial^4{\Lambda_\kappa^2\eta}_j{\mathcal A}^\kappa_{j3}\bar\partial^4 \psi^\kappa_i{\mathcal A}^\kappa_{i3},
\end{equation}
which eliminates those two troublesome terms, up to an error of $l.o.t.$ terms.
This can then finish the tangential energy estimates.  When doing the curl and divergence estimates, unfortunately, we will encounter another difficulty due to the presence of the magnetic field. More precisely, we will need to estimate $\norm{b_0\cdot\nabla  \Lambda_{\kappa}^2\eta}_4$, which is out of control due to the commutator term $\norm{[b_0\cdot\nabla, \Lambda_{\kappa}^2]\eta}_4$. This term is bad since in general $b_{03}$ only vanishes on $\Gamma$ but may not vanish in $\Omega$. Our way to overcome this difficulty is to approximate $\a=\a(\eta)$ in a slightly different way, i.e., by $\ak=\a(\eta^\kappa)$ where $\eta^\kappa$ is the {\it boundary smoother} of $\eta$ defined by \eqref{etadef}. The advantage of this approximation is that it allows the estimates of $\norm{b_0\cdot\nabla  \eta^\kappa}_4$ (see \eqref{tes2}) and preserves the estimates \eqref{ine} since $\eta^{\kappa}=\Lambda_{\kappa}^2\eta$ on $\Gamma$. This then finishes the curl and divergence estimates. Finally, employing the Hodge-type elliptic estimates, we thus close $\kappa$-independent estimates for the nonlinear $\kappa$-approximate problem \eqref{approximate}.

What now remains in the proof of the local well-posedness of \eqref{eq:mhd} is to constructing solutions to the $\kappa$-approximate problem \eqref{approximate}, which is more feasible than the original problem since it avoids losing derivatives on the boundary estimates in the linearization iteration due to the boundary smoothing effect for each fixed $\kappa>0$. For the incompressible Euler equations \cite{CS07,DS_10}, the existence of solutions to the smoothed approximate problem is obtained by using simple transport-type arguments that rely on the pressure gaining one regularity for fixed $\kappa>0$ so that one can define $v$ by $\dt v=-\nabla_\ak q$ in the linearization iteration. However,  for the incompressible MHD equations, we are not able to define $v$ in such a  way since we only have a priori $(b_0\cdot\nabla)^2  \eta\in H^{3}(\Omega)$. Our idea to overcome this difficulty is to smooth $(b_0\cdot\nabla)^2  \eta$ in a proper way. We find that the right choice for us is to further approximate the flow map as
\begin{equation}
\partial_t\eta-\varepsilon (b_0\cdot\nabla)^2\eta =v+\psi^\kappa.
\end{equation}
For fixed $\varepsilon>0$, we have a priori $(b_0\cdot\nabla)^2  \eta\in  L^2(0,T;H^{4}(\Omega))$, which is just sufficient for us to define $v$ by $\dt v=-\nabla_{\tilde\a} q+(b_0\cdot\nabla)^2  \eta$. We find that it will be a bit more convenient to employ this $\varepsilon$-approximation in the linearization. Hence, our construction of solutions to the nonlinear $\kappa$-approximate problem \eqref{approximate} is as follows. We introduce the linearized $\kappa$-approximate problem \eqref{lvapproximate} and employ the second level of approximation by considering the linearized $\varepsilon$-$\kappa$-approximate problem \eqref{epsapproximate}. We first show the existence of solutions to \eqref{epsapproximate} for each $\varepsilon>0$ ($\kappa>0$ is fixed!) by a somewhat standard argument. Then we derive the $\varepsilon$-independent estimates, in a similar way as that for \eqref{approximate}, so as to pass to the limit as $\varepsilon\rightarrow0$ to show the existence of solutions to \eqref{lvapproximate}. Then by the linearization iteration and a contraction argument, we produce solutions to the nonlinear $\kappa$-approximate problem \eqref{approximate}. Consequently, the construction of solutions to the incompressible MHD equations \eqref{eq:mhd} is completed.

\section{Preliminary}
%%%%%%%%%%%%%%%%%%%%%%%%%%%%%%%%%%%%%%%%%%%%%%%%%%%%%%%%%%%%%%%%%%%%%%%%%%%%%%%%%%%%

%\subsection{Notation}
%
%\Blue{FIXME???Move notations to here?}

\subsection{Notation}\label{nota sec}

Einstein's summation convention is used throughout the paper, and repeated
Latin indices $i,j,$ etc., are summed from 1 to 3, and repeated Greek indices
$\alpha,\beta$, etc., are summed from 1 to 2. We will use the Levi-Civita symbol
\begin{equation}\nonumber
\epsilon_{ij\ell}=\left\{\begin{aligned}
1 &,\,\,\text{even permutation of}\,\, \{1,2,3\},\\
-1 &,\,\,\text{odd permutation of}\,\, \{1,2,3\},\\
0 &,\,\,\text{otherwise.} \\
\end{aligned}\right.
\end{equation}

We work with the usual $L^p$ and Sobolev spaces $W^{m,p}$ and $H^m=W^{m,2}$ on both the domain $\Omega$ and the boundary $\Gamma$. For notational simplifications, we denote the norms of these spaces defined on $\Omega$ by $\norm{\cdot}_{L^p}, \norm{\cdot}_{W^{m,p}}$ and $\norm{\cdot}_{m}$, and the norms of these spaces defined on $\Gamma$ by $\abs{\cdot}_{L^p}, \abs{\cdot}_{W^{m,p}}$ and $\abs{\cdot}_{m}$. We also use $\norm{\cdot}_{L^p_TH^m}$ to denote the norm of the space $L^p([0,T];H^m(\Omega))$.

We use $D$ to denote the spatial derives, $\bar{\partial}$ to denote the tangential derivatives, $\Delta$ to denote the Laplacian on $\Omega$ and $\Delta_{\ast}$ to denote the Laplacian on $\Gamma$. We also use $\int_{\Omega} f$, $\int_{\Gamma} f$ as the integrals abbreviation of $\int_{\Omega}f\,dx$, $\int_{\Gamma}f\,dx_{\ast}$, with $x_\ast=(x_1,x_2)$.

We use $C$ to denote generic constants, which only depends on the domain $\Omega$ and the boundary $\Gamma$, and  use $f\ls g$ to denote $f\leq Cg$.  We use $P$ to denote a generic polynomial function of its arguments, and the polynomial coefficients are generic constants $C$.

%\Blue{Please find out all the notations that used below but have not been listed in the subsection.}

\subsection{Product and commutator estimates}

We recall the following product and commutator estimates.
\begin{lemma}
It holds that

\noindent $(i)$ For $|\alpha|=k\geq 0$,
\begin{equation}
\label{co0}
\norm{D^{\alpha}(gh)}_0 \ls \norm{g}_{k}\norm{h}_{[\frac{k}{2}]+2}+\norm{g}_{[\frac{k}{2}]+2}\norm{h}_{k}.
\end{equation}
$(ii)$ For $|\alpha|=k\geq 1$, we define the commutator
\begin{equation}
[D^{\alpha}, g]h = D^{\alpha}(gh)-gD^{\alpha} h.
\end{equation}
Then we have
\begin{align}
\label{co1}
&\norm{[D^{\alpha}, g]h}_0\ls\norm{Dg}_{k-1}\norm{h}_{[\frac{k-1}{2}]+2}+\norm{Dg}_{[\frac{k-1}{2}]+2}\norm{h}_{k-1} ,\\
\label{co3}
&\abs{[D^{\alpha}, g]h}_0\ls\abs{Dg}_{k-1}\abs{h}_{[\frac{k-1}{2}]+\frac{3}{2}}+\abs{Dg}_{[\frac{k-1}{2}]+\frac{3}{2}}\abs{h}_{k-1} .
\end{align}
$(iii)$ For $|\alpha|=k\geq 2$, we define the symmetric commutator
\begin{equation}
\left[D^{\alpha}, g, h\right] = D^{\alpha}(gh)-D^{\alpha}g h-gD^{\alpha} h.
\end{equation}
Then we have
\begin{equation}
\label{co2}
\norm{\left[D^{\alpha}, g, h\right]}_0\ls\norm{Dg}_{k-2}\norm{Dh}_{[\frac{k-2}{2}]+2}+ \norm{Dg}_{[\frac{k-2}{2}]+2}\norm{Dh}_{k-2} .
\end{equation}
\end{lemma}
\begin{proof}
The proof of these estimates is standard; we first use the Leibniz formula to expand these terms as sums of products and then control the $L^2$ norm of each product with the lower order derivative term in $L^\infty\subset H^2 $ (or $L^\infty\subset H^{3/2} $) and the higher order derivative term in $L^2$. See for instance Lemma A.1 of \cite{Wang_15}.
\end{proof}

We will also use the following lemma.
\begin{lemma}
It holds that
\begin{equation}
\label{co123}
\abs{gh}_{1/2} \ls \abs{g}_{W^{1,\infty}}\abs{h}_{1/2}.
\end{equation}
 \end{lemma}
\begin{proof}
It is direct to check that $\abs{gh}_{s} \ls \abs{g}_{W^{1,\infty}}\abs{h}_{s}$ for $s=0,1$. Then the estimate \eqref{co123} follows by the interpolation.
\end{proof}

\subsection{Hodge decomposition elliptic estimates}

Our derivation of high order energy estimates is based on the following Hodge-type elliptic estimates.
\begin{lemma}
\label{hodge}
Let $s\ge 1$, then it holds that
\begin{equation}
\norm{\omega}_s\ls \norm{\omega}_0+\norm{\curl \omega}_{s-1}+\norm{\Div \omega}_{s-1}+\abs{\bar{\partial}\omega \cdot N}_{s-3/2} .\label{hodd}
\end{equation}
\end{lemma}
\begin{proof}
The estimates are well-known and follow from the identity $-\Delta\omega=\curl \curl \omega-\nabla\Div \omega$. We refer the reader to Section 5.9 of \cite{Taylor}.
\end{proof}

\subsection{Normal trace estimates}

For our use of the above Hodge-type elliptic estimates, we also need the following normal trace estimates.
\begin{lemma}\label{normal trace}
It holds that
\begin{equation}\label{gga}
\abs{\bar{\partial}\omega \cdot N}_{-1/2}\ls \norm{\bar{\partial}\omega}_0+\norm{\Div\omega}_0
\end{equation}
\begin{proof}
Let $\phi$ be a scalar function in $H^{1/2}(\Gamma)$, and let $\tilde \phi\in H^1(\Omega)$ be a bounded extension. Then
\begin{equation}\int_\Gamma \bp \omega \cdot N \phi=\int_\Omega \Div( \bp \omega  \tilde\phi)=\int_\Omega   \bp \omega \nabla \tilde\phi - \int_\Omega \Div  \omega \bp \tilde\phi  \ls \left(\norm{\bar{\partial}\omega}_0+\norm{\Div\omega}_0\right) {\abs{\phi}_{1/2}}.
\end{equation}
The estimate \eqref{gga} then follows.
\end{proof}
\end{lemma}

\subsection{Horizontal convolution-by-layers and commutation estimates}

As \cite{CS07,DS_10}, we will use the operation of horizontal convolution-by-layers which is defined as follows. Let $0\le \rho(x_\ast)\in C_0^{\infty}(\bR^2)$ be a standard mollifier such that $\text{spt}(\rho)=\overline{B(0,1)}$ and $\int_{\bR^2} \rho \,dx_{\ast}=1$, with corresponding dilated function $\rho_{\kappa}(x_\ast)=\frac{1}{\kappa^2}\rho(\frac{x_\ast}{\kappa}), \kappa>0$. We then define
\begin{equation}\label{lambdakdef}
\Lambda_{\kappa}g(x_{\ast},x_3)=\int_{\bR^2}\rho_{\kappa}(x_{\ast}-y_{\ast})g(y_{\ast},x_3)\,dy_{\ast}.
\end{equation}
By standard properties of convolution, the following estimates hold:
\begin{align}
&\abs{\Lambda_{\kappa}h}_s\ls \abs{h}_s,\quad  s\ge 0, \label{test3}\\
&\abs{\bar\partial\Lambda_{\kappa}h}_0\ls \dfrac{1}{\kappa^s}\abs{h}_s,\quad 0\le s\le 1.\label{loss}
\end{align}

The following commutator estimates play an important role in the boundary estimates.
\begin{lemma}\label{comm11}
For $\kappa>0$, we define the commutator
\begin{equation}
\left[\Lambda_{\kappa}, h\right]g\equiv \Lambda_{\kappa}(h g)-h\Lambda_{\kappa}g.
\end{equation}
Then we have
\begin{align}
&\abs{[\Lambda_{\kappa}, h]g}_0\ls  \abs{h}_{L^\infty}|g|_0,\label{es0-0}\\
%&\abs{[\Lambda_{\kappa}, h]g}_0\ls\kappa \abs{h}_{W^{1,\infty}}|g|_0,\\
&\abs{[\Lambda_{\kappa}, h]\bar\partial g}_0\ls \abs{h}_{W^{1,\infty}}|g|_0,\label{es1-0}\\
&\abs{[\Lambda_{\kappa}, h]\bar\partial g}_{1/2}\ls \abs{h}_{W^{1,\infty}}\abs{g}_{1/2}.\label{es1-1/2}
\end{align}
\end{lemma}
\begin{proof}
The first estimate follows from the definition \eqref{lambdakdef}. The estimate \eqref{es1-0} can be found in Lemma 5.1 of \cite{DS_10}. To prove the last estimate, we note that
\begin{align}
\nonumber\bar\partial ([\Lambda_{\kappa}, h]\bar\partial g)&=\bar\partial\left (\Lambda_{\kappa}(h\bar\partial g)-h\Lambda_{\kappa} \bar\partial g\right)\\&=\Lambda_{\kappa}(\bar \partial h\bar\partial g)+\Lambda_{\kappa}(h\bar\partial^2 g)-\bar\partial h\Lambda_{\kappa} \bar\partial g-h\Lambda_{\kappa}\bar\partial^2g =[\Lambda_{\kappa}, \bar\partial h]\bar\partial g+[\Lambda_{\kappa}, h]\bar\partial^2 g.
\end{align}
Then by \eqref{es0-0} and \eqref{es1-0}, we have
\begin{align}\label{lekldaf}
\nonumber \abs{\bar\partial ([\Lambda_{\kappa}, h]\bar\partial g)}_0&\leq \abs{[\Lambda_{\kappa}, \bar\partial h]\bar\partial g}_0+\abs{[\Lambda_{\kappa}, h]\bar\partial^2 g}_0\\ &\ls |\bar\partial h|_{L^{\infty}}\abs{\bar\partial g}_0+  \abs{h}_{W^{1,\infty}}\abs{\bar\partial g}_0\ls\abs{h}_{W^{1,\infty}}\abs{g}_1.
\end{align}
Consequently, the estimate \eqref{es1-1/2} follows from \eqref{es1-0} and \eqref{lekldaf} by the interpolation.
\end{proof}

\subsection{Geometric identities}
We recall some useful identities which can be checked directly.
For a nonsingular matrix $\mathcal F$, we have the following identities for differentiating its determinant $J$ and $\mathcal A=\mathcal F^{-T}$:
\begin{align}
\label{dJ}
&\partial J=\dfrac{\partial J}{\partial {\mathcal F}_{ij}}\partial {\mathcal F}_{ij} =  J{\mathcal{A}}_{ij}\partial {\mathcal F}_{ij},\\
&\partial {\mathcal{A}}_{ij}  = -{\mathcal{A}}_{i\ell}\partial {\mathcal F}_{m\ell}{\mathcal{A}}_{mj},
\label{partialF}
\end{align}
where $\partial$ can be $D$, $\bar{\partial}$ and $\partial_t$ operators. Moreover, we have the Piola identity
\begin{equation}\label{polia}
\partial_j\left(J{\mathcal{A}}_{ij}\right) =0.
\end{equation}

\section{Nonlinear $\kappa$-approximate problem}

Our goal of this section is to introduce our approximation of \eqref{eq:mhd} and then derive the uniform estimates for the approximate solutions.

\subsection{The nonlinear approximate $\kappa$-problem}

%Finally, we denote the divergence, curl operators and curl matrix in the Eulerian coordinates as follows
%\begin{subequations}
%\begin{align}
%\label{lag_d}
%\left[\nabla_{\a}g\right]_{ir}&\equiv \mathcal A_{rs} \partial_sg_i,\\
%\label{lag_div}
%\Div_{\a}g &\equiv\mathcal A_{rs}\partial_sg_r,\\
%\label{lag_curl}
%\left[\curl_{\a}g\right]_i &\equiv \epsilon_{ij\ell} \mathcal A_{js}\partial_sg_{\ell}.
%\end{align}
%\end{subequations}
For $\kappa>0$, we consider the following sequence of approximate problems:
\begin{equation}\label{approximate}
\begin{cases}
\partial_t\eta =v+\fk &\text{in } \Omega,\\
\partial_tv  +\nabla_{\a^\kappa} q =(b_0\cdot\nabla)^2\eta  &\text{in } \Omega,\\
 \Div_{\a^\kappa} v = 0 &\text{in  }\Omega,\\
  q=0 & \text{on  }\Gamma,\\
 (\eta,v)\mid_{t=0} =(\text{Id}, v_0).
 \end{cases}
\end{equation}
Here the matrix $\mathcal A^{\kappa}=\a(\eta^\kappa)$ (and $J^{\kappa}$, etc.) with $\eta^\kappa$ the boundary smoother of $\eta$ defined as the solution to the following elliptic equation
\begin{equation}
\label{etadef}
\begin{cases}
-\Delta \eta^{\kappa}=-\Delta\eta &\text{in }\Omega,\\
\eta^{\kappa}=\Lambda_{\kappa}^2\eta &\text{on }\Gamma.
\end{cases}
\end{equation}
In the first equation of \eqref{approximate} we have introduced the modification term $\psi^{\kappa}=\psi^{\kappa}(\eta,v)$ as the solution to the following elliptic equation
\begin{equation}
\label{etaaa}
\begin{cases}
-\Delta \psi^{\kappa}=0&\text{in } \Omega,
\\  \psi^{\kappa}= \Delta_{*}^{-1} \mathbb{P}\left(\Delta_{*}\eta_{j}\a^{\kappa}_{j\alpha}\partial_{\alpha}{\Lambda_{\kappa}^2 v}-\Delta_{*}{\Lambda_{\kappa}^2\eta}_{j}\a^{\kappa}_{j\alpha}\partial_{\alpha} v\right) &\text{on }\Gamma,
\end{cases}
\end{equation}
where $\mathbb{P} f=f- \int_{\mathbb T^2}f$ and $\Delta_{*}^{-1}$ is the standard inverse of the Laplacian $\Delta_\ast$ on $\mathbb T^2$.

\begin{remark}
 Note that the modification term $\fk\rightarrow 0$ as $\kappa\rightarrow 0$. The introduction of $\fk$ is to eliminate two troublesome terms arising in the tangential energy estimates, which combinedly vanish as $\kappa\rightarrow 0$ but are out of control when $\kappa>0$.
\end{remark}

\begin{remark}
 Note that our definition of $\eta^{\kappa}$ is different from $\Lambda_{\kappa}^2\eta$ used in \cite{CS07,DS_10}, it only smooths $\eta$ on the boundary $\Gamma$ but not smooths horizontally in the interior $\Omega$. Though, it is sufficient to restore the nonlinear symmetric structure on the boundary in the tangential energy estimates since our $\eta^{\kappa}$ still equals to $\Lambda_{\kappa}^2\eta$ on $\Gamma$. On the other hand, our definition of $\eta^\kappa$ enables us to have the control of $\norm{b_0\cdot\nabla  \eta^{\kappa}}_4$ which arises in the curl and divergence estimates, while $\norm{b_0\cdot\nabla  \Lambda_{\kappa}^2\eta}_4$ is out of control due to the bad commutator term $\norm{[b_0\cdot\nabla, \Lambda_{\kappa}^2]\eta}_4$ since generally $b_{03}\neq 0$ and the commutator estimates \eqref{es1-0} and \eqref{es1-1/2} only work in tangential directions.
\end{remark}

\subsection{$\kappa$-independent energy estimates}\label{apest}

For each $\kappa>0$, we will show in Section \ref{exist} that there exists a time $T_\kappa>0$ depending on the initial data and $\kappa>0$ such that there is a unique solution $(v, q, \eta)=(v(\kappa), q(\kappa), \eta(\kappa))$ to \eqref{approximate} on the time interval $[0,T_\kappa]$. For notational simplification, we will not explicitly write the dependence of the solution on $\kappa$. The purpose of this section is to derive the $\kappa$-independent estimates of the solutions to \eqref{approximate}, which enables us to consider the limit of this sequence of solutions as $\kappa\rightarrow 0$.

We take the time $T_{\kappa}>0$  sufficiently small so that for $t\in[0,T_{\kappa}]$,
\begin{align}
\label{ini2}
&-\nabla q (t)\cdot N\geq \dfrac{\lambda}{2} \text{ on }\Gamma,\\
\label{inin3}&\abs{J^{\kappa}(t)-1}\leq \dfrac{1}{8} \text{ and } \abs{\a_{ij}^{\kappa}(t)-\delta_{ij}}\leq \dfrac{1}{8} \text{ in }\Omega.
\end{align}
We define the high order energy functional:
\begin{equation}
\mathfrak{E}^{\kappa} =\norm{v}_4^2+\norm{\eta}_4^2+\norm{b_0\cdot\nabla \eta}_4^2+\abs{\bar\partial^4\Lambda_{\kappa}\eta_i\a^{\kappa}_{i3}}_0^2.
\end{equation}
We will prove that $\mathfrak{E}^{\kappa}$ remains bounded on a time interval independent of $\kappa$, which is stated as the following theorem.

\begin{theorem} \label{th43}
There exists a time $T_1$ independent of $\kappa$ such that
\begin{equation}
\label{bound}
\sup_{[0,T_1]}\mathfrak{E}^{\kappa}(t)\leq 2M_0,
\end{equation}
where $M_0=P(\norm{v_0}_4^2+\norm{b_0}_4^2).$
\end{theorem}

%%%%%%%%%%%%%%%%%%%%%%%%%%%%%%%%%%%%%%%%%%%%%%%%%%%%%%%%%%%%%%%%%%%%%%%%%%%%%%%%%%%%%%%
\subsubsection{Preliminary estimates of $\eta^{\kappa}$ and $\psi^{\kappa}$}
%%%%%%%%%%%%%%%%%%%%%%%%%%%%%%%%%%%%%%%%%%%%%%%%%%%%%%%%%%%%%%%%%%%%%%%%%%%%%%%%%%%%%%%

We begin our estimates with the boundary smoother $\eta^{\kappa}$ defined by \eqref{etadef} and the modification term $\psi^{\kappa}$ defined by \eqref{etaaa}.

\begin{lemma}
\label{preest}
The following estimates hold:
\begin{align}
\label{tes1}\norm{\eta^{\kappa}}_4&\ls \norm{\eta}_4,\\
\label{tes2}\norm{b_0\cdot\nabla\eta^{\kappa}}_4^2&\leq P(\norm{\eta}_4,\norm{b_0}_4,\norm{b_0\cdot\nabla\eta}_4),\\
\label{tes0}\norm{\partial_t\eta^{\kappa} }_4&\leq P(\norm{\eta}_4, \norm{v}_4),\\
\label{fest1}\norm{\fk}_{4}&\leq P(\norm{\eta}_4,\norm{v}_3),\\
\label{fest22}\norm{b_0\cdot\nabla \fk}_4&\leq P(\norm{\eta}_4,\norm{v}_4,\norm{b_0\cdot\nabla\eta}_4),\\
\label{fest2}\norm{\partial_t \fk}_4&\leq P(\norm{\eta}_4,\norm{v}_4,\norm{\partial_t v}_3).
\end{align}
\end{lemma}
\begin{proof}
First, the standard elliptic regularity theory on the problem \eqref{etadef}, the trace theorem and the estimate \eqref{test3} yield
\begin{equation*}
\norm{\eta^{\kappa}}_4\ls \norm{\Delta\eta}_2+\abs{\Lambda_{\kappa}^2\eta}_{7/2}\ls\norm{\eta}_4+\abs{\eta}_{7/2}\ls \norm{\eta}_4,
\end{equation*}
which implies \eqref{tes1}. To prove \eqref{tes2}, we apply $b_0\cdot\nabla$ to \eqref{etadef} to find that, since $b_0\cdot N=0$ on $\Gamma$,
\begin{equation*}
\begin{cases}
-\Delta(b_0\cdot\nabla\eta^{\kappa})=-\Delta(b_0\cdot\nabla\eta )-[b_0\cdot\nabla, \Delta]\eta +[b_0\cdot\nabla, \Delta]\eta^{\kappa} &\text{in } \Omega,\\
b_0\cdot\nabla\eta^{\kappa}=b_0\cdot\nabla\Lambda_{\kappa}^2\eta &\text{on } \Gamma.
\end{cases}
\end{equation*}
We then have, since $H^2$ is a multiplicative algebra and by \eqref{tes1},
\begin{equation*}
\begin{split}
\norm{b_0\cdot\nabla\eta^{\kappa}}_4&\ls \norm{-\Delta(b_0\cdot\nabla\eta )-[b_0\cdot\nabla, \Delta]\eta +[b_0\cdot\nabla, \Delta]\eta^{\kappa}}_2
+\abs{b_0\cdot\nabla\Lambda_{\kappa}^2\eta}_{7/2}
\\&\ls \norm{b_0\cdot\nabla\eta}_4+\norm{b_0}_4(\norm{ \eta}_4+\norm{ \eta^{\kappa}}_4)
+\abs{\Lambda_{\kappa}^2(b_0\cdot\nabla\eta)}_{7/2}+\abs{[b_0\cdot\nabla, \Lambda_{\kappa}^2]\eta}_{7/2}
\\&\ls \norm{b_0\cdot\nabla\eta}_4+\norm{b_0}_4\norm{\eta}_4+\abs{b_0\cdot\nabla\eta}_{7/2}
+\abs{b_0}_{7/2}\abs{\eta}_{7/2} \\
&\ls \norm{b_0\cdot\nabla\eta}_4+\norm{b_0}_4\norm{\eta}_4.
\end{split}
\end{equation*}
Here we have used the estimates \eqref{es1-1/2} to estimate
\begin{equation*}
\begin{split}
 \abs{[b_0\cdot\nabla, \Lambda_{\kappa}^2]\eta}_{7/2}&\le \abs{[\Lambda_{\kappa}^2,b_{0\alpha}]\pa_\alpha\eta}_{1/2}+\abs{[\Lambda_{\kappa}^2,b_{0\alpha}]\pa_\alpha \bp^3\eta}_{1/2}
 +\abs{\left[\bp^3,[\Lambda_{\kappa}^2,b_{0\alpha}]\pa_\alpha \right]\eta}_{1/2}
\\&\ls   \abs{b_0}_{W^{1,\infty}}\abs{\eta}_{7/2}
+\abs{b_0}_{7/2}\abs{\eta}_{7/2}\ls \abs{b_0}_{7/2}\abs{\eta}_{7/2}.
\end{split}
\end{equation*}
This proves \eqref{tes2}.

We now turn to prove \eqref{fest1}. By the boundary condition in \eqref{etaaa} and the elliptic theory, we obtain, using the identity \eqref{partialF}, the a priori assumption \eqref{inin3} and the estimates \eqref{tes1},
\begin{equation*}
\begin{split}
\abs{\fk}_{7/2}&\ls\abs{\Delta_{*}\eta_{j}\a^{\kappa}_{j\alpha}\partial_{\alpha}{\Lambda_{\kappa}^2 v}-\Delta_{*}{\Lambda_{\kappa}^2\eta}_{j}\a^{\kappa}_{j\alpha}\partial_{\alpha} v}_{3/2}
\ls\norm{\Delta_{*}\eta_{j}\a^{\kappa}_{j\alpha}\partial_{\alpha}{\Lambda_{\kappa}^2 v}-\Delta_{*}{\Lambda_{\kappa}^2\eta}_{j}\a^{\kappa}_{j\alpha}\partial_{\alpha} v}_{2}
\\&\ls\norm{ \eta}_4\norm{\a^{\kappa}}_{2}\norm{v}_{3}
\leq P(\norm{\eta}_4,\norm{v}_3).
\end{split}
\end{equation*}
This proves \eqref{fest1} by using further the elliptic theory and the trace theorem.

Finally, the estimate \eqref{tes0} can be obtained similarly as \eqref{tes1} by applying $\partial_t$ to \eqref{etadef} and then using the equation $\partial_t\eta =v+\fk$ and the estimate \eqref{fest1}.
The estimates \eqref{fest22} and \eqref{fest2} could be achieved similarly as \eqref{tes2} and \eqref{tes0} by applying $b_0\cdot\nabla$ and $\partial_t$ to \eqref{etaaa} and using the estimates \eqref{tes1}--\eqref{fest1}. This concludes the lemma.
\end{proof}

\subsubsection{Transport estimates of  $\eta$}

The transport estimate of $\eta$ is recorded as follows.
\begin{proposition}
For $t\in [0,T]$ with $T\le T_\kappa$, it holds that
\begin{equation}\label{etaest}
\norm{\eta(t)}_4^2\leq M_0+TP\left(\sup_{t\in[0,T]} \mathfrak{E}^{\kappa}(t)\right).
\end{equation}
\end{proposition}
\begin{proof}
It follows by using  $\partial_t\eta =v+\fk$ and the estimate \eqref{fest1}.
\end{proof}

\subsubsection{Pressure estimate}\label{pressure1}

In the estimates in the later sections, one needs to estimate the pressure $q$. For this, applying $J^{\kappa}\divak$ to the second equation in \eqref{approximate}, by the third equation and the Piloa indentity \eqref{polia}, one gets
\begin{equation}\label{ell}
-\Div(E\nabla q )=G:=G^1+b_0\cdot\nabla G^2,
\end{equation}
where the matrix $E=J^{\kappa}(\a^{\kappa})^T\a^{\kappa}$ and the function $G$ is given by
\begin{align}
&G^1_i=J^{\kappa}\partial_t\a^{\kappa}_{ij}\partial_j v_i+\left[ J^{\kappa}\a^{\kappa}_{ij}\partial_j,b_0\cdot\nabla\right]b_0\cdot\nabla\eta_i,
 \\& G^2_i= J^{\kappa}\a^{\kappa}_{ij}\partial_j(b_0\cdot\nabla\eta_i).
\end{align}
Note that  by \eqref{inin3} the matrix $E$ is symmetric and positive.

We shall now prove the estimate for the pressure $q$.
\begin{proposition}\label{pressure}
The following estimate holds:
\begin{equation}
\label{press}
\norm{q }_4^2+\norm{\partial_tq }_3^2\leq P\left( \mathfrak{E}^{\kappa} \right).
\end{equation}
\end{proposition}
\begin{proof}
Multiplying \eqref{ell} by $q$ and then integrating by parts over $\Omega$, since $q=0$ and $b_0$ satisfies \eqref{bcond}, one obtains
\begin{equation}
\int_\Omega E\nabla q \cdot \nabla q = \int_\Omega G^1 q -\int_\Omega G^2  b_0\cdot\nabla q.
\end{equation}
Then we have
\begin{equation}
\norm{\nabla q}^2_0\ls \norm{b_0}_4\left(\norm{G^1}_{0}+\norm{G^2}_{0}\right)\norm{q}_{1}\leq \hal \norm{q}_1^2+ C\left(\norm{G^1}_{0}^2+\norm{G^2}_{0}^2\right)\norm{b_0}_4^2.
\end{equation}
By the Poincar\'e inequality $\norm{q}_1^2\ls \norm{\nabla q}_0^2$, one can get
\begin{equation} \label{eeeq}
\begin{split}
\norm{  q}^2_1 &\ls \norm{b_0}_4^2(\norm{ G^1}_{0}^2+\norm{ G^2}_{0}^2)
\\&\le P(\norm{b_0}_4^2, \norm{Dv}_0^2, \norm{D\partial_t\eta^{\kappa}}_0^2, \norm{D(b_0\cdot\nabla\eta)}_0^2, \norm{D\eta^{\kappa}}_0^2, \norm{J^{\kappa}\ak}_{L^{\infty}}^2)\leq P\left(\mathfrak{E}^{\kappa} \right).
\end{split}
 \end{equation}

Next, applying $\bar\partial^\ell $ with $\ell=1,2,3$ to the equation \eqref{ell} leads to
\begin{equation}
-\Div(E\nabla \bar\partial^\ell q )= \bar\partial^\ell G^1+b_0\cdot\nabla \bar\partial^\ell G^2+\left[\bar\partial^\ell, b_0\cdot\nabla \right]G^2
+\Div \left[\bar\partial^\ell, E\nabla \right]q.
\end{equation}
Then as for \eqref{eeeq}, we obtain
\begin{equation}\label{dddddddd}
\norm{ \bar\partial^\ell q}_{1}^2
\ls \norm{ \bar\partial^\ell G^1}_{0}^2+\norm{ \bar\partial^\ell G^2}_{0}^2
+ \norm{\left[\bar\partial^{\alpha},b_0\cdot\nabla\right]G_2}_{0}^2+\norm{\left[\bar\partial^{\alpha},E\nabla\right] q}_{0}^2.
\end{equation}
We then estimate the right hand side of \eqref{dddddddd}. By the estimates \eqref{tes1}--\eqref{tes0}, we may estimate
\begin{equation}\label{qqqq0}
\begin{split}
&\norm{\bar\partial^\ell G^1}_{0}^2+\norm{b_0\cdot\nabla \bar\partial^\ell G^2}_{0}^2+\norm{\left[\bar\partial^\ell, b_0\cdot\nabla \right]G^2}_{0}^2\\&\le P\left(\norm{\bar\partial^\ell \partial_t\eta^{\kappa}}_1,\norm{\bar\partial^\ell v}_1, \norm{\bar\partial^\ell \eta^{\kappa}}_1, \norm{\bar\partial^\ell (b_0\cdot\nabla\eta^{\kappa})}_1,\norm{\bar\partial^\ell (b_0\cdot\nabla\eta)}_1,\norm{b_0}_4,\norm{\eta^{\kappa}}_4\right)\\&\leq P\left( \mathfrak{E}^{\kappa} \right).
\end{split}
\end{equation}
For the last commutator term, by the H\"older inequality, we estimate for each $\ell=1,2,3,$
\begin{align}\label{qqqq1}
&\norm{[\bar\partial,E\nabla] q}_{0}\ls\norm{\bar\partial E}_{L^{\infty} }\norm{\nabla q}_{0}\ls\norm{  E}_{3}\norm{\nabla q}_{0},\\
\label{qqqq2}&\norm{[\bar\partial^2,E\nabla] q}_{0}\ls \norm{\bar\partial E}_{L^{\infty} }\norm{\bar\partial\nabla q}_{0}+\norm{\bar\partial^2E}_{L^3 }\norm{\nabla q}_{L^6 }\ls\norm{  E}_{3}\norm{\nabla q}_{1},\\
\label{qqqq3}&\norm{[\bar\partial^3,E\nabla] q}_{0}\ls\norm{\bar\partial E}_{L^{\infty} }\norm{\bar\partial^2\nabla q}_{0}+\norm{\bar\partial^2E}_{L^3 }\norm{\nabla q}_{L^6 }+\norm{\bar\partial^3 E}_{0}\norm{\nabla q}_{L^{\infty} }\ls\norm{  E}_{3}\norm{\nabla q}_{2}.
\end{align}
By \eqref{partialF}, \eqref{inin3} and \eqref{tes1} imply $\norm{E}_3\leq P(\norm{\eta}_4)$, plugging the estimates \eqref{qqqq0}--\eqref{qqqq3} into \eqref{dddddddd}, we obtain
\begin{equation}
\label{q1}
\norm{ \bar\partial^\ell q}_{1}^2  \leq  P\left( \mathfrak{E}^{\kappa} \right)\left(1+\norm{\nabla q}_{\ell-1}^2\right).
\end{equation}
On the other hand, the equation \eqref{ell} gives
\begin{equation}
\label{np}
\partial_{33}q=\dfrac{1}{E_{33}}\left(G-\sum_{i+j\neq6}\partial_i(E_{ij}\partial_jq)-\partial_3E_{3j}\partial_jq\right).
\end{equation}
This implies that we can estimate the normal derivatives of $q$ in terms of those $q$ terms with less normal derivatives. Hence, using the equation \eqref{np} and the estimates \eqref{eeeq} and \eqref{q1}, inductively on $\ell$, we obtain
\begin{equation}
\label{3q}
\norm{  q}^2_4\leq P\left( \mathfrak{E}^{\kappa} \right).
\end{equation}

We now estimate $\partial_t q$. Applying $\partial_t$ to the equation \eqref{ell} leads to
\begin{equation}
-\Div(E\nabla \partial_t q)=\partial_tG_1+b_0\cdot\nabla\partial_tG_2+\Div(\partial_tE\nabla q).
\end{equation}
By arguing similarly as for \eqref{3q}, we can obtain
\begin{equation}\label{3qt}
\norm{ \dt q}^2_3\leq P\left( \mathfrak{E}^{\kappa} \right).
\end{equation}
Here we have used the estimates \eqref{tes1}--\eqref{fest2} and noted that by using the second equation in \eqref{approximate} and the estimates \eqref{3q}:
\begin{equation}
\label{vt}
\norm{\partial_tv}_3^2=\norm{-\nabak q+(b_0\cdot \nabla)^2 \eta}_3^2 \leq P\left( \mathfrak{E}^{\kappa} \right).
\end{equation}
Consequently, the estimates \eqref{3q} and \eqref{3qt} give \eqref{press}.
\end{proof}

\subsubsection{Tangential energy estimates}\label{tan}

We start with the basic $L^2$ energy estimates.
\begin{proposition}\label{basic}
For $t\in [0,T]$ with $T\le T_\kappa$, it holds that
\begin{equation}\label{00estimate}
\norm{v(t)}_0^2+\norm{(b_0\cdot\nabla\eta)(t)}_0^2\leq M_0+TP\left(\sup_{t\in[0,T]}\mathfrak{E}^{\kappa}(t)\right).
\end{equation}
\end{proposition}
\begin{proof}
Taking the $L^2(\Omega)$ inner product of the second equation in \eqref{approximate} with $v$ yields
\begin{equation}\label{hhl1}
 \dfrac{1}{2}\dfrac{d}{dt}\int_{\Omega}\abs{v}^2 +\int_{\Omega} \nak q\cdot v -\int_{\Omega}(b_0\cdot\nabla)^2\eta \cdot v =0.
\end{equation}
By the integration by parts and using the third equation and the boundary condition $q=0$ on $\Gamma$, using the pressure estimates \eqref{press}, we have
\begin{equation}
-\int_{\Omega} \nak q\cdot v=\int_{\Omega}\partial_j\a_{ij} q v_i\le\norm{D \a^\kappa}_{L^{\infty}}\norm{v}_0\norm{q}_0\le P\left(\sup_{t\in[0,T]}\mathfrak{E}^{\kappa}(t)\right).
\end{equation}
Since $b_0$ satisfies \eqref{bcond}, by the integration by parts and using $\dt \eta=v+\fk$, we obtain
\begin{equation}\label{hhl3}
\begin{split}
-\int_{\Omega}(b_0\cdot\nabla)^2\eta\cdot v&=\int_{\Omega}b_0\cdot\nabla\eta_ib_0\cdot\nabla v_i\\&=\int_{\Omega}b_0\cdot\nabla\eta_ib_0\cdot\nabla \partial_t\eta_i\,dx-\int_{\Omega}b_0\cdot\nabla\eta_ib_0\cdot\nabla \psi^\kappa_i
\\&=\dfrac{1}{2}\dfrac{d}{dt}\int_{\Omega}\abs{b_0\cdot\nabla\eta}^2\,dx-\int_{\Omega}b_0\cdot\nabla\eta_ib_0\cdot\nabla \psi^\kappa_i
\end{split}
\end{equation}
Then \eqref{hhl1}--\eqref{hhl3} implies, using the estimates \eqref{fest1},
\begin{equation}
\dfrac{d}{dt}\int_{\Omega}\abs{v}^2+\abs{b_0\cdot\nabla\eta}^2 \leq P\left(\sup_{t\in[0,T]}\mathfrak E^{\kappa}(t)\right).
\end{equation}
Integrating directly in time of the above yields \eqref{00estimate}.
\end{proof}
In order to perform higher order tangential energy estimates, one needs to compute the equations satisfied by $(\bp^4 v, \bp^4 q, \bp^4 \eta)$, which requires to commutate $\bp^4$ with each term of $\pa^\ak_i$. It is thus useful to establish the following general expressions and estimates for commutators.
It turns out that it is a bit more convenient to consider the equivalent differential operator $\bp^2\Delta_\ast$ so that we can employ the structure of $\Delta_\ast \fk$ on $\Gamma$. For $i=1,2,3,$ we have
\begin{equation}
 \bp^2\Delta_\ast (\pa^\ak_if) =  \pa^\ak_i \bar\partial^2\Delta_* f + \bp^2\Delta_* {\a^{\kappa}_{ij}} \pa_j f+\left[\bar\partial^2\Delta_*, {\a^{\kappa}_{ij}} ,\pa_j f\right].
\end{equation}
By the identity \eqref{partialF}, we have that
\begin{equation}
\begin{split}
&\bp^2\Delta_* (\a^\kappa_{ij} \pa_j f)=-\bp\Delta_*(\a^\kappa_{i\ell}\bp\pa_\ell  \eta^{\kappa}_m  \a^\kappa_{mj})\pa_j f
\\&\quad=-\a^\kappa_{i\ell}\pa_\ell \bp^2\Delta_*\eta^{\kappa}_m\a^\kappa_{mj}\pa_j f
-\left[\bp\Delta_*, \a^\kappa_{i\ell}\a^\kappa_{mj} \right]\bp \pa_\ell  \eta^{\kappa}_m \pa_j f
\\&\quad=-\pa^\ak_i( \bp^2\Delta_*\eta^{\kappa}\cdot\nak f)+\bp^2\Delta_*\eta^{\kappa}\cdot\nak( \pa^\ak_i f)
-\left[\bp\Delta_*, \a^\kappa_{i\ell}\a^\kappa_{mj}\right]\bp\pa_\ell \eta^{\kappa}_m\pa_j f.
\end{split}
\end{equation}
It then holds that
\begin{equation}\label{commf}
 \bp^2\Delta_\ast (\pa^\ak_if) =  \pa^\ak_i\left(\bar\partial^2\Delta_* f- \bp^2\Delta_*\eta^{\kappa}\cdot\nak f\right)+ \mathcal{C}_i(f).
\end{equation}
where the commutator $\mathcal{C}_i(f)$ is given by
\begin{equation}
\mathcal{C}_i(f)=\left[\bar\partial^2\Delta_*, {\a^{\kappa}_{ij}} ,\pa_j f\right]-\bp^2\Delta_*\eta^{\kappa}\cdot\nak( \pa^\ak_i f)
+\left[\bp\Delta_*, \a^\kappa_{i\ell}\a^\kappa_{mj}\right]\bp\pa_\ell \eta^{\kappa}_m\pa_j f
\end{equation}
It was first observed by Alinhac \cite{Alinhac} that the highest order term of $\eta$ will be cancelled when one uses the good unknown $\bar\partial^2\Delta_* f- \bp^2\Delta_*\eta^{\kappa}\cdot\nak f$, which allows one to perform high order energy estimates.

The following lemma deals with the estimates of the commutator $\mathcal C_i(f)$.
\begin{lemma}
The following estimate holds:
\begin{equation}\label{comest}
\norm{\mathcal C_i (f)}_0\leq  P(\norm{\eta}_4) \norm{f}_4.
\end{equation}
\end{lemma}
\begin{proof}
First, by the commutator estimates \eqref{co2} and the estimates \eqref{tes1}, we have
\begin{equation}\label{Calpha1}
\norm{[\bp^2\Delta_*, \a^{\kappa}_{ij}, \partial_j f]}_0\ls\norm{\a^{\kappa}}_{3}\norm{D f}_{3}\leq P(\norm{\eta}_4) \norm{f}_4.
\end{equation}
Next, by the estimates \eqref{tes1} again, we get
\begin{equation}
\norm{\bp^2\Delta_*\eta^{\kappa}\cdot\nak( \pa^\ak_i f)}_0\le \norm{\bp^4\eta^\kappa}_0\norm{\nak( \pa^\ak_i f)}_{L^{\infty} } \leq P(\norm{\eta}_4) \norm{f}_4.
\end{equation}
Finally, by the commutator estimates \eqref{co1} and the estimates \eqref{tes1}, we obtain
\begin{equation}\label{Calpha3}
\norm{\left[\bp\Delta_*, \a^\kappa_{i\ell}\a^\kappa_{mj}\right]\bp\pa_\ell \eta^{\kappa}_m\pa_j f}_0\le \norm{\left[\bp\Delta_*, \a^\kappa_{i\ell}\a^\kappa_{mj}\right]\bp\pa_\ell \eta^{\kappa}_m}_0\norm{Df}_{L^{\infty} }
 \leq P(\norm{\eta}_4) \norm{f}_3.
\end{equation}

Consequently, the estimate \eqref{comest} follows by collecting \eqref{Calpha1}--\eqref{Calpha3}.
\end{proof}

We now introduce the good unknowns
\begin{equation}
\label{gun}
\mathcal{V}=\bar\partial^2\Delta_* v- \bp^2\Delta_*\eta^{\kappa}\cdot\nak v,\quad \mathcal{Q}=\bar\partial^2\Delta_* q- \bp^2\Delta_*\eta^{\kappa}\cdot\nak q.
\end{equation}
Applying $\bar\partial^2\Delta_*$ to the second and third equations in \eqref{approximate}, by \eqref{commf}, one gets
\begin{equation}\label{eqValpha}
\begin{split}
& \mathcal{V}_t + \nak \mathcal{Q}-(b_0\cdot\nabla)\left(\bar\partial^2\Delta_*(b_0\cdot\nabla\eta)\right)\\&\quad =F:=  \dt\left(\bp^2\Delta_*\eta^{\kappa}\cdot\nak v\right) - \mathcal{C}_i(q) +\left[\bar\partial^2\Delta_*, b_0\cdot\nabla\right]b_0\cdot\nabla\eta \text{ in }\Omega,
\end{split}
\end{equation}
and
\begin{equation} \label{divValpha}
  \nabla_\ak\cdot \mathcal{V}=- \mathcal{C}_i(v_i) \text{ in }\Omega.
\end{equation}
Note that $q=0$ on $\Gamma$ implies,
\begin{equation}\label{qValpha}
\mathcal{Q}=- \bp^2\Delta_*\Lambda_{\kappa}^2\eta_i\a^\kappa_{i3}\pa_3 q \text{ on }\Gamma.
\end{equation}

We shall now derive the $\bp^4$-energy estimates.
\begin{proposition}\label{tane}
For $t\in [0,T]$ with $T\le T_\kappa$, it holds that
\begin{equation}
\label{teee}
\norm{\bar\partial^4 v(t)}_0^2+\norm{\bar\partial^4(b_0\cdot\nabla\eta)(t)}_0^2+\abs{\bar\partial^4 \Lambda_{\kappa}\eta_i \a^{\kappa}_{i3}(t)}_0^2\leq M_0+TP\left(\sup_{t\in[0,T]}\mathfrak{E}^{\kappa}(t)\right).
\end{equation}
\end{proposition}
\begin{proof}
Taking the $L^2(\Omega)$ inner product of \eqref{eqValpha} with $\mathcal{V}$ yields
\begin{equation}\label{ttt1}
 \dfrac{1}{2}\dfrac{d}{dt}\int_{\Omega}\abs{\mathcal{V}}^2+\int_{\Omega} \nak \mathcal{Q}\cdot \mathcal{V} +\int_{\Omega} \bar\partial^2\Delta_*(b_0\cdot\nabla\eta_i)   b_0\cdot\nabla \mathcal{V}_i =\int_{\Omega} F\cdot \mathcal{V}.
\end{equation}
By the estimates \eqref{tes1}, \eqref{tes0}, \eqref{comest}, \eqref{vt} and  \eqref{press}, by the definition of $\mathcal{V}$, we have
\begin{equation}
\label{ggg0}
\begin{split}
\int_{\Omega} F\cdot \mathcal{V}&\le\left(\norm{\dt(\bp^4\eta^{\kappa}\cdot\nabla_{\a^{\kappa}}v) }_0+\norm{\mathcal C_i (q)}_0+\norm{[\bp^2\Delta_*, b_0\cdot\nabla]b_0\cdot\nabla\eta}_0  \right)\norm{\mathcal{V}}_0
\\&\leq \left( P(\norm{\eta}_4, \norm{v}_4, \norm{\dt v}_3)+P(\norm{\eta}_4)\norm{q}_4+\norm{b_0}_4\norm{b_0\cdot\nabla\eta}_4\right)P(\norm{\eta^\kappa}_4, \norm{v}_4)
\\   &\leq P\left(\sup_{t\in[0,T]} \mathfrak{E}^{\kappa}(t)\right).
\end{split}
\end{equation}
We now estimate the second and third terms on the left hand side of \eqref{ttt1}. By the definition of $\mathcal V$ and  recalling that $v=\partial_t\eta-\fk$, we have that for the third term
\begin{align}
\label{gg1}
&\int_{\Omega}  \bar\partial^2\Delta_*(b_0\cdot\nabla\eta_i)    b_0\cdot\nabla \mathcal{V}_i\nonumber
\\&\quad= \int_{\Omega}  \bar\partial^2\Delta_*(b_0\cdot\nabla\eta_i)   b_0\cdot\nabla\left(\bar\partial^2\Delta_* v_i- \bp^2\Delta_*\eta^{\kappa}\cdot\nak v_i\right)\nonumber
\\&\quad=\int_{\Omega}  \bar\partial^2\Delta_*(b_0\cdot\nabla\eta_i)   \left(\bar\partial^2\Delta_*(b_0\cdot\nabla v_i)-[\bar\partial^2\Delta_*,b_0\cdot\nabla] v_i-b_0\cdot\nabla(\bp^2\Delta_*\eta^{\kappa}\cdot\nak v_i)\right)\nonumber
\\&\quad=\hal \dfrac{d}{dt}\int_{\Omega}\abs{\bar\partial^2\Delta_*(b_0\cdot\nabla\eta)}^2\nonumber
\\&\quad\quad-\int_{\Omega}  \bar\partial^2\Delta_*(b_0\cdot\nabla\eta_i)   \left(\bar\partial^2\Delta_*(b_0\cdot\nabla  \psi^\kappa_i)+[\bar\partial^2\Delta_*,b_0\cdot\nabla] v_i+b_0\cdot\nabla(\bp^2\Delta_*\eta^{\kappa}\cdot\nak v_i)\right)\nonumber
\\ &\quad\ge \hal \dfrac{d}{dt}\int_{\Omega}\abs{\bar\partial^2\Delta_*(b_0\cdot\nabla\eta)}^2  - \norm{b_0\cdot\nabla\eta}_4 P(\norm{b_0\cdot\nabla\fk}_4,\norm{b_0}_4,\norm{v}_4,\norm{\eta^\kappa}_4,\norm{b_0\cdot\nabla\eta^\kappa}_4)\nonumber
\\ &\quad\ge \hal \dfrac{d}{dt}\int_{\Omega}\abs{\bar\partial^2\Delta_*(b_0\cdot\nabla\eta)}^2  - P\left(\sup_{t\in[0,T]} \mathfrak{E}^{\kappa}(t)\right) .
\end{align}
Here we have used the commutator estimates \eqref{co1}, the estimates \eqref{tes1}, \eqref{tes2} and \eqref{fest22}.

Now for the second term, by the integration by parts and using the equation \eqref{divValpha} and the boundary condition \eqref{qValpha}, we obtain
\begin{equation}
\label{gg2}
\begin{split}
\int_{\Omega} \nak \mathcal{Q}\cdot \mathcal{V}&=\int_{\Gamma} \mathcal{Q}\a^{\kappa}_{i\ell}N_\ell\mathcal{V}_i -\int_{\Omega} \mathcal{Q} \nabla_\ak\cdot \mathcal{V}-\int_{\Omega}\pa_\ell (\a^{\kappa}_{i\ell} )\mathcal{Q} \mathcal{V}_i\\&=
- \int_{\Gamma} \pa_3 q \bp^2\Delta_*(\Lambda_\kappa^2\eta_j)\a^\kappa_{j3} \a^{\kappa}_{i\ell}N_\ell\mathcal{V}_i +\underbrace{\int_{\Omega} \mathcal{Q} \mathcal{C}_i(v_i)-\pa_\ell (\a^{\kappa}_{i\ell} )\mathcal{Q} \mathcal{V}_i }_{\mathcal R}.
\end{split}
\end{equation}
By the definition of $\mathcal{Q}$ and $\mathcal{V}$ and the estimates \eqref{press}, \eqref{comest} and \eqref{tes1},  we have
\begin{equation}
\label{gggg3}
\begin{split}
  \mathcal R&\ls\norm{\mathcal Q}_0(\norm{\mathcal C_i(v_i)}_0+\norm{D\a^{\kappa}}_{L^{\infty} }\norm{\mathcal V}_0)
 \\&\le P(\norm{\eta^\kappa}_4, \norm{q}_4)\left(P(\norm{\eta}_4)\norm{v}_4
 +\norm{D\a^{\kappa}}_{L^{\infty} }P(\norm{\eta^\kappa}_4, \norm{v}_4)\right)
 \\&\leq  P\left(\sup_{t\in[0,T]} \mathfrak{E}^{\kappa}(t)\right).
 \end{split}
\end{equation}
By the definition of $\mathcal{V}$, since $\eta^{\kappa}=\Lambda_\kappa^2\eta$ on $\Gamma$ and $v=\partial_t\eta-\fk$, we have
\begin{equation}
\begin{split}
&- \int_{\Gamma} \pa_3 q\bp^2\Delta_* \Lambda_\kappa^2\eta_j \a^\kappa_{j3} \a^{\kappa}_{i\ell}N_\ell\mathcal{V}_i
\\&\quad = \int_{\Gamma} (-\nabla q\cdot N )\a^\kappa_{j3}\bp^2\Delta_* \Lambda_\kappa^2\eta_j  \a^{\kappa}_{i3} \mathcal{V}_i
 \\&\quad=\int_{\Gamma}(-\nabla q\cdot N )\a^\kappa_{j3}\bp^2\Delta_* \Lambda_\kappa^2\eta_j  \a^{\kappa}_{i3} (\bar\partial^2\Delta_* v_i- \bp^2\Delta_* \Lambda_\kappa^2\eta \cdot\nak v_i)
  \\&\quad=\int_{\Gamma}(-\nabla q\cdot N )\a^\kappa_{j3}\bp^2\Delta_* \Lambda_\kappa^2\eta_j  \a^{\kappa}_{i3} \left(\bar\partial^2\Delta_* \partial_t\eta_i - \bar\partial^2\Delta_* \psi^\kappa _i - \bp^2\Delta_* \Lambda_\kappa^2\eta \cdot\nak v_i\right)  .
\end{split}
\end{equation}
Note that
\begin{equation}
\begin{split}
&\int_{\Gamma}(-\nabla q\cdot N )\a^\kappa_{j3}\bp^2\Delta_* \Lambda_\kappa^2\eta_j  \a^{\kappa}_{i3} \bar\partial^2\Delta_* \partial_t\eta_i  \\
  &\quad
=\int_{\Gamma}(-\nabla q\cdot N )\a^\kappa_{j3}\bp^2\Delta_* \Lambda_\kappa \eta_j  \a^{\kappa}_{i3} \bar\partial^2\Delta_* \Lambda_\kappa \partial_t\eta_i
 + \int_{\Gamma}\bp^2\Delta_* \Lambda_\kappa \eta_j \left[\Lambda_\kappa,  -\nabla q\cdot N \a^\kappa_{j3}\a^{\kappa}_{i3}\right] \bar\partial^2\Delta_* \partial_t\eta_i \\
&\quad
=\hal \frac{d}{dt}\int_{\Gamma}(-\nabla q\cdot N )\abs{\a^\kappa_{i3}\bp^2\Delta_* \Lambda_\kappa \eta_i  }^2  +\hal \int_{\Gamma}\dt( \nabla q\cdot N ) \abs{\a^\kappa_{i3}\bp^2\Delta_* \Lambda_\kappa \eta_i  }^2
\\&\qquad +\int_{\Gamma} \nabla q\cdot N  \a^\kappa_{j3}\bp^2\Delta_* \Lambda_\kappa \eta_j  \dt\a^{\kappa}_{i3} \bar\partial^2\Delta_*  \Lambda_\kappa \eta_i
  +\int_{\Gamma}\bp^2\Delta_* \Lambda_\kappa \eta_j \left[\Lambda_\kappa, -\nabla q\cdot N \a^\kappa_{j3}\a^{\kappa}_{i3}\right] \bar\partial^2\Delta_* \partial_t\eta_i .
\end{split}
\end{equation}
Therefore, we obtain
\begin{equation}
\label{gg3}
\begin{split}
&- \int_{\Gamma} \pa_3 q\bp^2\Delta_* \Lambda_\kappa^2\eta_j \a^\kappa_{j3} \a^{\kappa}_{i\ell}N_\ell\mathcal{V}_i
 \\
&\quad
=\hal \frac{d}{dt}\int_{\Gamma}(-\nabla q\cdot N )\abs{\a^\kappa_{i3}\bp^2\Delta_* \Lambda_\kappa \eta_i  }^2
+\underbrace{\hal \int_{\Gamma}\dt(\nabla q\cdot N )\abs{\a^\kappa_{i3}\bp^2\Delta_* \Lambda_\kappa \eta_i  }^2   }_{\mathcal{I}_1}
\\&\qquad +\underbrace{\int_{\Gamma} \nabla q\cdot N  \a^\kappa_{j3}\bp^2\Delta_* \Lambda_\kappa \eta_j  \dt\a^{\kappa}_{i3} \bar\partial^2\Delta_*  \Lambda_\kappa \eta_i   }_{\mathcal{I}_2}
+\underbrace{\int_{\Gamma}\bp^2\Delta_* \Lambda_\kappa \eta_j \left[\Lambda_\kappa,  -\nabla q\cdot N \a^\kappa_{j3}\a^{\kappa}_{i3}\right] \bar\partial^2\Delta_* \partial_t\eta_i }_{\mathcal{I}_3}
\\&\qquad +\underbrace{\int_{\Gamma} \nabla q\cdot N \a^\kappa_{j3}\bp^2\Delta_* \Lambda_\kappa^2\eta_j  \a^{\kappa}_{i3}  \bp^2\Delta_* \Lambda_\kappa^2\eta \cdot\nak v_i   }_{\mathcal{I}_4}
  +\underbrace{\int_{\Gamma} \nabla q\cdot N \a^\kappa_{j3}\bp^2\Delta_* \Lambda_\kappa^2\eta_j  \a^{\kappa}_{i3}  \bar\partial^2\Delta_* \psi^\kappa_i  }_{\mathcal{I}_5}.
\end{split}
\end{equation}

We now estimate $\mathcal{I}_1$--$\i_5$. By the estimates \eqref{press}, we deduce
\begin{equation}
\label{i0}
\mathcal{I}_1 \ls \abs{\partial_3\partial_t q}_{L^{\infty}}\abs{\a^\kappa_{i3}\bp^2\Delta_* \Lambda_\kappa \eta_i  }^2\ls \norm{\partial_tq}_3\abs{\a^\kappa_{i3}\bp^2\Delta_* \Lambda_\kappa \eta_i  }^2\le P\left(\sup_{t\in[0,T]} \mathfrak{E}^{\kappa}(t)\right).
\end{equation}
By the identity \eqref{partialF}, we have
\begin{align}
\nonumber \i_2&= \int_{\Gamma} (-\nabla q\cdot N) \a^\kappa_{j3}\bp^2\Delta_* \Lambda_\kappa \eta_j   \a^{\kappa}_{i\ell}\pa_\ell\dt\eta^{\kappa}_m\a^\kappa_{m3} \bar\partial^2\Delta_*  \Lambda_\kappa \eta_i
\\&= \underbrace{\int_{\Gamma}  (-\nabla q\cdot N) \a^\kappa_{j3}\bp^2\Delta_* \Lambda_\kappa \eta_j   \a^{\kappa}_{i3}\pa_3\dt\eta^{\kappa}_m\a^\kappa_{m3} \bar\partial^2\Delta_*  \Lambda_\kappa \eta_i   }_{\i_{2a}}\nonumber
\\&\quad+\int_{\Gamma}  (-\nabla q\cdot N)  \a^\kappa_{j3}\bp^2\Delta_* \Lambda_\kappa \eta_j   \a^{\kappa}_{i\alpha}\pa_\alpha\dt\Lambda_{\kappa}^2\eta_m\a^\kappa_{m3} \bar\partial^2\Delta_*  \Lambda_\kappa \eta_i.\label{tmp1}
\end{align}
As usual, we obtain
\begin{equation}
\label{1a}
\i_{2a} \ls\abs{\a^\kappa_{j3}\bp^2\Delta_* \Lambda_\kappa \eta_j  }_0^2\abs{\partial_3 q\pa_3\dt\eta^{\kappa}_m\a^\kappa_{m3}}_{L^{\infty}}\le P\left(\sup_{t\in[0,T]} \mathfrak{E}^{\kappa}(t)\right).
\end{equation}
On the other hand, using $\dt\eta=v+\fk$ we have
\begin{align}
&\int_{\Gamma} (-\nabla q\cdot N) \a^\kappa_{j3}\bp^2\Delta_* \Lambda_\kappa \eta_j   \nonumber \a^{\kappa}_{i\alpha}\pa_{\alpha}\dt\Lambda_{\kappa}^2\eta_m\a^\kappa_{m3} \bar\partial^2\Delta_*  \Lambda_\kappa \eta_i
\\&\quad=\underbrace{\int_{\Gamma}  (-\nabla q\cdot N) \a^\kappa_{j3}\bp^2\Delta_* \Lambda_\kappa \eta_j   \a^{\kappa}_{i\alpha}\pa_\alpha \Lambda_{\kappa}^2 v_m \a^\kappa_{m3} \bar\partial^2\Delta_*  \Lambda_\kappa \eta_i   }_{\i_{2b}}\nonumber
\\&\qquad+\underbrace{\int_{\Gamma}  (-\nabla q\cdot N) \a^\kappa_{j3}\bp^2\Delta_* \Lambda_\kappa \eta_j   \a^{\kappa}_{i\alpha}\pa_\alpha \Lambda_{\kappa}^2 \psi^\kappa_m\a^\kappa_{m3} \bar\partial^2\Delta_*  \Lambda_\kappa \eta_i   }_{\i_{2c}}
\label{ttt}
\end{align}
To estimate $\i_{2c}$, the difficulty is that one can not have an $\kappa$-independent control of $\abs{\bar\partial^2\Delta_*  \Lambda_\kappa \eta_i }_0$. Our observation is that since $\fk\rightarrow 0$ as $\kappa\rightarrow 0$, this motives us to deduce the following estimates:
 \begin{equation}\label{lwuqing}
 \abs{\bar\partial   \fk}_{L^{\infty} }\leq \sqrt\kappa P(\norm{\eta}_4,\norm{v}_3).
 \end{equation}
Indeed, we can rewrite the boundary condition in \eqref{etaaa} as
\begin{align}
 \fk=\Delta_{*}^{-1}\mathbb{P} f^\kappa,\ f^\kappa:= \Delta_{*} (\eta_j-{\Lambda_{\kappa}^2\eta_j})\a^{\kappa}_{j\alpha}\partial_{\alpha}{\Lambda_{\kappa}^2 v}-\Delta_{*}{\Lambda_{\kappa}^2\eta}_j\a^{\kappa}_{j\alpha}\partial_{\alpha}(v-{\Lambda_{\kappa}^2 v}) .
\end{align}
By using Morrey's inequality and the Sobolev embeddings and the trace theorem,
\begin{equation}
\abs{g-\Lambda_{\kappa}g}_{L^\infty}\ls \sqrt{\kappa}\abs{\bp g}_{L^4} \ls \sqrt{\kappa}\abs{g}_{3/2}\ls \sqrt{\kappa}\norm{g}_{2},
\end{equation}
we obtain
\begin{equation}
\begin{split}
\abs{f^\kappa}_{L^\infty}&\ls\abs{\a^{\kappa}_{j\alpha}\partial_{\alpha}\Lambda_{\kappa}^2 v}_{L^\infty}\abs{\Delta_*\eta-\Lambda_{\kappa}^2\Delta_*\eta}_{L^\infty}+ \abs{\Delta_{*}{\Lambda_{\kappa}^2\eta}_j\a^{\kappa}_{j\alpha}}_{L^\infty}\abs{\partial_{\alpha}v-\Lambda_{\kappa}^2 \partial_{\alpha}v}_{L^\infty} \\&\ls\sqrt\kappa P(\norm{\eta}_4,\norm{v}_3).
\end{split}
\end{equation}
Then by the elliptic estimate and the Sobolev embeddings, we deduce
 \begin{equation}
\abs{\bar\partial \fk}_{L^{\infty} }\ls\abs{\bar\partial \fk}_{W^{1,4}}\ls\abs{f^\kappa}_{L^4}\ls\abs{f^\kappa}_{L^\infty} \ls \sqrt\kappa P(\norm{\eta}_4,\norm{v}_3),
 \end{equation}
 which proves \eqref{lwuqing}. Hence, by \eqref{lwuqing} together with \eqref{loss}, we have
\begin{equation}\label{22c}
\begin{split}
\i_{2c}&\ls \abs{\partial_3q\a^\kappa_{m3}\a^{\kappa}_{i\alpha}}_{L^{\infty} }\abs{\a^\kappa_{j3}\bp^2\Delta_* \Lambda_\kappa \eta_j }_0\abs{\bar\partial^2\Delta_* \Lambda_\kappa \eta_i}_0\abs{\pa_\alpha \Lambda_{\kappa}^2 \psi^\kappa_m}_{L^{\infty} }
\\ &\ls\abs{\partial_3q\a^\kappa_{m3}\a^{\kappa}_{i\alpha}}_{L^{\infty} }\abs{\a^\kappa_{j3}\bp^2\Delta_* \Lambda_\kappa \eta_j }_0\dfrac{1}{\sqrt{\kappa}}\abs{\eta}_{7/2}\sqrt{\kappa}P(\norm{\eta}_4,\norm{v}_3)
\\&\leq P\left(\sup_{t\in[0,T]} \mathfrak{E}^{\kappa}(t)\right).
\end{split}
\end{equation}
Note that the term $\i_{2b}$ is out of control by an $\kappa$-independent bound alone.

For $\i_3$, by the commutator estimates \eqref{es1-1/2}, \eqref{co123}, \eqref{press}, \eqref{tes1} and \eqref{fest1}, we obtain
\begin{align}
\nonumber
\i_3&\leq \abs{\bp^2\Delta_* \Lambda_\kappa \eta_j }_{-1/2}\abs{\left[\Lambda_\kappa, (-\nabla q\cdot N )\a^\kappa_{j3}\a^{\kappa}_{i3}\right] \bar\partial(\bp\Delta_* \partial_t\eta_i)}_{1/2}
\\\nonumber &\ls \abs{\bp\Delta_* \Lambda_\kappa \eta_j }_{1/2}\abs{ \partial_3 q\a^{\kappa}_{j3}\a^{\kappa}_{i3}}_{W^{1,\infty} }\abs{\bp\Delta_*\partial_t\eta}_{1/2}\nonumber
\\&\ls \norm{\eta}_{4}\norm{\partial_3 q\a^{\kappa}_{j3}\a^{\kappa}_{i3}}_3\norm{v+\psi^{\kappa}}_{4} \leq P\left(\sup_{t\in[0,T]} \mathfrak{E}^{\kappa}(t)\right).
\label{commm1}
\end{align}

To control $\i_4$, similarly as \eqref{tmp1}, we write
\begin{align}
\nonumber
\i_4=&\underbrace{\int_{\Gamma}(\nabla q\cdot N )\a^\kappa_{j3}\bp^2\Delta_* \Lambda_\kappa^2\eta_j  \a^{\kappa}_{m3}  \bp^2\Delta_*\Lambda_\kappa^2\eta_{i}\a^{\kappa}_{i3}\partial_3 v_m   }_{\i_{4a}}\\&+\underbrace{\int_{\Gamma}(\nabla q\cdot N )\a^\kappa_{j3}\bp^2\Delta_* \Lambda_\kappa^2\eta_j \a^{\kappa}_{i\alpha} \partial_{\alpha} v_m\a^{\kappa}_{m3} \bp^2\Delta_*\Lambda_\kappa^2\eta_i  }_{\i_{4b}}.\label{ttttt}
\end{align}
By the commutator estimates \eqref{es1-0}, we have
\begin{equation}
\label{ennennene}
\begin{split}
\abs{\a^\kappa_{j3}\bp^2\Delta_* \Lambda_\kappa^2\eta_j }_0&\ls \abs{[\Lambda_{\kappa}, \a^\kappa_{j3}]\bp (\bp\Delta_*\Lambda_{\kappa}\eta_j)}_0+\abs{\a^\kappa_{j3}\bp^2\Delta_*(\Lambda_\kappa\eta_j)}_0
\\&\ls\abs{\a^{\kappa}}_{W^{1,\infty} }\abs{\bp\Delta_*\Lambda_{\kappa}\eta_j}_0+\abs{\a^\kappa_{j3}\bp^2\Delta_*(\Lambda_\kappa\eta_j)}_0.
\end{split}
\end{equation}
Then we obtain
\begin{equation}\label{3b}
\i_{4a}\ls
\left(\abs{\a^\kappa_{j3}\bp^2\Delta_* \Lambda_\kappa^2\eta_j }_0^2+\abs{\ak}_{W^{1,\infty}}^2\norm{\eta}_4^2\right)\abs{\partial_3q\a^{\kappa}_{m3}\partial_3 v_m}_{L^{\infty} }
 \leq P\left(\sup_{t\in[0,T]} \mathfrak{E}^{\kappa}(t)\right).
\end{equation}
Note that the term $\i_{4b}$ is also out of control by an $\kappa$-independent bound alone.

Now we take care of $\i_{2b}$ and $\i_{4b}$. Notice that $\mathcal{I}_{2b}$ and $\mathcal{I}_{4b}$ are cancelled out in the limit $\kappa\rightarrow0$, however, it is certainly not the case when $\kappa>0$. This is most involved thing in the tangential energy estimates. Note also that we can not use the commutator estimate to interchange the position of the mollifier operator $\Lambda_\kappa$ in each of two terms since  $\abs{\bar\partial^2\Delta_*\eta}_{L^\infty}$ is out of control.
The key point here is to use the term $\i_5$, by the definition of the modification term $\fk$, to kill out both $\i_{2b}$ and $\i_{4b}$; this is exactly the reason that we have introduced $\fk$. By the boundary condition in \eqref{etaaa}, we deduce
\begin{equation}
\begin{split}
\i_5&=\int_{\Gamma}\nabla q\cdot N \a^\kappa_{j3}\bp^2\Delta_* \Lambda_\kappa^2\eta_j  \a^{\kappa}_{i3}  \bar\partial^2 \left( \Delta_{*}\eta_m \a^\kappa_{m\alpha}\pa_\alpha{\Lambda_{\kappa}^2 v_i}-\Delta_{*}\Lambda_{\kappa}^2\eta_m \a^\kappa_{m\alpha}\pa_\alpha{ v_i}\right)
\\&=\int_{\Gamma} \nabla q\cdot N  \a^\kappa_{j3}\bp^2\Delta_* \Lambda_\kappa^2\eta_j  \a^{\kappa}_{i3}    \bar\partial^2 \Delta_{*}\eta_m \a^\kappa_{m\alpha}\pa_\alpha{\Lambda_{\kappa}^2 v_i}
\\&\quad+\underbrace{\int_{\Gamma} (-\nabla q\cdot N ) \a^\kappa_{j3}\bp^2\Delta_* \Lambda_\kappa^2\eta_j  \a^{\kappa}_{i3}
\bar\partial^2 \Delta_{*}\Lambda_{\kappa}^2\eta_m \a^\kappa_{m\alpha}\pa_\alpha{ v_i} }_{-\i_{4b}}
\\&\quad+\underbrace{\int_{\Gamma} \nabla q\cdot N \a^\kappa_{j3}\bp^2\Delta_* \Lambda_\kappa^2\eta_j  \a^{\kappa}_{i3}  \left( \left[\bar\partial^2, \a^\kappa_{m\alpha}\pa_\alpha{\Lambda_{\kappa}^2 v_i}\right]\Delta_{*}\eta_m -\left[\bar\partial^2,\a^\kappa_{m\alpha}\pa_\alpha{ v_i}\right] \Delta_{*}\Lambda_{\kappa}^2\eta_m \right)  }_{\i_{5a}}
\label{tmp2}
\end{split}
\end{equation}
By doing estimates as usual and using \eqref{ennennene} again, we have
\begin{equation}
\label{2a}
\begin{split}
\i_{5a} &\ls |\partial_3q\a^{\kappa}_{i3}|_{L^{\infty}(\Omega)}\abs{\a^\kappa_{j3}\bp^2\Delta_* \Lambda_\kappa^2\eta_j }_0\abs{\left(\left[\bar\partial^2, \a^\kappa_{m\alpha}\pa_\alpha{\Lambda_{\kappa}^2 v_i}\right]\Delta_{*}\eta_m -\left[\bar\partial^2,\a^\kappa_{m\alpha}\pa_\alpha{ v_i}\right] \Delta_{*}\Lambda_{\kappa}^2\eta_m \right)}_0
\\
&\le P\left(\sup_{t\in[0,T]} \mathfrak{E}^{\kappa}(t)\right).
\end{split}
\end{equation}
We rewrite the first term as
\begin{align}
\nonumber&\int_{\Gamma} \nabla q\cdot N \a^\kappa_{j3}\bp^2\Delta_* \Lambda_\kappa^2\eta_j  \a^{\kappa}_{i3}    \bar\partial^2 \Delta_{*}\eta_m \a^\kappa_{m\alpha}\pa_\alpha{\Lambda_{\kappa}^2 v_i}
\\\nonumber&\quad=\underbrace{\int_{\Gamma} \nabla q\cdot N  \a^\kappa_{j3}\bp^2\Delta_* \Lambda_\kappa \eta_j  \a^{\kappa}_{i3}    \bar\partial^2 \Delta_{*} \Lambda_\kappa \eta_m \a^\kappa_{m\alpha}\pa_\alpha{\Lambda_{\kappa}^2 v_i}}_{-\i_{2b}}
\\ &\qquad+\underbrace{\int_{\Gamma}\bp^2\Delta_* \Lambda_\kappa \eta_j \left[\Lambda_\kappa,  \nabla q\cdot N \a^\kappa_{j3}\a^{\kappa}_{i3} \a^\kappa_{m\alpha}\pa_\alpha{\Lambda_{\kappa}^2 v_i}\right]\bar\partial^2 \Delta_{*}\eta_m   }_{\i_{5b}}.\label{temp3}
\end{align}
By arguing similarly as \eqref{commm1} for $\i_3$, we have
\begin{equation}
\label{pw}
\i_{5b}\le \norm{\eta}_{4}\norm{\partial_3 q\a^\kappa_{j3}\a^{\kappa}_{i3} \a^\kappa_{m\alpha}\pa_\alpha{\Lambda_{\kappa}^2 v_i}}_3\norm{\eta}_{4}\leq P\left(\sup_{t\in[0,T]} \mathfrak{E}^{\kappa}(t)\right).
\end{equation}

Now combining \eqref{tmp1}, \eqref{ttt}, \eqref{ttttt}, \eqref{tmp2} and \eqref{temp3},  and using the estimates \eqref{1a}, \eqref{22c}, \eqref{3b}, \eqref{2a} and \eqref{pw}, we deduce
\begin{equation}\label{ggiip}
\i_2+\i_4+\i_5=\i_{2a}+\i_{2c}+\i_{4a}+\i_{5a}+\i_{5b}\le P\left(\sup_{t\in[0,T]} \mathfrak{E}^{\kappa}(t)\right).
\end{equation}

Finally, combining \eqref{gg1}, \eqref{gg2} and \eqref{gg3}, and using the estimates \eqref{ggg0}, \eqref{gggg3}, \eqref{commm1} and \eqref{ggiip}, we obtain
\begin{equation}\label{ohoh}
 \dfrac{d}{dt}\left(\int_{\Omega}\abs{\mathcal{V}}^2+\abs{\bar\partial^2\Delta_*( b_0\cdot\nabla\eta)}^2 + \int_{\Gamma}(-\nabla q\cdot N )\abs{\bp^2\Delta_* \Lambda_\kappa \eta_i \a^\kappa_{i3} }^2   \right)\leq  P\left(\sup_{t\in[0,T]} \mathfrak{E}^{\kappa}(t)\right).
\end{equation}
Integrating \eqref{ohoh} directly in time, by the a priori assumption \eqref{ini2}, we have
\begin{equation}\label{ohoh1}
\norm{\mathcal V(t)}_0^2+\norm{\bar\partial^4 (b_0\cdot\nabla \eta)(t)}_0^2+\abs{\bar\partial^2\Delta_*\Lambda_{\kappa}\eta_i \a^{\kappa}_{i3}(t)}_0^2\leq M_0+TP\left(\sup_{t\in[0,T]} \mathfrak{E}^{\kappa}(t)\right).
\end{equation}
By the definition of $\mathcal{V}$, using \eqref{ohoh1} and  \eqref{etaest}, using the fundamental theorem of calculous, we get
\begin{equation}
\begin{split}
\norm{\bar\partial^4 v(t)}_0^2&\ls \norm{\bar\partial^2\Delta_* v(t)}_0^2\ls\norm{\mathcal V(t)}_0^2+\norm{\bar\partial^2\Delta_*\eta}_0^2\norm{\a^{\kappa}_{j\ell}\partial_{\ell} v}_{L^{\infty}}^2
\\&\leq M_0+TP\left(\sup_{t\in[0,T]} \mathfrak{E}^{\kappa}(t)\right).
\end{split}
\end{equation}
We thus conclude the proposition.
\end{proof}
\subsubsection{Curl and divergence estimates}\label{curle}
In view of the Hodge-type elliptic estimates, we can control one vector's all derivatives just by its curl, divergence and normal trace. Thus, we now derive the curl and divergence estimates.
We begin with the $\curl$ estimates.
\begin{proposition}
For $t\in [0,T]$ with $T\le T_\kappa$, it holds that
\begin{equation}
\label{curlest}
\norm{\curl v(t)}_{3}^2+\norm{\curl(b_0\cdot\nabla\eta)(t)}_{3}^2\leq M_0+TP\left(\sup_{t\in[0,T]} \mathfrak{E}^{\kappa}(t)\right).
\end{equation}
\end{proposition}
\begin{proof}
By taking the $\curl_{\a^{\kappa}}$ of the second equation in \eqref{approximate}, we have that
\begin{equation}
\curl_{\a^{\kappa}} \partial_tv=\curl_{\a^{\kappa}}\left((b_0 \cdot\nabla)^2\eta\right),
\end{equation}
where $(\curl_{\ak}g)_i=\epsilon_{ij\ell}\a^{\kappa}_{jm}\partial_m g_{\ell}$.
It follows that
\begin{equation}\label{curl}
\partial_t(\curl_{\a^{\kappa}}v)_i-b_0\cdot\nabla\left(\curl_{\a^{\kappa}}\left(b_0\cdot\nabla \eta\right)\right)_i=\left(\left[\curl_{\a^{\kappa}}, b_0\cdot\nabla\right](b_0\cdot\nabla \eta)\right)_i+\epsilon_{ij\ell} \partial_t\a^{\kappa}_{jm} \partial_mv_{\ell}.
\end{equation}
Apply further $D^3$ to \eqref{curl} to get
\begin{equation}\label{curl2}
\partial_t(D^3\curl_{\a^{\kappa}}v)_i-b_0\cdot\nabla\left(D^3\curl_{\a^{\kappa}}\left(b_0\cdot\nabla \eta\right)\right)_i=F_i,
\end{equation}
with
\begin{equation}
F_i:=[D^3,b_0\cdot\nabla](\curl_{\a^{\kappa}}(b_0\cdot\nabla \eta))+D^3\left(\left(\left[\curl_{\a^{\kappa}}, b_0\cdot\nabla\right](b_0\cdot\nabla \eta)\right)_i+\epsilon_{ij\ell} \partial_t\a^{\kappa}_{jm} \partial_mv_{\ell}\right).
\end{equation}

Taking the $L^2$ inner product of \eqref{curl2} with $D^3\curl_{\a^{\kappa}}v$, by the integration by parts, we get
\begin{align}\label{j00}
&\dfrac{1}{2}\dfrac{d}{dt}\int_{\Omega}\abs{D^3\curl_{\a^{\kappa}}v}^2+\underbrace{\int_{\Omega}D^3\curl_{\a^{\kappa}}(b_0\cdot\nabla \eta)\cdot D^3\curl_{\a^{\kappa}}(b_0\cdot\nabla v)}_{\mathcal{J}_1}\\&\quad=\underbrace{\int_\Omega F \cdot D^3\curl_{\a^{\kappa}}v}_{\j_2}+\underbrace{\int_{\Omega}D^{3}\curl_{\a^{\kappa}}(b_0\cdot\nabla \eta)\cdot \left[D^3\curl_{\a^{\kappa}}, b_0\cdot\nabla\right] v}_{\j_3}.\nonumber
\end{align}
Since $v=\dt\eta-\fk$, we have
\begin{equation}
\label{j0}
\begin{split}
\j_1=&\dfrac{1}{2}\dfrac{d}{dt}\int_{\Omega}\abs{D^3\curl_{\a^{\kappa}}(b_0\cdot\nabla \eta)}^2
\underbrace{-\int_{\Omega}D^3\curl_{\a^{\kappa}}(b_0\cdot\nabla \eta)\cdot D^3\curl_{\a^{\kappa}}(b_0\cdot\nabla \fk)}_{\j_{1a}}\\&-\underbrace{\dfrac{1}{2}\int_{\Omega}D^3(\curl_{\a^{\kappa}}(b_0\cdot\nabla \eta))_i\cdot D^3(\epsilon_{ij\ell} \partial_t\a^{\kappa}_{jm} \partial_m(b_0\cdot\nabla\eta_\ell))}_{\j_{1b}}.
\end{split}
\end{equation}
By the estimates \eqref{fest22} and \eqref{tes1}, we obtain
\begin{equation}
\label{j0b}
\j_{1a}\ls \norm{b_0\cdot\nabla\eta}_4\norm{\a^{\kappa}}_3^2\norm{b_0\cdot\nabla\fk}_4\leq P\left(\sup_{t\in[0,T]} \mathfrak{E}^{\kappa}(t)\right).
\end{equation}
By the identity \eqref{partialF} and the estimates \eqref{tes0} and \eqref{fest1}, we have
\begin{equation}
\j_{1b}
\ls  \norm{\a^{\kappa}}_3\norm{D(b_0\cdot\nabla\eta)}_3\norm{\partial_t\a^{\kappa}}_3\norm{D(b_0\cdot\nabla\eta)}_3
  \le P\left(\sup_{t\in[0,T]} \mathfrak{E}^{\kappa}(t)\right).
\end{equation}
Hence, we obtain
\begin{equation}
\label{j1}
\j_1\ge \dfrac{1}{2}\dfrac{d}{dt}\int_{\Omega}\abs{D^3\curl_{\a^{\kappa}}(b_0\cdot\nabla \eta)}^2
-P\left(\sup_{t\in[0,T]} \mathfrak{E}^{\kappa}(t)\right).
\end{equation}

We now turn to estimate the right hand side of \eqref{j00}. By the estimates \eqref{tes1}--\eqref{tes0} and the identity \eqref{partialF}, we may have
\begin{equation}
\label{j2}
\begin{split}
\j_2&\leq  \norm{F}_0 \norm{\curl_{\ak} v}_3
\leq P\left(\norm{b_0}_4,  \norm{b_0\cdot\nabla\eta}_4, \norm{Dv}_3, \norm{b_0\cdot\nabla\a^{\kappa}}_3, \norm{\a^{\kappa}}_3\right)
\\&\leq P\left(\sup_{t\in[0,T]} \mathfrak{E}^{\kappa}(t)\right).
\end{split}
\end{equation}
Similarly,
\begin{align}\label{j3}
\nonumber\j_3&\ls \norm{D^3\curl_{\a^{\kappa}}(b_0\cdot\nabla \eta)}_0\norm{[D^3\curl_{\a^{\kappa}}, b_0\cdot\nabla] v}_0
\\&\leq P\left(\norm{b_0}_4,  \norm{b_0\cdot\nabla\eta}_4, \norm{Dv}_3, \norm{b_0\cdot\nabla\a^{\kappa}}_3, \norm{\a^{\kappa}}_3\right)
 \leq P\left(\sup_{t\in[0,T]} \mathfrak{E}^{\kappa}(t)\right).
\end{align}

Consequently, plugging the estimates \eqref{j1}--\eqref{j3} into \eqref{j00}, we obtain
\begin{equation} \label{kllj}
\dfrac{d}{dt}\int_{\Omega}\abs{D^3\curl_{\a^{\kappa}}v}^2+\abs{D^3\curl_{\a^{\kappa}}(b_0\cdot\nabla \eta)}^2
\le P\left(\sup_{t\in[0,T]} \mathfrak{E}^{\kappa}(t)\right).
\end{equation}
Integrating \eqref{kllj} directly in time, and applying the fundamental theorem of calculous,
\begin{equation}
\norm{\curl  f(t)}_3\le \norm{\curl_{\a^{\kappa}} f(t)}_3+\norm{\int_0^t\partial_t\a^{\kappa} d\tau D f(t)}_3,
\end{equation}
we then conclude the proposition.
\end{proof}

We now derive the divergence estimates.
\begin{proposition}
For $t\in [0,T]$ with $T\le T_\kappa$, it holds that
\begin{equation}
\label{divest}
\norm{\Div v(t)}_3^2+\norm{\Div(b_0\cdot\nabla\eta)(t)}_3^2 \leq TP\left(\sup_{t\in[0,T]} \mathfrak{E}^{\kappa}(t)\right).
\end{equation}
\end{proposition}
\begin{proof}
From $\divak v=0$, we see that
\begin{equation}
\Div v=-\int_0^t\partial_t\a^{\kappa}_{ij}\,d\tau\partial_jv_i.
\end{equation}
Hence, it is clear that by the identity \eqref{partialF} and the estimates \eqref{tes1} and \eqref{tes0},
\begin{equation}\label{divv1}
\norm{\Div v(t)}_{3}^2\le TP\left(\sup_{t\in[0,T]} \mathfrak{E}^{\kappa}(t)\right).
\end{equation}
%Thus, with \eqref{co2}, we have
%\begin{equation}
%\begin{split}
%\|\Div v(t)\|_{3}^2&\leq T\norm{\partial_t\a^{\kappa}}_{L^{\infty}(\Omega)}\norm{v}_4+\norm{\eta}_4\norm{\a^{\kappa}Dv}_{L^{\infty}(\Omega)}
%\\&\quad+\norm{D\a^{\kappa}}_1\norm{D^2v}_1+\norm{D\a^{\kappa}}_{L^{\infty}(\Omega)}\norm{D^2\eta Dv}_0\\
%&\leq M_0+CTP\left(\sup_{t\in[0,T]} \mathfrak{E}^{\kappa}(t)\right).
%\end{split}
%\end{equation}

From $\divak v=0$ again, we have
\begin{equation*}
\Div_{\a^{\kappa}}(b_0\cdot\nabla v)=[\Div_{\a^{\kappa}},b_0\cdot\nabla] v.
\end{equation*}
This together with the equation $v=\partial_t\eta-\fk$, we have
\begin{equation}
\label{divb}
\begin{split}
\partial_t\left(\Div_{\ak}\left(b_0\cdot \nabla \eta\right)\right)=\Div_{\a^{\kappa}}\left(b_0\cdot\nabla \fk\right)+[\Div_{\a^{\kappa}},b_0\cdot\nabla] v+ \partial_t \a^{\kappa}_{i\ell} \partial_{\ell}\left(b_0\cdot\nabla\eta_i\right).
\end{split}
\end{equation}
This implies that, by doing the $D^3$ energy estimate and using the estimates \eqref{tes1}--\eqref{fest22} and the identity \eqref{partialF},
\begin{equation}\label{di2}
\norm{\Div_{\ak}\left(b_0\cdot \nabla \eta\right)}_{3}^2
  \leq  TP\left(\sup_{t\in[0,T]} \mathfrak{E}^{\kappa}(t)\right).
\end{equation}
And then applying the fundamental theorem of calculous, and
\begin{equation}
\norm{\Div  f(t)}_3\le \norm{\Div_{\a^{\kappa}} f(t)}_3+\norm{\int_0^t\partial_t\a^{\kappa} d\tau D f(t)}_3,
\end{equation}
we arrive at
\begin{equation}\label{divv12}
\norm{\Div(b_0\cdot \nabla \eta)}_{3}^2
  \leq  TP\left(\sup_{t\in[0,T]} \mathfrak{E}^{\kappa}(t)\right).
\end{equation}

Consequently, we conclude the proposition by the estimates \eqref{divv1} and \eqref{divv12}.
\end{proof}

\subsubsection{Synthesis}

We now collect the estimates derived previously to conclude our estimates and also verify the a priori assumptions \eqref{ini2} and \eqref{inin3}. That is, we shall now present the
\begin{proof}[Proof of Theorem \ref{th43}]
It follows from the normal trace estimates \eqref{gga} that
\begin{equation}
\abs{\bar\partial^4 v\cdot N}_{-1/2}\ls \norm{\bar\partial^4 v}_0+\norm{\Div\bar\partial^3v}_0.
\end{equation}
Combining this and the estimates \eqref{00estimate}, \eqref{teee}, \eqref{curlest} and \eqref{divest},  by using the Hodge-type elliptic estimates \eqref{hodd} of Lemma \ref{hodge}, we obtain
\begin{align}
\norm{v}_4^2 \ls \norm{v}_0^2+\norm{\Div v}_{3}^2+\norm{\curl v}_{3}^2+\abs{\bar\partial v\cdot N}_{5/2}^2
\leq M_0+TP\left(\sup_{t\in[0,T]} \mathfrak{E}^{\kappa}(t)\right).
\end{align}
Similarly, we have
\begin{align}
\norm{b_0\cdot\nabla\eta}_4^2 \leq M_0+TP\left(\sup_{t\in[0,T]} \mathfrak{E}^{\kappa}(t)\right).
\end{align}
By these two estimates and \eqref{etaest}, \eqref{teee}, we finally get that
\begin{equation*}
\sup_{[0,T]}\mathfrak{E}^{\kappa}(t)\leq M_0+TP\left(\sup_{t\in[0,T]} \mathfrak{E}^{\kappa}(t)\right).
\end{equation*}
This provides us with a time of existence $T_1$ independent of $\kappa$ and an estimate on $[0,T_1]$ independent of $\kappa$ of the type:
\begin{equation}
\sup_{[0,T_1]}\mathfrak{E}^{\kappa}(t)\leq 2M_0.
\end{equation}
Since by \eqref{taylor}, $-\nabla q_0\cdot N\ge \lambda$ on $\Gamma$, $\a^\kappa(0)=I$ and $J^\kappa(0)=1$, the bound \eqref{bound} and \eqref{press} verify in turn the a priori bounds \eqref{ini2} and \eqref{inin3} by the fundamental theorem of calculous with taking $T_1$ smaller if necessary.
The proof of Theorem \ref{th43} is thus completed.
\end{proof}

\section{Construction of solutions to nonlinear $\kappa$-approximate problem}\label{exist}

Throughout this section, we fix $0<\kappa<1$. Our goal of this section is to construct solutions to the nonlinear $\kappa$-approximate problem \eqref{approximate} on a time interval $[0,T_\kappa]$ for some $T_\kappa>0$.
\subsection{The linearized $\kappa$-approximate problem}\label{lso}

Given $\psi$ and $\tilde\a^\kappa =\a(\tilde \eta^\kappa)$ (and $\tilde J^{\kappa}$, etc.) with $\tilde \eta^{\kappa}$  determined by the given $\tilde\eta$ through \eqref{etadef}. We consider the following linearized $\kappa$-approximate problem
\begin{equation}\label{lvapproximate}
	\begin{cases}
	\partial_t\eta =v+\psi &\text{in } \Omega,\\
	\partial_tv  +\nabla_{\tilde\a^\kappa} q =(b_0\cdot\nabla)^2\eta  &\text{in } \Omega,\\
	\Div_{\tilde\a^\kappa} v = 0 &\text{in  }\Omega,\\
	q=0 & \text{on  }\Gamma,\\
	(\eta,v)\mid_{t=0} =(\text{Id}, v_0).
	\end{cases}
	\end{equation}	
We assume $T>0$ and that $\tilde\eta, b_0\cdot\nabla\tilde\eta$, $\partial_t\tilde\eta$, $\psi$, $b_0\cdot\nabla\psi$ $\in L^{\infty}(0,T;H^4(\Omega))$, $\tilde\eta\mid_{t=0}=Id$ and that
\begin{align}
\abs{\tilde J^{\kappa}(t)-1}\leq \dfrac{1}{8} \text{ and } \abs{\tilde\a_{ij}^{\kappa}(t)-\delta_{ij}}\leq \dfrac{1}{8} \text{ in }\Omega.
\end{align}

 It is simple to see \eqref{lvapproximate} as two coupled linear
problems: a transport problem for $\eta$ with $v$ given, and a linear forced Euler problem for $(v,q)$ with $\eta$ given.
This perspective works well for the incompressible Euler equations \cite{CS07,DS_10}, however, in the presence of the magnetic field the regularity demand of the term $(b_0\cdot\nabla)^2\eta$ is one order regularity higher than what can be recovered from the transport equation for $\eta$, given the target regularity
for $v$. Our way of getting around this difficulty is to recover a solution to \eqref{lvapproximate} as a limit as $\varepsilon\rightarrow0$ of the following regularized linear $\varepsilon$-$\kappa$-approximate problem:
\begin{equation}\label{epsapproximate}
	\begin{cases}
	\partial_t\eta-\varepsilon (b_0\cdot\nabla)^2\eta =v+\psi &\text{in } \Omega,\\
	\partial_tv  +\nabla_{\tilde\a^\kappa} q =(b_0\cdot\nabla)^2\eta  &\text{in } \Omega,\\
	\Div_{\tilde\a^\kappa} v = 0 &\text{in  }\Omega,\\
	q=0 & \text{on  }\Gamma,\\
	(\eta,v)\mid_{t=0} =(\text{Id}, v_0).
	\end{cases}
	\end{equation}
Here $0<\ep<1$ is the artificial viscosity coefficient.	By adding this artificial viscosity term, the regularity of $(b_0\cdot\nabla)^2\eta$ obtaining from the first
equation of \eqref{epsapproximate} is just sufficient for solving the linear forced Euler problem for $(v,q)$. It is important that since $b_0\cdot N=0$ on $\Gamma$, we do not need to impose boundary conditions for $\eta$  and thus there is no boundary layer appearing  as $\varepsilon\rightarrow0$.

We first construct solutions to the linear $\varepsilon$-$\kappa$-approximate problem \eqref{epsapproximate} by using a fixed point argument which is based on the solvability of three linear problems, i.e., \eqref{eqforeta}, \eqref{eqforv} and \eqref{eqforq}. Then, we derive an $\varepsilon$-independent estimates of the solutions, which allows us to pass to the limit as $\varepsilon \rightarrow 0$ to produce the solution to  the linear $\kappa$-approximate problem \eqref{lvapproximate}. Finally, the existence of a unique solution to \eqref{lvapproximate} is recorded in Theorem \ref{linearthm}.

We write the dependence of constants and polynomials on $\kappa$ and $\varepsilon$ as $C_\kappa, C_\varepsilon, P_\kappa, P_\varepsilon$, etc. %Our estimates will also depend on $\tilde\psi$ and $\tilde \eta^{\kappa}$, and we may write these dependence of polynomials as $\tilde P: =P(\norm{\tilde\eta^{\kappa}}_4^2, \norm{\partial_t\tilde\eta^{\kappa}}_4^2, \norm{b_0\cdot\nabla\tilde\eta^{\kappa}}_4^2,\norm{\tilde\psi}_4^2, \norm{b_0\cdot\nabla\tilde\psi}_4^2)$.
\subsubsection{Three linear problems}
The first linear problem is the following linear degenerate parabolic problem of $\eta$ with given $\mathfrak{f}^1$:
\begin{equation}
\label{eqforeta}
\begin{cases}
\partial_t\eta-\epsilon(b_0\cdot\nabla)^2\eta=\mathfrak{f}^1 \,\, &\text{in } \Omega,\\
\eta|_{t=0}=\eta_0.
\end{cases}
\end{equation}
\begin{proposition}\label{f1}
Given $\mathfrak{f}^1 \in L^{2}(0,T; H^4(\Omega))$ and suppose that $\eta_0,b_0\cdot\nabla\eta_0 \in H^4(\Omega)$. Then the problem \eqref{eqforeta}
admits a unique solution $\eta$ that achieves
the initial data $\eta_0$ and satisfies
\begin{equation}\label{b1}
\norm{  \eta(t)}_4+ \varepsilon\norm{ b_0\cdot\nabla\eta(t) }_4^2 + \varepsilon^2\int_0^t\norm{ ( b_0\cdot\nabla)^2\eta }_4^2d\tau
\ls  \norm{  \eta_0}_4+ \varepsilon\norm{ b_0\cdot\nabla\eta_0 }_4^2  +\int_0^t\norm{\mathfrak{f}^1}_4^2d\tau.
\end{equation}
\end{proposition}
\begin{proof}
The solvability of \eqref{eqforeta} is classical; one may first prove the existence of weak solutions by using the Galerkin method, and then improve the regularity of weak solutions by employing the standard arguments of using the difference quotients. We omit such a procedure and only focus on the derivation of
the estimate \eqref{b1}.

We apply $D^4$ to the equation \eqref{eqforeta} to find that
\begin{equation}
\label{eqforeta12}
\partial_t (D^4\eta)-\epsilon b_0\cdot\nabla  D^4(b_0\cdot\nabla \eta) =\epsilon\left[D^4,b_0\cdot\nabla\right](b_0\cdot\nabla\eta)+D^4\mathfrak{f}^1  .
\end{equation}
Taking the $L^2(\Omega)$ inner product of \eqref{eqforeta12} with $D^4\eta$, by the integration by parts, we obtain
\begin{equation}
\label{ee0}
\begin{split}
&\dfrac{1}{2}\dfrac{d}{dt}\int_{\Omega}\abs{D^4\eta}^2 +\epsilon\int_{\Omega}\abs{D^4(b_0\cdot\nabla\eta)}^2
\\&\quad=\int_{\Omega}\varepsilon D^4(b_0\cdot\nabla\eta_i)\left[D^4,b_0\cdot\nabla\right]  \eta_i+\epsilon \left[D^4,b_0\cdot\nabla\right](b_0\cdot\nabla\eta_i)D^4\eta_i +D^4 \mathfrak{f}^1\cdot D^4\eta
 \\&\quad\ls \epsilon\norm{b_0\cdot\nabla\eta}_4\norm{b_0}_4\norm{ \eta}_4+\left(\epsilon\norm{ b_0}_4\norm{b_0\cdot\nabla\eta}_4+\norm{\mathfrak{f}^1}_4 \right)\norm{ \eta}_4
  \\&\quad\ls  \left(\epsilon\norm{ b_0}_4\norm{b_0\cdot\nabla\eta}_4+\norm{\mathfrak{f}^1}_4 \right)\norm{ \eta}_4.
\end{split}
\end{equation}
While taking the $L^2(\Omega)$ inner product of \eqref{eqforeta12} with $D^4\dt\eta$ yields
\begin{equation}
\label{ee01}
\begin{split}
&\int_{\Omega}\abs{D^4\dt\eta}^2 +\dfrac{\epsilon}{2}\dfrac{d}{dt}\int_{\Omega}\abs{D^4(b_0\cdot\nabla\eta)}^2
\\&\quad=\int_{\Omega}\varepsilon D^4(b_0\cdot\nabla\eta_i)\left[D^4,b_0\cdot\nabla\right] \dt \eta_i+\epsilon \left[D^4,b_0\cdot\nabla\right](b_0\cdot\nabla\eta_i)D^4\dt\eta_i +D^4 \mathfrak{f}^1\cdot D^4\dt\eta
 \\&\quad\ls  \left(\epsilon\norm{ b_0}_4\norm{b_0\cdot\nabla\eta}_4+\norm{\mathfrak{f}^1}_4 \right)\norm{ \dt\eta}_4.
\end{split}
\end{equation}
These two imply, by using Cauchy's inequality and the equation \eqref{eqforeta},
\begin{equation}
\label{ee0121212}
  \dfrac{d}{dt}\left(\norm{  \eta}_4^2+ \varepsilon\norm{ b_0\cdot\nabla\eta }_4^2\right)+\varepsilon \norm{ b_0\cdot\nabla\eta }_4^2+\varepsilon^2\norm{ ( b_0\cdot\nabla)^2\eta }_4^2
 \ls  \norm{ b_0}_4^2\norm{  \eta}_4^2  +\norm{\mathfrak{f}^1}_4^2.
\end{equation}
An application of the Gronwall lemma of \eqref{ee0121212} yields the estimate \eqref{b1}.
\end{proof}

The second linear problem is the simple transport problem of $v$ with given $\mathfrak{f}^2$:
\begin{equation}
\label{eqforv}
\begin{cases}
	\partial_t v =\mathfrak{f}^2 &\text{in } \Omega,\\
v|_{t=0}=v_0.
\end{cases}
\end{equation}

\begin{proposition}\label{f2}
Given $\mathfrak{f}^2 \in L^{1}(0,T; H^4(\Omega))$ and suppose that $v_0 \in H^4(\Omega)$. Then the problem \eqref{eqforv}
admits a unique solution $v$ that achieves
the initial data $v_0$ and  satisfies
\begin{equation}
\label{b2}
\norm{v(t)}_4  \leq \norm{v_0}_4 + \int_0^t\norm{\mathfrak{f}^2 }_4 \,d\tau .
\end{equation}
\end{proposition}
\begin{proof}
The proof is trivial by integrating \eqref{eqforv} in time directly.
\end{proof}

The last linear problem is the most substantial elliptic problem of $q$ with given $\mathfrak{f}^3$:
\begin{equation}
\label{eqforq}
\begin{cases}
-\tilde\a^\kappa_{ij}\pa_j( \tilde\a^\kappa_{i\ell}\pa_\ell q)= \mathfrak{f}^3 \,\, &\text{in } \Omega,\\
q=0 \,\, &\text{on } \Gamma.
\end{cases}
\end{equation}

\begin{proposition}\label{f3}
Given $\mathfrak{f}^3 \in H^3(\Omega)$. Then the problem \eqref{eqforq}
admits a unique solution $q$ that satisfies
\begin{equation}\label{b3}
\norm{\nabla_{\tilde\a^\kappa} q}_{4}
 \leq  \kappa^{-4} P(\norm{\tilde\eta }_{4})\norm{\mathfrak{f}^3}_{3}.
\end{equation}
\end{proposition}

\begin{proof}
The solvability of \eqref{eqforq} is standard; the point lies in the estimate \eqref{b3} is that it only involves $\norm{\tilde\eta }_{4}$ rather than $\norm{\tilde\eta }_{5}$  for fixed $\kappa>0$.

We will prove \eqref{b3} by following an argument in Theorem 7.2 of \cite{DS_10}. For this, we define $
p=q\circ (\tilde\eta^{\kappa})^{-1}, \mathfrak{h}=\mathfrak{f}^3 \circ (\tilde\eta^{\kappa})^{-1}.
$
Then we have
\begin{equation}\label{jalf}
\begin{cases}
-\Delta p=\mathfrak{h} &\text{in }  \tilde\eta^{\kappa}(\Omega,t),\\
p=0 &\text{on }  \tilde\eta^{\kappa}(\Gamma,t).
\end{cases}
\end{equation}
Let $\tilde\Phi$ be the harmonic extension of $\eta^\kappa\mid_\Gamma$ onto $\Omega$, which inherits the smoothness of $\eta^\kappa$ on $\Gamma$. Indeed, for $s\ge r> 1/2$, by the estimates \eqref{loss}, it holds that
\begin{equation}\label{phies}
\|\tilde\Phi\|_{s}\ls   \abs{\tilde\eta^\kappa}_{s-1/2} \ls  \kappa^{-{(s-r)}}\abs{\tilde\eta^\kappa}_{r-1/2}\ls  \kappa^{-{(s-r)}} \norm{\tilde\eta }_{r}.
\end{equation}
We then define $Q=p\circ \tilde\Phi$ and denote $\tilde{\mathcal{B}}=(\nabla\tilde\Phi)^{-T}$. Then we have
\begin{equation}\label{jalfs}
\begin{cases}
-\tilde{\mathcal{B}}_{ij}\pa_j( \tilde{\mathcal{B}}_{i\ell}\pa_\ell Q)=\mathfrak{h}\circ \tilde\Phi &\text{in } \Omega ,\\
Q=0 &\text{on }  \Gamma.
\end{cases}
\end{equation}
The standard elliptic estimates shows that, by \eqref{phies},
\begin{equation}
\norm{Q}_{5}\leq P(\|\tilde\Phi\|_{5})\|\mathfrak{h}\circ \tilde\Phi\|_{3} \leq P(\|\tilde\Phi\|_{5})\norm{\mathfrak{h}}_{3} \leq  \kappa^{-4} P(\norm{\tilde\eta }_{1})\norm{\mathfrak{h}}_{3}.
\end{equation}
Since $p=Q\circ \tilde\Phi^{-1}  $,
\begin{equation}
\norm{\nabla p \circ \tilde\eta^\kappa}_{4}\ls \norm{\nabla p }_4\norm{\tilde\eta^\kappa}_{4} \leq P(\|\tilde\Phi\|_{5})\norm{\nabla Q }_4\norm{\tilde\eta^\kappa}_{4}
 \leq  \kappa^{-4} P(\norm{\tilde\eta }_{4})\norm{\mathfrak{h}}_{3}.
\end{equation}
This implies \eqref{b3} by the chain rule.
\end{proof}

%We first solve linearization of the $\varepsilon$-problem \eqref{epsapproximate}. Given $(\tilde\eta, \tilde v)$, we can obtain $\tilde \eta^{\kappa}$ by \eqref{etadef} and related $\tilde \a=\a(\tilde \eta^{\kappa})$, (and $\tilde{\mathcal{F}}$, $\tilde J$, etc.) And then $\tilde \fk$ is determined by \eqref{etaaa} for $\tilde\eta, \tilde\eta^{\kappa}, \tilde v$ and we have the following linearized problem:
%\begin{equation}
%\label{lvapproximate}
%\begin{cases}
%\partial_t\eta-\epsilon (b_0\cdot\nabla)^2 \eta=v+\tilde\fk,\,\,&\text{in}\,\,\Omega,\\
%\partial_tv_{i} =- \tilde\a_{is}\partial_sq+(b_0\cdot \nabla)^2 \eta_i,\,\,&\text{in}\,\,\Omega,\\
%\tilde\a_{is}\partial_sv_i=0,\,\,&\text{in}\,\,\Omega,\\
%q=0,\,\,\,&\text{on}\,\,\Gamma,\\
%(\eta,v)|_{t=0}=(\text{Id}, v_0).
%\end{cases}
%\end{equation}
\subsubsection{Solvability of \eqref{epsapproximate}}\label{solv es}
We employ a fixed point argument in order to produce a solution to the linear $\varepsilon$-$\kappa$-approximate problem \eqref{epsapproximate}.

For $0 < T <1$ and $M>0$, we define the metric space in which to work:
\begin{equation}\label{metric_space_def}
\begin{split}
	\mathfrak{X}(M,T ) = &\left\{ (w, \pi, \zeta)\;\vert\; w, \zeta, b_0\cdot\nabla\zeta \in C([0,T]; H^4(\Omega)) \text{ satisfy that } (w, \zeta)\mid_{t=0} = (v_0, \text{Id}) \right.
  \\  &\qquad\qquad\quad\text{and }\left.\norm{  (w, \zeta, b_0\cdot\nabla\zeta  )}_{L^\infty_TH^4}  +  \norm{\left( ( b_0\cdot\nabla)^2\zeta , \nabla_{\tilde\a^\kappa}\pi \right) }_{L^{2}_TH^4} \le M .\right\}
  \end{split}
\end{equation}
Note that $\mathfrak{X}(M,T )$ is a Banach space. We then define a mapping $\mathcal{M}:\mathfrak{X}(M,T )\rightarrow\mathfrak{X}(M,T )$ as $\mathcal{M}(w, \pi, \zeta)=(v, q, \eta)$, where $\eta, v$ and $ q$ are determined as follows. Given $(w, \pi, \zeta)\in\mathfrak{X}(M,T )$, we first define $\eta$ as the solution to  \eqref{eqforeta} with $\mathfrak{f}^1=w+ \psi$ and the initial data $\eta_0=Id$, and $v$ as the solution to \eqref{eqforv} with $\mathfrak{f}^2=- \nabla_{\tilde\a^\kappa}\pi+(b_0\cdot \nabla)^2 \zeta$ and the initial data $v_0$. Moreover, the estimates \eqref{b1} implies
\begin{equation}\label{b1000}
\begin{split}
\norm{  (\eta, b_0\cdot\nabla\eta)}_{L^\infty_TH^4}^2 +  \norm{ ( b_0\cdot\nabla)^2\eta}_{L^2_TH^4}^2
 &\le C_\varepsilon\left( \norm{ b_0 }_4  + \norm{w }_{L^2_TH^4}^2+\norm{\psi }_{L^2_TH^4}^2\right)
   \\& \le C_\varepsilon\left( \norm{ b_0 }_4  +\norm{\psi }_{L^2_TH^4}^2+T M^2\right) ,
\end{split}
\end{equation}
and the estimates \eqref{b2} implies
\begin{equation}
\label{b2000}
\norm{v}_{L^\infty_TH^4}^2  \ls \norm{v_0}_4^2 + \left( \norm{\nabla_{\tilde\a^\kappa}\pi }_{L^1_TH^4}+\norm{(b_0\cdot \nabla)^2 \zeta   }_{L^1_TH^4}\right)^2
\ls\norm{v_0}_4^2   +T  M^2  .
\end{equation}
We now define $q$ as the solution to \eqref{eqforq} with $\mathfrak{f}^3=\partial_t\tilde \a_{i\ell}\partial_{\ell} v_i+ \tilde\a^\kappa_{i\ell}\partial_{\ell}\big((b_0\cdot\nabla)^2\eta_i\big)$, where $v$ and $\eta$ are the functions constructed in the above. Moreover,  the estimates \eqref{b3} implies, by the bounds \eqref{b1000} and \eqref{b2000},
 \begin{equation}\label{b3000}
 \begin{split}
 \norm{\nabla_{\tilde\a^\kappa} q}_{L^2_TH^4}^2
 &\leq  \int_0^t\kappa^{-8} P(\norm{\tilde\eta }_{4})\norm{\partial_t\tilde \a_{i\ell}^\kappa\partial_{\ell} v_i+ \tilde\a^\kappa_{i\ell}\partial_{\ell}\big((b_0\cdot\nabla)^2\eta_i\big)}_{3}^2\,d\tau
 \\&\leq\int_0^t\kappa^{-8} P(\norm{\tilde\eta }_{4},\norm{\dt\tilde\eta }_{4})\left(\norm{  v }_4^2+ \norm{(b_0\cdot\nabla)^2\eta }_{4}^2\right)d\tau
  \\&\le P_{\varepsilon,\kappa}\left(\norm{ b_0 }_4,\norm{v_0}_4,   \norm{(\tilde\eta,\dt\tilde\eta,\psi) }_{L^\infty_TH^4}^2 \right)+ C_\varepsilon T M^2.
 \end{split}
\end{equation}
Hence, if $M$ is taken to be sufficiently large with respect to $b_0,v_0,\tilde\eta,\psi$ and $\varepsilon,\kappa$ and then $0<T<1$ is taken to be sufficiently small (depending on $M$ and $\varepsilon$), then $(v, q, \eta)\in \mathfrak{X}(M,T )$. This implies that the mapping $\mathcal{M}: \mathfrak{X}(M,T )\rightarrow \mathfrak{X}(M,T )$ is well-defined.

We shall now show that the mapping $\mathcal{M}$ has a fixed
 point in the space $\mathfrak{X}(M,T )$ by proving the contraction. Let $(w^i, \pi^i, \zeta^i)\in \mathfrak{X}(M,T )$ and $(v^i, q^i, \eta^i)=\mathcal{M}(w^i, \pi^i, \zeta^i)\in \mathfrak{X}(M,T )$, $i=1,2$. Then we find that $\bar\eta=\eta^1-\eta^2$ solves  \eqref{eqforeta} with $\mathfrak{f}^1=w^1-w^2$ and the initial data $\bar\eta_0=0$, $\bar v=v^1-v^2$ solves \eqref{eqforv} with $\mathfrak{f}^2=- \nabla_{\tilde\a^\kappa}(\pi^1-\pi^2)+(b_0\cdot \nabla)^2 (\zeta^1-\zeta^2)$ and the initial data $\bar v_0=0$, and $\bar q=q^1-q^2$ solves \eqref{eqforq} with $\mathfrak{f}^3=\partial_t\tilde \a_{i\ell}\partial_{\ell}  \bar v_i+ \tilde\a^\kappa_{i\ell}\partial_{\ell}\big((b_0\cdot\nabla)^2\bar\eta_i\big)$.
Moreover, we have the following estimates:
\begin{equation}\label{b10'}
\norm{  (\bar\eta, b_0\cdot\nabla\bar\eta)}_{L^\infty_TH^4}^2+  \norm{ ( b_0\cdot\nabla)^2\bar\eta}_{L^2_TH^4}^2
\le C_\varepsilon\norm{w^1-w^2}_{L^2_TH^4}^2
  \le C_\varepsilon T  \norm{w^1-w^2}_{L^\infty_TH^4}^2,
\end{equation}
\begin{equation}
\label{b20'}
\norm{\bar v}_{L^\infty_TH^4}^2 \le T  \norm{\left(\nabla_{\tilde\a^\kappa}(\pi^1-\pi^2),(b_0\cdot \nabla)^2 (\zeta^1-\zeta^2)\right)}_{L^2_TH^4}^2
\end{equation}
and, by \eqref{b10'} and \eqref{b20'},
 \begin{equation}\label{b30'}
 \begin{split}
 \norm{\nabla_{\tilde\a^\kappa}\bar q}_{L^2_TH^4}^2
 &\leq\int_0^t\kappa^{-8} P(\norm{\tilde\eta }_{4},\norm{\dt\tilde\eta }_{4})\left(\norm{ \bar v}_4^2+ \norm{(b_0\cdot\nabla)^2\bar\eta}_{4}^2\right)d\tau
  \\&\le P_{\varepsilon,\kappa}\left(\norm{ b_0 }_4,\norm{v_0}_4,   \norm{(\tilde\eta,\dt\tilde\eta ) }_{L^\infty_TH^4}^2 \right)\\ & \quad\times T\left(\norm{w^1-w^2}_{L^\infty_TH^4}^2+\norm{\left(\nabla_{\tilde\a^\kappa}(\pi^1-\pi^2),(b_0\cdot \nabla)^2 (\zeta^1-\zeta^2)\right)}_{L^2_TH^4}^2 \right).
 \end{split}
\end{equation}
We then see that if we further restrict $T$ smaller in terms of $b_0,v_0,\tilde\eta$ and $\varepsilon,\kappa$, we then have that the mapping $\mathcal{M}: \mathfrak{X}(M,T )\rightarrow \mathfrak{X}(M,T )$ is  a contraction and therefore admits a unique fixed point $\mathcal{M}(v, q, \eta)=(v, q, \eta)$.

Finally, we must verify that the unique fixed point $(v, q, \eta)$ is a solution to \eqref{epsapproximate}, and the remaining thing is to recover the third equation of \eqref{epsapproximate}. For this, taking $\diverge_{\tilde\a^\kappa}$ to the second equation of \eqref{epsapproximate}, and then using the elliptic problem of $q$, we find that $\dt\left(\diverge_{\tilde\a^\kappa} v\right)=0$. Then the conclusion follows by requiring the initial condition $\diverge v_0=0$.

\subsubsection{$\varepsilon$-independent estimates of \eqref{epsapproximate}}\label{epes}

For each $\varepsilon>0$, by Section \ref{solv es}, there exists a time $T_\varepsilon=T_\varepsilon(b_0,v_0,\tilde\eta,\psi,\kappa)>0$ such that there is a unique solution $(v, q, \eta)=(v(\varepsilon), q(\varepsilon), \eta(\varepsilon))$ to \eqref{epsapproximate} on $[0,T_\varepsilon]$. For notational simplifications, we will not explicitly write the dependence of the solution on $\varepsilon$. The purpose of this section is to derive the $\varepsilon$-independent estimates of the solutions to \eqref{epsapproximate}, which enables us to consider the limit of this sequence of solutions as $\varepsilon\rightarrow 0$.

We define the high order energy functional
\begin{equation}
\mathfrak{E^{\epsilon}}(t):=\norm{v(t)}_4^2+\norm{\eta(t)}_4^2+\norm{b_0\cdot\nabla\eta(t)}_4^2+\epsilon\int_0^t\norm{(b_0\cdot\nabla)^2\eta }_4^2d\tau.
 \end{equation}
We claim that there exists a time $T_\kappa=T_\kappa(b_0,v_0,\tilde\eta,\psi)>0$ such that
\begin{equation}  \label{epsclaim}
\sup_{t\in[0,T_\kappa]} \mathfrak{E}^{\epsilon}(t) \leq 2 M_0 .
\end{equation}

\begin{proof}[Proof of the claim \eqref{epsclaim}]
We divide our proof into several steps.

{\it Step 1: estimates of $\eta$.} Employing the estimates \eqref{b1} with $\mathfrak{f}^1=v+ \psi$ and $\eta_0=Id$, we have
\begin{equation}\label{etap}
\begin{split}
\norm{  \eta(t)}_4^2  + \varepsilon^2\int_0^t\norm{ ( b_0\cdot\nabla)^2\eta }_4^2\,d\tau
&\ls   P(\norm{b_0}_4) +\int_0^t(\norm{v }_4^2+\norm{ \psi }_4^2)\,d\tau
\\&\leq M_0+TP\left(\sup_{t\in[0,T]}\mathfrak{E^{\epsilon}}(t)\right)+T   \norm{ \psi}_{L^\infty_TH^4}^2 .
\end{split}
\end{equation}

{\it Step 2: estimates of $q$.} Repeating the arguments in Proposition \ref{pressure}, we may deduce that
\begin{equation}\label{qqes}
\norm{q}_4^2\leq P\left(\norm{(\eta,b_0\cdot\nabla\eta,v,\tilde \eta ,b_0\cdot\nabla\tilde\eta ,\partial_t\tilde\eta) }_4^2\right).
\end{equation}

{\it Step 3: tangential energy estimates.} We use again Alinac's good unknowns
\begin{equation}
\tilde{\mathcal V}:=\bp^2\Delta_\ast v-\bp^2\Delta_\ast\tilde\eta^{\kappa}\cdot\nabla_{\tilde\a^\kappa} v, \,\,\tilde{\mathcal Q}:=\bp^2\Delta_\ast q-\bp^2\Delta_\ast\tilde\eta^{\kappa}\cdot\nabla_{\tilde\a^\kappa} q.
\end{equation}
As in Section \ref{tan}, we deduce
\begin{equation}\label{tttt1}
 \dfrac{1}{2}\dfrac{d}{dt}\int_{\Omega}\abs{\tilde{\mathcal{V}}}^2   +\int_{\Omega} \bp^2\Delta_\ast (b_0\cdot\nabla\eta_i)   b_0\cdot\nabla \tilde{\mathcal{V}}_i
 =\int_{\Gamma} \pa_3 q \bp^2\Delta_\ast(\Lambda_\kappa^2\tilde\eta_j)\tilde\a^\kappa_{j3} \tilde\a^\kappa_{i\ell}N_\ell\V_i-\tilde {\mathcal{R}}+\int_{\Omega} \tilde F\cdot \tilde{\mathcal{V}}.
\end{equation}
Here $\tilde F$ represents the $F$ in \eqref{eqValpha}  and $\tilde {\mathcal{R}}$ represents the $\mathcal{R}$ in \eqref{gg2}, with $\eta$ replacing with $\tilde\eta$.
We make use of the estimates \eqref{loss} to avoid the loss of derivatives on the boundary $\Gamma$ for fixed $\kappa>0$, by the definition of $\V$ and \eqref{qqes}, to deduce
\begin{equation}\label{eses1}
\begin{split}
&\int_{\Gamma} \pa_3 q \bp^2\Delta_\ast(\Lambda_\kappa^2\tilde\eta_j)\tilde\a^\kappa_{j3} \tilde\a^\kappa_{i\ell}N_\ell\V_i
\ls \abs{\pa_3 q \tilde\a^\kappa_{j3} \tilde\a^\kappa_{i\ell}N_\ell}_{W^{1,\infty}} \abs{\bp^2\Delta_\ast(\Lambda_\kappa^2\tilde\eta_j)}_{1/2}\abs{\V_i}_{-1/2} \\
&\quad\ls \abs{Dq\tilde\a^\kappa\tilde\a^\kappa}_{W^{1,\infty}}\dfrac{1}{\kappa}\abs{\bp^3\tilde\eta}_{1/2}\left(\abs{\bp^3v}_{1/2}
+\abs{\bp^3\tilde\eta}_{1/2}\abs{\tilde\a^\kappa Dv}_{W^{1,\infty}}\right)
\\&\quad\le \dfrac{1}{\kappa} P\left(\norm{(\eta,b_0\cdot\nabla\eta,v,\tilde \eta,b_0\cdot\nabla\tilde\eta,\partial_t\tilde\eta) }_4^2\right) .
\end{split}
\end{equation}
While by using the first equation of \eqref{epsapproximate}, we have
\begin{align}
\label{pp}
&\int_{\Omega}  \bp^2\Delta_\ast (b_0\cdot\nabla\eta_i)    b_0\cdot\nabla \V_i \nonumber
\\\nonumber&\quad=\hal \dfrac{d}{dt}\int_{\Omega}\abs{\bp^2\Delta_\ast (b_0\cdot\nabla\eta)}^2
  -\epsilon\int_{\Omega}  \bp^2\Delta_\ast (b_0\cdot\nabla\eta_i) \bp^2\Delta_\ast (b_0\cdot\nabla)^3\eta_i
\\\nonumber&\quad\quad-\int_{\Omega}  \bp^2\Delta_\ast (b_0\cdot\nabla\eta_i)   \left(\bp^2\Delta_\ast (b_0\cdot\nabla  \psi_i)-\left[\bp^2\Delta_\ast ,b_0\cdot\nabla\right] v_i-b_0\cdot\nabla(\bp^2\Delta_\ast\tilde\eta^{\kappa}\cdot\nabla_{\tilde\a^\kappa} v_i)\right)
\\\nonumber &\quad\ge \hal \dfrac{d}{dt}\int_{\Omega}\abs{\bp^2\Delta_\ast (b_0\cdot\nabla\eta)}^2 + \epsilon\int_{\Omega}\abs{\bp^2\Delta_\ast ((b_0\cdot\nabla)^2\eta)}^2
\\\nonumber&\quad\quad +  \epsilon\int_{\Omega}\bp^2\Delta_\ast (b_0\cdot\nabla\eta_i)\left[\bp^2\Delta_\ast, b_0\cdot\nabla\right](b_0\cdot\nabla)^2\eta_i-\left[b_0\cdot\nabla, \bp^2\Delta_\ast \right](b_0\cdot\nabla\eta_i)\bp^2\Delta_\ast\left((b_0\cdot\nabla)^2\eta_i\right)
\\\nonumber&\quad \quad -  P(\norm{(b_0 ,b_0\cdot\nabla\eta, v,\tilde\eta, b_0\cdot\nabla\tilde\eta,b_0\cdot\nabla\psi)}_4)
\\ \nonumber&\quad\ge \hal \dfrac{d}{dt}\int_{\Omega}\abs{\bp^2\Delta_\ast (b_0\cdot\nabla\eta)}^2 + \epsilon\int_{\Omega}\abs{\bp^2\Delta_\ast((b_0\cdot\nabla)^2\eta)}^2   \\ &\qquad -  P(\norm{(b_0 ,\eta,b_0\cdot\nabla\eta, v,\tilde\eta, b_0\cdot\nabla\tilde\eta,b_0\cdot\nabla\psi)}_4)\left(1+\varepsilon \norm{ ( b_0\cdot\nabla)^2\eta }_4\right).
\end{align}
By estimating the $\tilde F$ and $\tilde {\mathcal{R}}$ terms as usual, integrating \eqref{tttt1} in time directly and then using the Cauchy-Schwarz inequality, by \eqref{eses1} and \eqref{pp}, we  obtain
\begin{equation}\label{vves}
\begin{split}
&\norm{\V(t)}_0^2+\norm{\bp^4  (b_0\cdot\nabla \eta)(t)}_0^2+\epsilon\int_0^t \norm{\bp^4((b_0\cdot\nabla)^2\eta)}_0^2\,d\tau
\\&\quad\leq M_0+\sqrt{T}P_\kappa\left(\sup_{t\in[0,T]} \mathfrak{E}^{\epsilon}(t)\right)  P_\kappa\left(\norm{(  \tilde\eta, b_0\cdot\nabla\tilde\eta,\partial_t\tilde\eta,b_0\cdot\nabla\psi)}_{L^\infty_TH^4}\right).
\end{split}
\end{equation}
Furthermore, by using the fundamental theorem of calculus, we have
\begin{equation}
\label{te}
\begin{split}
\norm{\bp^4  v(t)}_0^2\ls & \norm{\V(t)}_0^2+\norm{\bp^2\Delta_\ast \tilde\eta^{\kappa}(t)}_0^2\norm{\tilde\a^\kappa_{j\ell}\partial_{\ell} v(t)}_{L^{\infty}}^2
\\\ls &\norm{\V(t)}_0^2+t\int_0^t\norm{\bp^2\Delta_\ast \partial_t\tilde\eta^\kappa}_0^2\,d\tau\norm{\tilde\a^\kappa_{j\ell}\partial_{\ell} v(t)}_{L^{\infty}}^2
\\\leq &M_0+\sqrt{T}P_\kappa\left(\sup_{t\in[0,T]} \mathfrak{E}^{\epsilon}(t)\right)  P_\kappa\left(\norm{(  \tilde\eta, b_0\cdot\nabla\tilde\eta,\partial_t\tilde\eta,b_0\cdot\nabla\psi)}_{L^\infty_TH^4}\right).
\end{split}
\end{equation}

{\it Step 4: curl and divergence estimates.}
Applying $\curl_{\tilde\a^\kappa}$ to the second equation of \eqref{epsapproximate} to get
\begin{equation}
\partial_t\left(\curl_{\tilde\a^\kappa}v\right)_i-b_0\cdot\nabla\left(\curl_{\tilde\a^\kappa}\left(b_0\cdot\nabla \eta\right)\right)_i=\left(\left[\curl_{\tilde\a^\kappa}, b_0\cdot\nabla\right](b_0\cdot\nabla \eta)\right)_i+\epsilon_{ij\ell} \partial_t\tilde\a^\kappa_{jm} \partial_mv_{\ell}.
\end{equation}
On the other hand, it follows from the third and first equations that
\begin{align}
\label{divk}
&\Div v=\int_0^t\partial_t\tilde\a^\kappa_{ij}\,d\tau\partial_j v^i,\\
&\partial_t\left(\Div_{\tilde\a^\kappa}(b_0\cdot \nabla \eta)\right)-\epsilon\Div_{\tilde\a^\kappa}\left((b_0\cdot\nabla)^3\eta\right)=\Div_{\tilde\a^\kappa}\left(b_0\cdot\nabla \tilde\psi\right)+\left[\Div_{\tilde\a^\kappa},b_0\cdot\nabla\right] v.
\end{align}
Then following the similar arguments in Section \ref{curle} and by the integration by parts as in \eqref{pp}, we may deduce that
\begin{equation}
\label{curleste}
\begin{split}
&\norm{\Div v(t)}_3^2+\norm{\curl v(t)}_{3}^2+\norm{\Div\left(b_0\cdot\nabla\eta\right)(t)}_3^2 +\norm{\curl\left(b_0\cdot\nabla\eta\right)(t)}_{3}^2
\\&\quad+\epsilon\int_0^t\norm{\Div\left((b_0\cdot\nabla)^2\eta\right) }_3^2+ \norm{\curl\left((b_0\cdot\nabla)^2\eta\right) }_3^2\,d\tau
\\&\quad \le M_0+\sqrt{T}P\left(\sup_{t\in[0,T]} \mathfrak{E}^{\epsilon}(t)\right)  P\left(\norm{(  \tilde\eta, b_0\cdot\nabla\tilde\eta,\partial_t\tilde\eta,b_0\cdot\nabla\psi)}_{L^\infty_TH^4}\right).
\end{split}
\end{equation}

{\it Step 5: synthesis.}
Finally, combining the estimates \eqref{etap}, \eqref{vves}, \eqref{te} and \eqref{curleste}, by using the normal trace estimates \eqref{gga} and the Hodge-type elliptic estimates \eqref{hodd}, we obtain
\begin{equation}
\label{eest}
\begin{split}
\sup_{t\in[0,T]} \mathfrak{E}^{\epsilon}(t)&\leq M_0+\sqrt{T}P_\kappa\left(\sup_{t\in[0,T]} \mathfrak{E}^{\epsilon}(t)\right)  P_\kappa\left(\norm{(  \tilde\eta, b_0\cdot\nabla\tilde\eta,\partial_t\tilde\eta,\psi,b_0\cdot\nabla\psi)}_{L^\infty_TH^4}\right).
\end{split}
\end{equation}
A continuous argument on \eqref{eest} proves the claim \eqref{epsclaim}.
\end{proof}

\subsubsection{Solvability of \eqref{lvapproximate}}
The existence of a unique solution to the linearized $\kappa$-approximate problem \eqref{lvapproximate} is recorded in the following theorem.

\begin{theorem}
\label{linearthm}
Suppose that the initial data $ v_0 \in H^4(\Omega)$ with $\Div v_0=0$ and that $b_0 \in H^4(\Omega)$ satisfies \eqref{bcond}. Then there exists a $T_{\kappa}>0$ and a unique solution $(v, q, \eta)$ to  \eqref{lvapproximate} on $[0, T_{\kappa}]$ that satisfy
\begin{equation}
\norm{v(t)}_4^2+\norm{\eta(t)}_4^2+\norm{b_0\cdot\nabla\eta(t)}_4^2  \leq 2 M_0 .
\end{equation}
\end{theorem}
\begin{proof}
The uniform estimates \eqref{epsclaim} allows us to pass to the limit as $\varepsilon \rightarrow 0$ in \eqref{epsapproximate} to produce a solution to  \eqref{lvapproximate}, while the uniqueness follows by the exactly same arguments.
\end{proof}

\subsection{Iteration scheme}\label{diff}
In order to produce the solution to the nonlinear $\kappa$-approximate problem \eqref{approximate}, we will pass to the limit as $n\rightarrow\infty$ in a sequence of approximate solutions $\{v^{(n)}, q^{(n)}), (\eta^{(n)}\}_{n=0}^\infty$, which will be constructed in the following via the linearization iteration by using the linearized $\kappa$-approximate problem \eqref{lvapproximate}.

In the following construction of $\{(v^{(n)}, q^{(n)}, \eta^{(n)})\}_{n=0}^\infty$, we denote $\a^{\kappa(n)} =\a( \eta^{\kappa(n)})$ with $\eta^{\kappa(n)}$  determined by the $\eta^{(n)}$ through \eqref{etadef} and $\psi^{\kappa(n)}=\psi^{\kappa}(v^{(n)}, \eta^{(n)}, \eta^{\kappa(n)})$ determined by the $\eta^{(n)}, v^{(n)}$ through \eqref{etaaa}. We set $(v^{(0)}, q^{(0)}, \eta^{(0)})=(v^{(1)}, q^{(1)}, \eta^{(1)})=(0, 0, \text{Id})$, and hence $\a^{\kappa(0)}=\a^{\kappa(1)}=I$ and $\psi^{\kappa(0)}=\psi^{\kappa(1)}=0$. Note that $\dt \eta^{(1)}=v^{(1)}+\psi^{\kappa(0)}$. Now suppose that $(v^{(n)}, q^{(n)}, \eta^{(n)})$, $n\geq 1$, and hence $\psi^{\kappa(n)}$ are given such that $v^{(n)}$, $\eta^{(n)}, b_0\cdot\nabla \eta^{(n)}$, $\partial_t \eta^{(n)}=v^{(n)}+\psi^{\kappa(n-1)}$, $\psi^{\kappa(n)}$, $b_0\cdot\nabla \psi^{\kappa(n)}$ $\in L^{\infty}(0,T;H^4(\Omega))$, $  \eta^{(n)}\mid_{t=0}=Id$ and that
\begin{align}
\abs{ J^{\kappa(n)}(t)-1}\leq \dfrac{1}{8} \text{ and } \abs{\a_{ij}^{\kappa(n)}(t)-\delta_{ij}}\leq \dfrac{1}{8} \text{ in }\Omega.
\end{align}
Then by Theorem \ref{linearthm}, we can construct $(v^{(n+1)},$ $q^{(n+1)}, \eta^{(n+1)})$ as the solution to \eqref{lvapproximate} with $\tilde\eta=\eta^{(n)}$ and $\psi=\psi^{\kappa(n)}$. That is,
\begin{equation}  \label{its}
\begin{cases}
\partial_t\eta^{(n+1)} =v^{(n+1)}+\psi^{\kappa(n)} &\text{in } \Omega,\\
\partial_tv^{(n+1)}  +\nabla_{\a^{\kappa(n)}} q^{(n+1)} =(b_0\cdot\nabla)^2\eta^{(n+1)}  &\text{in } \Omega,\\
\Div_{\a^{\kappa(n)}} v^{(n+1)} = 0 &\text{in  }\Omega,\\
q^{(n+1)}=0 & \text{on  }\Gamma,\\
(\eta^{(n+1)}, v^{(n+1)})\mid_{t=0} =(\text{Id}, v_0).
\end{cases}
\end{equation}

We define a higher order energy functional
\begin{equation}
\mathfrak{E}^{(n)}(t):=\norm{v^{(n)}(t)}_4^2+\norm{\eta^{(n)}(t)}_4^2+\norm{b_0\cdot\nabla\eta^{(n)}(t)}_4^2.
\end{equation}
We claim that by taking $T_\kappa$ sufficiently small (only depends on $M_0$ and $\kappa>0$), it holds that
\begin{equation}\label{claim1}
\sup_{t\in[0,T]}\mathfrak{E}^{(n)}(t)\leq 2M_0.
\end{equation}
We shall prove the claim \eqref{claim1} by the induction. First it holds for $n=0,1$. Now suppose that it holds for $ n\le m$ for $m\ge 1$, then we prove that it holds for $n=m+1$. First, it holds from the induction assumption and the estimates \eqref{tes1}--\eqref{fest22} that
\begin{equation}
\norm{\left(\partial_t \eta^{(m)},\psi^{\kappa(m)},b_0\cdot\nabla \psi^{\kappa(m)}\right)}_{L^\infty_TH^4}\le P\left(M_0 \right).
\end{equation}
Then from the estimates \eqref{eest}, we obtain
\begin{equation}
\sup_{t\in[0,T]}\mathfrak{E}^{(m+1)}(t)\leq M_0+\sqrt TP_\kappa\left(\sup_{t\in[0,T]}\mathfrak{E}^{(m+1)}(t)\right)P_\kappa\left(M_0 \right).
\end{equation}
Taking $T_\kappa$ sufficiently small (only depends on $M_0$ and $\kappa>0$), we conclude the claim \eqref{claim1}.

Now we will prove that the sequence $\{(v^{(n)}, q^{(n)}, \eta^{(n)})\}_{n=0}^\infty$ converges in certain strong norm. Let $n\ge 3$ and denote the differences:
\begin{equation}
\quad \bar v^{(n)}=v^{(n+1)}-v^{(n)},\quad\bar q^{(n)}=q^{(n+1)}-q^{(n)}, \bar\eta^{(n)} =\eta^{(n+1)}-\eta^{(n)}.
\end{equation}
Also, we denote
\begin{equation}
\bar\a^{\kappa(n)}=\a^{\kappa(n)}-\a^{\kappa(n-1)},\quad \bar\psi^{\kappa(n)}=\psi^{\kappa(n)}-\psi^{\kappa(n-1)}.
\end{equation}
We find that
\begin{equation}
\label{diffe}
\begin{cases}
\partial_t\bar\eta^{(n)}=\bar v^{(n)}+\bar\psi^{\kappa(n)} &\text{in } \Omega\\
\partial_t\bar v^{(n)}_i+\a^{\kappa(n)}_{ij}\partial_j\bar q^{(n)}=(b_0\cdot \nabla)^2 \bar\eta^{(n)}_i-\bar\a^{\kappa(n)}_{ij}\partial_j q^{(n)} &\text{in } \Omega,\\
\a^{\kappa(n)}_{ij}\partial_j\bar v^{(n)}_i=-\bar\a^{\kappa(n)}_{ij}\partial_j v^{(n)}_i &\text{in } \Omega,\\
\bar q^{(n)}=0 &\text{on }\Gamma,\\
 (\bar \eta^{(n)},\bar v^{(n)})|_{t=0}=(0, 0).
\end{cases}
\end{equation}

We will now estimate the differences in $H^3$ norm. We will use a similar strategy which was used in Sections \ref{apest} and \ref{epes}, and again we divide our estimates into several steps.

{\it Step 1: preliminary estimates of $\bar\a^{\kappa(n)}$ and $\bar\psi^{\kappa(n)}$.}  First, note that
\begin{align}
\bar\a^{\kappa(n)}_{ij}(t) &= \int_0^t\partial_t\left (\a^{\kappa(n)}_{ij}-\a^{\kappa(n-1)}_{ij}\right )\,d\tau
 \\&=\int_0^t \left(\bar\a^{\kappa(n)}_{i\ell}\partial_t\partial_{\ell}\eta^{\kappa(n)}_m\a^{\kappa(n)}_{mj}
+\a^{\kappa(n-1)}_{i\ell}\partial_t\partial_{\ell}\bar\eta^{\kappa(n-1)}_m\a^{\kappa(n)}_{mj}
+\a^{\kappa(n-1)}_{i\ell}\partial_t\partial_{\ell}\eta_m^{\kappa(n-1)}\bar\a^{\kappa(n)}_{mj} \right). \nonumber
\end{align}
Hence, we have, by the first equation of \eqref{diffe},
\begin{equation}\label{in1}
\begin{split}
\norm{\bar\a^{\kappa(n)}_{ij}(t)}_2^2
&\le P(M_0)T^2\left( \norm{\bar\a^{\kappa(n)}_{ij}}_{L^{\infty}_TH^2}^2+\norm{\partial_t \bar\eta^{\kappa(n-1)}}_{L^{\infty}_TH^3}^2\right)
\\&\le P(M_0)T^2\left(\norm{\bar\a^{\kappa(n)}_{ij}}_{L^{\infty}_TH^2}^2+\norm{\left(\bar v^{(n-1)},\bar\psi^{\kappa(n-1)}\right)}_{L^{\infty}_TH^3}^2\right).
\end{split}
\end{equation}
On the other hand, note that
\begin{equation}
-\Delta \bar\psi^{\kappa(n)}=0 \text{ in } \Omega,
\end{equation}
and
\begin{equation}
\begin{split}
\bar\psi^{\kappa(n)}=\Delta_{*}^{-1} \mathbb{P}\bigg(&\Delta_{*}\bar\eta^{(n-1)}_{j}\a^{\kappa(n)}_{j\alpha}\partial_{\alpha}\Lambda_{\kappa}^2 v^{(n)}+\Delta_*\eta_j^{(n-1)}\bar\a^{\kappa(n)}_{j\alpha}\partial_{\alpha}\Lambda_{\kappa}^2 v^{(n)}\\&+\Delta_*\eta_j^{(n-1)}\a^{\kappa(n-1)}_{j\alpha}\partial_{\alpha}\Lambda_{\kappa}^2\bar v^{(n-1)}-\Delta_{*}\Lambda_{\kappa}^2\bar\eta^{(n-1)}_{j}\a^{\kappa(n)}_{j\alpha}\partial_{\alpha} v^{(n)}\\&-\Delta_{*}\Lambda_{\kappa}^2\eta^{(n-1)}_{j}\bar\a^{\kappa(n)}_{j\alpha}\partial_{\alpha}v^{(n)}-\Delta_{*}\Lambda_{\kappa}^2\eta^{(n-1)}_{j}\a^{\kappa(n-1)}_{j\alpha}\partial_{\alpha} \bar v^{(n-1)}\bigg)\ \text{on } \Gamma.
\end{split}
\end{equation}
Thus, by the elliptic estimates, we have
\begin{equation}
\label{in2}
\norm{\bar\psi^{\kappa(n)}}_3^2\ls \abs{\bar\psi^{\kappa(n)}}_{5/2}^2\le P(M_0)\left(\norm{\bar\eta^{(n-1)}}_3^2+\norm{\bar v^{(n-1)}}_2^2+\norm{\bar \a^{\kappa(n)}}_1^2\right).
\end{equation}

Consequently, combining \eqref{in1} and \eqref{in2}, we get
\begin{equation}
\label{dest1}
\norm{\bar\a^{\kappa(n)}(t)}_2^2\le P(M_0) T^2\left(\norm{\left(\bar\a^{\kappa(n)}, \bar\a^{\kappa(n-1)}\right)}_{L^{\infty}_TH^2}^2+\norm{\left(\bar v^{(n-1)}, \bar v^{(n-2)}, \bar\eta^{(n-2)}\right)}_{L^{\infty}_TH^3}^2\right).
\end{equation}

{\it Step 2: estimates of $\bar\eta^{(n)}$.}  By the first equation of \eqref{diffe} and the estimate \eqref{in2}, we have
\begin{equation}
\label{dest2}
\begin{split}
\norm{\bar\eta^{(n)}(t)}_3^2&\leq T^2\norm{\left(\bar v^{(n)},\bar\psi^{\kappa(n)}\right)}_{L^{\infty}_TH^3}^2
\\&\le P(M_0) T^2 \left(\norm{\left(\bar v^{(n)}, \bar v^{(n-1)}, \bar\eta^{(n-1)}\right)}_{L^{\infty}_TH^3}^2+\norm{\bar\a^{\kappa(n)}}_{L^{\infty}_TH^2}^2\right).
\end{split}
\end{equation}

{\it Step 3: estimates of $\bar q^{(n)}$.} Applying $J^{\kappa(n)}\a^{\kappa(n)}_{i\ell}\partial_\ell$ to the second equation of \eqref{diffe}, we have
\begin{equation}
\begin{split}
\partial_\ell(J^{\kappa(n)}\a^{\kappa(n)}_{i\ell}\a^{\kappa(n)}_{ij}\partial_j \bar q^{(n)})=&J^{\kappa(n)}\partial_t\a^{\kappa(n)}_{i\ell}\partial_\ell \bar v^{(n)}_i-J^{\kappa(n)}\partial_t\left(\bar\a^{\kappa(n)}_{i\ell}\partial_\ell v^{(n)}_i\right)
\\&-J^{\kappa(n)}\a^{\kappa(n)}_{i\ell}\partial_\ell\left(\bar\a^{\kappa(n)}_{ij}\partial_jq^{(n)}\right)
+\a^{\kappa(n)}_{i\ell}\partial_\ell\left((b_0\cdot\nabla)^2\bar\eta^{(n)}_i\right).
\end{split}
\end{equation}
By the similar arguments in Section \ref{pressure1}, we obtain
\begin{equation}
\label{dpr}
\begin{split}
\norm{\bar q^{(n)}}_3^2\le P(M_0) \left( \norm{\left( \bar v^{(n)},b_0\cdot\nabla\bar\eta^{(n)}, \bar v^{(n-1)}, \bar v^{(n-2)}, \bar\eta^{(n-2)} \right) }_{3}^2+\norm{\left(\bar\a^{\kappa(n)},\bar\a^{\kappa(n-1)}\right)}_{2}^2\right).
\end{split}
\end{equation}

{\it Step 4: tangential energy estimates.} We will not use \eqref{diffe} to estimate the tangential derivatives of the differences.  We shall again use Alinac's good unknowns:
\begin{equation}
\mathcal V^{(n+1)}=\bp^3v^{(n+1)}-\bp^3\eta^{\kappa(n)}_m\a^{\kappa(n)}_{mj}\partial_j v^{(n+1)},\quad
\mathcal Q^{(n+1)}=\bp^3q^{(n+1)}-\bp^3\eta^{\kappa(n)}_m\a^{\kappa(n)}_{mj}\partial_j q^{(n+1)}.
\end{equation}
We denote $\bar{\mathcal V}^{(n)}=\mathcal V^{(n+1)}-\mathcal V^{(n)}, \bar{\mathcal Q}^{(n)}=\mathcal Q^{(n+1)}-\mathcal Q^{(n)}$, then we have that
\begin{equation}
\label{p1}
\partial_t\bar{\mathcal V}^{(n)}_i+\a^{\kappa(n)}_{ij}\partial_j \bar{\mathcal Q}^{(n)}-(b_0\cdot\nabla)\bp^3(b_0\cdot\nabla\bar\eta^{(n)}_i)=-\bar\a^{\kappa(n)}_{ij}\partial_j\mathcal Q^{(n)}+f^{(n)}
\end{equation}
and
\begin{equation}
\a^{\kappa(n)}_{ij}\partial_j \bar{\mathcal V}^{(n)}_i=-\bar\a^{\kappa(n)}_{ij}\partial_j \mathcal V^{(n)}_i+g^{(n)},
\end{equation}
where
\begin{align*}
f^{(n)}=&\partial_t\left(\bar\partial^3{\bar\eta^{\kappa(n-1)}}_m\a^{\kappa(n)}_{mj}\partial_jv^{(n+1)}_i
+\bp^3\eta^{\kappa(n-1)}_m\bar\a^{\kappa(n)}_{mj}\partial_jv^{(n+1)}_i+\bp^3\eta^{\kappa(n-1)}_m\a^{\kappa(n-1)}_{mj}\partial_j\bar v^{(n)}_i\right)
\\&+\bar\a^{\kappa(n)}_{mj}\partial_{j}\left(\a^{\kappa(n)}_{i\ell}\partial_\ell q^{(n+1)}\right)\bar\partial^3\eta^{\kappa(n)}_m
+\a^{\kappa(n-1)}_{mj}\partial_{j}\left(\bar\a^{\kappa(n)}_{i\ell}\partial_\ell q^{(n+1)}\right)\bar\partial^3\eta^{\kappa(n)}_m
\\&+\a^{\kappa(n-1)}_{mj}\partial_{j}\left(\a^{\kappa(n-1)}_{i\ell}\partial_\ell\bar q^{(n)}\right)\bar\partial^3\eta^{\kappa(n)}_m+\a^{\kappa(n-1)}_{mj}\partial_{j}\left(\a^{\kappa(n-1)}_{i\ell}\partial_\ell q^{(n)}\right)\bar\partial^3\bar\eta^{\kappa(n-1)}_m
\\&-\left[\bar\partial^{2},\bar\a^{\kappa(n)}_{mj}\a^{\kappa(n)}_{i\ell}\bar\partial\right]\partial_{\ell}\eta^{\kappa(n)}_m\partial_j q^{(n+1)}
-\left[\bar\partial^{2},\a^{\kappa(n-1)}_{mj}\bar\a^{\kappa(n)}_{i\ell}\bar\partial\right]\partial_{\ell}\eta^{\kappa(n)}_m\partial_j q^{(n+1)}
\\
&-\left[\bar\partial^{2},\a^{\kappa(n-1)}_{mj}\a^{\kappa(n-1)}_{i\ell}\bar\partial\right]\partial_{\ell}\bar\eta^{\kappa(n-1)}_m\partial_j q^{(n+1)}-\left[\bar\partial^{2},\a^{\kappa(n-1)}_{mj}\a^{\kappa(n-1)}_{i\ell}\bar\partial\right]\partial_{\ell}\eta^{\kappa(n-1)}_m\partial_j \bar q^{(n)}
\\
&-\left[\bar\partial^3, \bar\a^{\kappa(n)}_{ij},\partial_j\right] q^{(n+1)}-\left[\bar\partial^3, \a^{\kappa(n-1)}_{ij},\partial_j\right] \bar q^{(n)}+\left[\bar\partial^3, b_0\cdot\nabla\right]b_0\cdot\nabla\bar\eta^{(n)}_i
\end{align*}
and
\begin{align*}
g^{(n)}=&-\left[\bar\partial^2, \bar\a^{\kappa(n)}_{mj}\a^{\kappa(n)}_{i\ell}\bar\partial\right]\partial_{\ell}\eta^{\kappa(n)}_m\partial_j v^{(n+1)}_i-\left[\bar\partial^2, \a^{\kappa(n-1)}_{mj}\bar\a^{\kappa(n)}_{i\ell}\bar\partial\right]\partial_{\ell}\eta^{\kappa(n)}_m\partial_j v^{(n+1)}_i
\\&-\left[\bar\partial^2, \a^{\kappa(n-1)}_{mj}\a^{\kappa(n-1)}_{i\ell}\bar\partial\right]\partial_{\ell}\bar\eta^{\kappa(n-1)}_m\partial_j v^{(n+1)}_i-\left[\bar\partial^2, \a^{\kappa(n-1)}_{mj}\a^{\kappa(n-1)}_{i\ell}\bar\partial\right]\partial_{\ell}\eta^{\kappa(n-1)}_m\partial_j \bar v^{(n)}_i
\\&-\left[\bar\partial^3, \bar\a^{\kappa(n)}_{ij},\partial_j\right]v^{(n+1)}_i-\left[\bar\partial^3, \a^{\kappa(n-1)}_{ij},\partial_j\right]\bar v^{(n)}_i+\bar\a^{\kappa(n)}_{mj}\partial_{j}\left(\a^{\kappa(n)}_{i\ell}\partial_\ell v^{(n+1)}_i\right)\bar\partial^3\eta^{\kappa(n)}_m
\\
&+\a^{\kappa(n-1)}_{mj}\partial_{j}\left(\bar\a^{\kappa(n)}_{i\ell}\partial_\ell v^{(n+1)}_i\right)\bar\partial^3\eta^{\kappa(n)}_m+\a^{\kappa(n-1)}_{mj}\partial_{j}\left(\a^{\kappa(n-1)}_{i\ell}\partial_\ell\bar v^{(n)}_i\right)\bar\partial^3\eta^{\kappa(n)}_m\\&+\a^{\kappa(n-1)}_{mj}\partial_{j}\left(\a^{\kappa(n-1)}_{i\ell}\partial_\ell v^{(n)}_i\right)\bar\partial^3\bar\eta^{\kappa(n-1)}_m.
\end{align*}
Then we find
\begin{equation}\label{ttttt1}
\begin{split}
 &\dfrac{1}{2}\dfrac{d}{dt}\int_{\Omega}\abs{\bar{\mathcal V}^{(n)}}^2+\int_{\Omega} \nabla_{\a^{\kappa(n)}} \bar{\mathcal Q}^{(n)}\cdot \bar{\mathcal V}^{(n)} +\int_{\Omega} \bar\partial^3\left(b_0\cdot\nabla\bar\eta^{(n)}_i\right)   b_0\cdot\nabla \bar{\mathcal V}^{(n)}_i
 \\&\quad=\int_{\Omega} \left(f^{(n)}-\bar\a^{\kappa(n)}_{ij}\partial_j\mathcal Q^{(n)}\right)\cdot \bar{\mathcal V}^{(n)}.
 \end{split}
\end{equation}

We estimate in an elementary way as usual to deduce that
 \begin{align}
 \label{hg}
&\int_{\Omega} \left(f^{(n)}-\bar\a^{\kappa(n)}_{ij}\partial_j\mathcal Q^{(n)}\right)\cdot \bar{\mathcal V}^{(n)}
\\& \quad\le P(M_0) \left(
 \norm{\left(\bar v^{(n)}, b_0\cdot\nabla\bar\eta^{(n)}, \bar v^{(n-1)},\bar\eta^{(n-1)}, \bar v^{(n-2)}, \bar\eta^{(n-2)} \right)}_{3}^2+\norm{\left(\bar\a^{\kappa(n)}, \bar\a^{\kappa(n-1)}\right)}_{2}^2\right).
\nonumber
 \end{align}
By the integration by parts and since $ \bar{q}^{(n)}=q^{(n)}=0$ on $\Gamma$, using \eqref{loss} and the estimate \eqref{dpr}, we obtain
\begin{align}
\nonumber
& \int_{\Omega} \nabla_{\a^{\kappa(n)}} \bar{\mathcal Q}^{(n)}\cdot \bar{\mathcal V}^{(n)} =\int_{\Gamma} \bar{\mathcal Q}^{(n)}\a^{\kappa(n)}_{i\ell}N_\ell \bar{\mathcal V}^{(n)}_i -\int_{\Omega} \bar{\mathcal Q}^{(n)} \nabla_{\a^{\kappa(n)}}\cdot \bar{\mathcal V}^{(n)}-\int_{\Omega}\pa_\ell (\a^{\kappa(n)}_{i\ell} )\bar{\mathcal Q}^{(n)}\bar{\mathcal V}^{(n)}_i \\\nonumber&\quad=
- \int_{\Gamma} \pa_3 \bar q^{(n)} \bp^3(\Lambda_\kappa^2\eta^{(n)}_j)\a^{\kappa(n)}_{j3} \a^{\kappa(n)}_{i\ell}N_\ell \bar{\mathcal V}^{(n)}_i \\\nonumber&\qquad-\int_{\Gamma} \pa_3 q^{(n)} \left(\bp^3(\Lambda_\kappa^2\bar\eta^{(n-1)}_j)\a^{\kappa(n)}_{j3}+\bp^3(\Lambda_\kappa^2\eta^{(n-1)}_j)\bar\a^{\kappa(n)}_{j3}\right) \a^{\kappa(n)}_{i\ell}N_\ell \bar{\mathcal V}^{(n)}_i \\\nonumber&\qquad+\int_{\Omega} \left(\bar{\mathcal{Q}}^{(n)}\bar\a^{\kappa(n)}_{i\ell}\partial_\ell\mathcal V^{(n)}_i-\bar{\mathcal Q}^{(n)}g^{(n)}-\pa_\ell (\a^{\kappa(n)}_{i\ell} )\bar{\mathcal Q}^{(n)} \bar{\mathcal V}^{(n)}_i\right)
\\&\nonumber\quad\le P(M_0) \left( \dfrac{1}{\kappa}\abs{\bp^2\bar\eta^{(n-1)}}_{1/2}\abs{\bar{\mathcal V}^{(n)}}_{-1/2}+\norm{\bar\a^{\kappa(n)}}_2\abs{\bar{\mathcal V}^{(n)}}_{-1/2}\right.
\\&\nonumber\qquad\qquad\qquad\left.+\norm{\left(\bar v^{(n)}, b_0\cdot\nabla\bar\eta^{(n)},\bar v^{(n-1)},\bar\eta^{(n-1)},\bar v^{(n-2)},\bar\eta^{(n-2)}\right)}_{3}^2+\norm{\left(\bar\a^{\kappa(n)},\bar\a^{\kappa(n-1)}\right)}_{2}^2\right)
\\&\quad\le P_\kappa(M_0)\left(\norm{\left(\bar v^{(n)}, b_0\cdot\nabla\bar\eta^{(n)},\bar v^{(n-1)},\bar\eta^{(n-1)},\bar v^{(n-2)},\bar\eta^{(n-2)}\right)}_{3}^2+\norm{\left(\bar\a^{\kappa(n)},\bar\a^{\kappa(n-1)}\right)}_{2}^2\right).
\label{hg2}
\end{align}
By using the first equation of \eqref{diffe}, we have
\begin{equation}
\begin{split}
 &\int_{\Omega}  \bp^3(b_0\cdot\nabla\bar\eta^{(n)}_i)    b_0\cdot\nabla \bar{\mathcal V}^{(n)}_i
\\ &\quad=\hal \dfrac{d}{dt}\int_{\Omega}\abs{\bp^3(b_0\cdot\nabla\bar\eta^{(n)})}^2
 -\int_{\Omega}  \bp^3(b_0\cdot\nabla\bar\eta^{(n)}_i)  \bp^3(b_0\cdot\nabla  \bar\psi^{\kappa(n)}_i)
\\&\qquad+\int_{\Omega} \bp^3(b_0\cdot\nabla\bar\eta^{(n)}_i) \left([\bar\partial^3,b_0\cdot\nabla] \bar v^{(n)}_i+b_0\cdot\nabla\left(\bp^3\bar\eta^{\kappa(n-1)}_m\a^{\kappa(n)}_{mj}\partial_jv^{(n+1)}_i\right)\right)
\\&\qquad+\int_{\Omega} \bp^3(b_0\cdot\nabla\bar\eta^{(n)}_i) b_0\cdot\nabla\left(
\bp^3\eta^{\kappa(n-1)}_m\bar\a^{\kappa(n)}_{mj}\partial_jv^{(n+1)}_i+\bp^3\eta^{\kappa(n-1)}_m\a^{\kappa(n-1)}_{mj}\partial_j \bar v^{(n)}_i\right)
\\ &\quad\ge\hal \dfrac{d}{dt}\int_{\Omega}\abs{\bp^3(b_0\cdot\nabla\bar\eta^{(n)})}^2
\\&\qquad- P(M_0)\bigg(\norm{\left(\bar v^{(n)},b_0\cdot\nabla\bar\eta^{(n)}, \bar v^{(n-1)},\bar\eta^{(n-1)}, b_0\cdot\nabla\bar\eta^{(n-1)}\right)}_3^2
+\norm{\left(\bar\a^{\kappa(n)},\bar\a^{\kappa(n-1)}\right)}_2^2\bigg).
\end{split}
\label{hg3}
\end{equation}
Here we have used the estimates
\begin{equation}
\norm{b_0\cdot\nabla\bar\psi^{\kappa(n)}}_3^2\leq P(M_0)\left(\norm{\left(\bar v^{(n-1)}, \bar\eta^{(n-1)}, b_0\cdot\nabla\bar\eta^{(n-1)}\right)}_3^2+
\norm{\bar\a^{\kappa(n)}}_2^2\right).
\end{equation}
Hence, integrating \eqref{ttttt1} in time directly  and then using the estimates \eqref{hg}--\eqref{hg3}, we obtain
\begin{align}
\nonumber&\norm{\bar{\mathcal V}^{(n)}(t)}_0^2+\norm{\bp^3 (b_0\cdot\nabla \bar\eta^{(n)})(t)}_0^2
\\&\leq P_\kappa(M_0) T^2\left(\norm{\left(\bar v^{(n)},b_0\cdot\nabla\bar\eta^{(n)},\bar v^{(n-1)},\bar\eta^{(n-1)},b_0\cdot\nabla\bar\eta^{(n-1)},\bar v^{(n-2)},\bar\eta^{(n-2)}\right)}_{L^{\infty}_TH^3}^2\right.
\\&\qquad\qquad\qquad\left.+\norm{\left(\bar\a^{\kappa(n)},\bar\a^{\kappa(n-1)}\right)}_{L^{\infty}_TH^2}^2\right).\nonumber
\end{align}
By the definition of $\bar{\mathcal{V}}^{(n)}$, we have
\begin{align}\label{dest3}
\nonumber&\norm{\bar\partial^3\bar v^{(n)}(t)}_0^2+\norm{\bar\partial^3(b_0\cdot\nabla\bar\eta^{(n)})(t)}_0^2\\&\leq P_\kappa(M_0) T^2\left( \norm{\left(\bar v^{(n)}, b_0\cdot\nabla\bar\eta^{(n)}, \bar v^{(n-1)}, \bar\eta^{(n-1)},b_0\cdot\nabla\bar\eta^{(n-1)},\bar v^{(n-2)},\bar\eta^{(n-2)}\right)}_{L^{\infty}_TH^3}^2\right.
\\&\qquad\qquad\qquad\left.+\norm{\left(\bar\a^{\kappa(n)},\bar\a^{\kappa(n-1)}\right)}_{L^{\infty}_TH^2}^2\right).\nonumber
\end{align}

{\it Step 5: curl and divergence estimates.} It follows from \eqref{diffe} that
\begin{align}
&\Div_{\a^{\kappa(n)}}\bar v^{(n)}=-\bar\a^{\kappa(n)}_{ij}\partial_j v^{(n)}_i,\\
&\partial_t\left(\Div_{\a^{\kappa(n)}}b_0\cdot\nabla\bar\eta^{(n)}\right)=h^{(n)},
\\&
 \partial_t\left(\curl_{\a^{\kappa(n)}}\bar v^{(n)}\right)-b_0\cdot\nabla\left(\curl_{\a^{\kappa(n)}}\left(b_0\cdot\nabla \bar\eta^{(n)}\right)\right) =\phi^{(n)},
\end{align}
where
\begin{equation}\nonumber
\begin{split}
h^{(n)}=&\partial_t \a^{\kappa(n)}_{ij}\partial_j \left(b_0\cdot\nabla\bar\eta^{(n)}_i\right)+\left[\Div_{\a^{\kappa(n)}}, b_0\cdot\nabla\right] \bar v^{(n)} +\Div_{\a^{\kappa(n)}}b_0\cdot\nabla\bar\psi^{\kappa(n)}\\&-b_0\cdot\nabla\left(\bar\a^{\kappa(n)}_{ij}\partial_j v^{(n)}_i\right),
\\
\phi_i^{(n)}=&\left(\left[\curl_{\a^{\kappa(n)}}, b_0\cdot\nabla\right](b_0\cdot\nabla \bar\eta^{(n)})\right)_i+\epsilon_{ij\ell} \partial_t\a^{\kappa(n)}_{jm} \partial_m\bar v^{(n)}_\ell-\epsilon_{ij\ell} \bar\a^{\kappa(n)}_{jm} \partial_t\partial_mv^{(n)}_\ell\\&+b_0\cdot\nabla\left(\epsilon_{ij\ell}\bar\a^{\kappa(n)}_{jm} \partial_m\left(b_0\cdot\nabla\eta^{(n)}_\ell\right)\right)+\epsilon_{ij\ell}\left[\bar\a^{\kappa(n)}_{jm}\partial_m, b_0\cdot\nabla\right]b_0\cdot\nabla\eta^{(n)}_\ell.
\end{split}
\end{equation}
Then following a similar argument in Section \ref{curle}, we can obtain%By following the estimates in subsection \ref{curle}, we obtain
\begin{align}
\label{dest4}
&\norm{\left(\Div \bar v^{(n)} ,\curl \bar v^{(n)},\Div (b_0\cdot\nabla\bar\eta^{(n)}) ,\curl(b_0\cdot\nabla\bar\eta^{(n)})\right)(t)}_2^2\nonumber\\&\quad\leq P_\kappa(M_0) T^2\left( \norm{\left(\bar v^{(n)}, b_0\cdot\nabla\bar\eta^{(n)}, \bar v^{(n-1)}, \bar\eta^{(n-1)}, b_0\cdot\nabla\bar\eta^{(n-1)}\right)}_{L^{\infty}_TH^3}^2\right.
\\&\quad\qquad\qquad\qquad\left.+\norm{\left(\bar\a^{\kappa(n)},\bar\a^{\kappa(n-1)}\right)}_{L^{\infty}_TH^2}^2\right).\nonumber
\end{align}
{\it Step 6: synthesis.} Finally, summing up the estimates \eqref{dest1}, \eqref{dest2}, \eqref{dpr}, \eqref{dest3}, and \eqref{dest4}, by using the trace estimates \eqref{gga} and Hodge-type elliptic estimates \eqref{hodd}, we obtain
\begin{equation}
\begin{split}
&\norm{\left(\bar v^{(n)},\bar\eta^{(n)}, b_0\cdot\nabla\bar\eta^{(n)}\right)(t)}_3^2+\norm{\bar\a^{\kappa(n)}(t)}_2^2\\& \quad\leq P_\kappa(M_0) T^2\bigg( \norm{\left(\bar v^{(n)},b_0\cdot\nabla\bar\eta^{(n)},\bar v^{(n-1)},\bar\eta^{(n-1)}, b_0\cdot\nabla\bar\eta^{(n-1)},\bar v^{(n-2)},\bar\eta^{(n-2)}\right)}_{L^{\infty}_TH^3}^2\\&\qquad\qquad\qquad\quad+\norm{\left(\bar\a^{\kappa(n)},\bar\a^{\kappa(n-1)}\right)}_{L^{\infty}_TH^2}^2\bigg).
\end{split}
\end{equation}
We denote
\begin{equation}\label{normnorm}
\Psi^{(n)}:=\norm{\left(\bar v^{(n)},\bar\eta^{(n)}, b_0\cdot\nabla\bar\eta^{(n)}\right)}_{L^{\infty}_TH^3}^2+\norm{\bar\a^{\kappa(n)}}_{L^{\infty}_TH^2}^2.
\end{equation}
Then by taking $T$ sufficiently small (only depends on $M_0$ and $\kappa>0$), we have
\begin{equation}
\Psi^{(n)}\le \frac{1}{8}\left(\Psi^{(n-1)}+\Psi^{(n-2)}\right).
\end{equation}
This implies $\Psi^{(n)}\le P_\kappa(M_0) 2^{-n}$, which yields the convergence of the sequence $\{(v^{(n)}, q^{(n)},\eta^{(n)})\}$ to a limit $(v, q, \eta)$ in the norm of \eqref{normnorm} as $n\rightarrow\infty$.

\subsection{Solvability of the nonlinear $\kappa$-approximate problem} We now record the existence of a unique solution to \eqref{approximate}.
\begin{theorem}
\label{nonlinearthm}
Suppose that the initial data $ v_0 \in H^4(\Omega)$ with $\Div v_0=0$ and that $b_0 \in H^4(\Omega)$ satisfies \eqref{bcond}. Then there exists a $T_{\kappa}>0$ and a unique solution $(v, q, \eta)$ to  \eqref{nonlinearthm} on $[0, T_{\kappa}]$ that satisfy
\begin{equation}\label{estima}
\norm{v(t)}_4^2+\norm{\eta(t)}_4^2+\norm{b_0\cdot\nabla\eta(t)}_4^2  \leq 2 M_0 .
\end{equation}
\end{theorem}
\begin{proof}
The strong convergence of $\{(v^{(n)}, q^{(n)}, \eta^{(n)})\}$ is more than sufficient to pass to the limit as $n \rightarrow 0$ in \eqref{its} to produce a solution to  \eqref{approximate}. The estimate \eqref{estima} follows from \eqref{claim1}, while the uniqueness follows by the exactly same arguments as that of showing the convergence.
\end{proof}

\section{Local well-posedness of \eqref{eq:mhd}}

In this section, we shall now present the
\begin{proof}[Proof of Theorem \ref{mainthm}]

For each $\kappa>0$, we recover the dependence of the solutions to the $\kappa$-approximate problem \eqref{approximate} on $\kappa$ as $(v(\kappa),q(\kappa),\eta(\kappa))$, which were constructed in Theorem \ref{nonlinearthm}.
The $\kappa$-independent estimates \eqref{bound} imply that  $(v(\kappa),q(\kappa),\eta(\kappa))$ is indeed a solution of \eqref{approximate} on the time interval $[0,T_1]$ and yield a strong convergence of $(v(\kappa),q(\kappa),\eta(\kappa))$ to a limit $(v ,q ,\eta )$, up to extraction of a subsequence, which is more than sufficient for us to pass to the limit as $\kappa\rightarrow0$ in \eqref{approximate} for each $t\in[0,T_1]$. We then find that $(v ,q ,\eta )$ is a strong solution to \eqref{eq:mhd} on $[0,T_1]$ and satisfies the estimates \eqref{enesti}. This shows the existence of solutions to \eqref{eq:mhd}.

For the uniqueness, we suppose that $(v^{i}, q^{i}, \eta^{i}),\ i=1,2,$ are two solutions to \eqref{eq:mhd} which satisfy the estimates \eqref{enesti}. We denote $\bar v=v^{1}-v^{2}, \bar q=q^{1}-q^{2}, \bar\eta=\eta^{1}-\eta^{2}$, and $\bar \a=\a^2-\a^1$, where $\a^i=\a(\eta^{i}),\ i=1,2.$ Then we find that
\begin{equation}
\label{dff}
\begin{cases}
\partial_t\bar\eta=\bar v &\text{in } \Omega,\\
\partial_t\bar v_i + \a^1_{ij}\partial_j\bar q-(b_0\cdot \nabla)^2 \bar\eta_i=\bar\a_{ij}\partial_jq^{2} &\text{in } \Omega,\\
\a^1_{ij}\partial_j\bar v_i=\bar\a_{ij}\partial_jv^{2}_i &\text{in } \Omega,\\
\bar q=0 &\text{on } \Gamma,\\
 (\bar\eta,\bar v)|_{t=0}=(0, 0).
\end{cases}
\end{equation}
To estimate the differences, we define the energy functional
\begin{equation}
\widetilde{\mathfrak{E}}(t)=\norm{\bar\eta}_2^2+\norm{\bar v}_2^2+\norm{b_0\cdot\nabla\bar\eta}_2^2+\abs{\bp^2\bar\eta_i\a^1_{i3}}_0^2.
 \end{equation}
Note that \eqref{dff} is comparable to \eqref{diffe}, and so our estimates will be mostly similar to that for \eqref{diffe} except when treating the estimates corresponding to \eqref{hg2} in tangential energy estimates since it involves the use of \eqref{loss} to avoid the loss of derivatives. The real difference is the treatment of the term
\begin{equation}\label{bbbd}
\int_{\Gamma} \bar{\mathcal Q}\a^{1}_{i\ell}N_\ell \bar{\mathcal V}_i,
 \end{equation}
 where we have defined the good unknows
 \begin{equation}
\mathcal V^{i}=\bp^2 v^{i}-\bp^2\eta^{i}_m\a^{i}_{mj}\partial_j v^{i},\quad
\mathcal Q^{i}=\bp^2 q^{i}-\bp^2\eta^{i}_m\a^{i}_{mj}\partial_j q^{i},\quad i=1,2,
\end{equation}
and denoted
\begin{equation}
\begin{split}
\bar{\mathcal V}&=\mathcal V^{1}-\mathcal V^{2}=\bp^2\bar v-\bp^2\bar\eta_m\a^1_{mj}\partial_j v^1+\bp^2\eta^2_m\bar\a_{mj}\partial_j v^1-\bp^2\eta^2_m\a^2_{mj}\partial_j \bar v,\\\bar{\mathcal Q}&=\mathcal Q^{1}-\mathcal Q^{2}=\bp^2\bar q-\bp^2\bar\eta_m\a^1_{mj}\partial_j q^1+\bp^2\eta^2_m\bar\a_{mj}\partial_j q^1-\bp^2\eta^2_m\a^2_{mj}\partial_j \bar q.
\end{split}
\end{equation}
We shall need to use the geometric transport-type structure to deal with the boundary integral \eqref{bbbd}. Indeed, we have
 \begin{align}
&  \int_{\Gamma} \bar{\mathcal Q}\a^{1}_{i\ell}N_\ell \bar{\mathcal V}_i \nonumber =- \int_{\Gamma} \pa_3 \bar q \bp^2\eta^{2}_j\a^{2}_{j3} \a^{2}_{i\ell}N_\ell \bar{\mathcal V}_i
-\int_{\Gamma} \pa_3 q^{1} \left(\bp^2\bar\eta_j\a^{1}_{j3}+\bp^2\eta^{2}_j\bar\a_{j3}\right)\a^{1}_{i\ell}N_\ell \bar{\mathcal V}_i \\&\nonumber\quad\le \hal \dfrac{d}{dt}\int_{\Gamma}(-\nabla q^1\cdot N)|\a^1_{i3}\bp^2\bar\eta_i|-\int_{\Gamma}\pa_3q^1\a^1_{n3}\bp^2\bar\eta_n\left(\bp^2\eta^2_m\bar\a_{mj}\partial_j v^1_i-\bp^2\eta^2_m\a^2_{mj}\partial_j \bar v_i\right)\a^1_{i\ell}N_\ell\\&\nonumber\qquad-\int_{\Gamma} \pa_3 q^{1} \left(\bp^2\bar\eta_j\a^{1}_{j3}+\bp^2\eta^{2}_j\bar\a_{j3}\right)\a^{1}_{i\ell}N_\ell \bar{\mathcal V}_i
\\&\quad\le \hal \dfrac{d}{dt}\int_{\Gamma}(-\nabla q^1\cdot N)|\a^1_{i3}\bp^2\bar\eta_i| +P(M_0)\left(\norm{\left(\bar v, \bar\eta\right)}_{2}^2+\abs{\a^1_{i3}\bp^2\bar\eta_i}^2\right).
\end{align}
By using the Taylor sign condition we can then finish the tangential energy estimate. This together with the other estimates that follows in a similar way as those of \eqref{diffe} in Section \ref{diff} yields the estimates of the following type:
\begin{equation}
\label{uniq}
 \widetilde{\mathfrak{E}}(t)\leq P(M_0) T\left(\sup_{t\in[0, T]}\widetilde{\mathfrak{E}}(t)\right).
\end{equation}
Then by taking $T_0$ sufficiently small (only depends on $M_0$), we have that $\widetilde{\mathfrak{E}}(t)=0$ for $t\in [0,T_0]$, which implies the uniqueness.
\end{proof}

%\section*{Acknowledgments}
%The author would like to thank Professor Zhouping Xin, Professor Zhen Lei for their helpful discussion. The author would also like to thank the hospitality of the Institute of Mathematical Sciences, CUHK. The first author was in supported by NSFC (grant No. 11601305).


\begin{thebibliography}{99}
\bibitem{AD}
T. Alazard, J. M. Delort.
Global solutions and asymptotic behavior for two dimensional gravity water waves.
\emph{Ann. Sci. \'Ec. Norm. Sup\'er.} \textbf{48}  (2015), no. 5, 1149--1238.

 \bibitem{Alinhac}
S. Alinhac.
Existence d'ondes de rar\'efaction pour des syst\`emes quasi-lin\'eaires hyperboliques multidimensionnels.(French. English summary) [Existence of
rarefaction waves for multidimensional hyperbolic quasilinear systems]
\emph{Comm. Partial Differential Equations} \textbf{14} (1989), no. 2, 173--230.
%\bibitem{AM}
%D. Ambrose, N. Masmoudi.
%Well-posedness of 3D vortex sheets with surface tension,
%\emph{Comm. Math. Sci.} \textbf{5} (2007), 391--430.


\bibitem{Chen_08}
G. Chen, Y. Wang.
Existence and stability of compressible current-vortex sheets in three-dimensional magnetohydrodynamics.
\emph{Arch. Ration. Mech. Anal.} \textbf{187} (2008), no. 3, 369--408.
\bibitem{CL_00}
D. Christodoulou, H. Lindblad.
On the motion of the free surface of a liquid.
\emph{Comm. Pure Appl. Math.} \textbf{53} (2000), no. 12, 1536--1602.


\bibitem{CMST}
J. F. Coulombel, A. Morando, P. Secchi, P. Trebeschi. A priori estimates for 3D incompressible current-vortex sheets. \emph{Comm. Math. Phys.} \textbf{311} (2012), no. 1, 247--275.

\bibitem{CS07}
D. Coutand, S. Shkoller.
Well-posedness of the free-surface incompressible Euler equations with or without surface tension.
\emph{J. Amer. Math. Soc.}  \textbf{20} (2007), no. 3, 829--930.

\bibitem{DS_10}
D. Coutand, S. Shkoller.
A simple proof of well-posedness for the free-surface incompressible Euler equations.
\emph{Discrete Contin. Dyn. Syst. Ser. S.} \textbf{3} (2010), no. 3, 429--449.
\bibitem{GMS1}
P. Germain, N. Masmoudi, J. Shatah.
Global solutions for the gravity water waves equation in dimension 3.
\emph{Ann. of Math. (2)} \textbf{175} (2012), no. 2, 691--754.
\bibitem{GMS2}
P. Germain, N. Masmoudi, J. Shatah.
Global solutions for capillary waves equation.
\emph{Comm. Pure Appl. Math.} \textbf{68} (2015), no. 4, 625--687.
\bibitem{Go_04}
J. Goedbloed, S. Poedts.
Principles of magnetohydrodynamics with applications to laboratory
and astrophysical plasmas, Cambridge University Press, Cambridge, (2004).
\bibitem{Hao_16}
C. Hao.
On the motion of free interface in ideal incompressible MHD.
Preprint (2016), arXiv: 1607.02731.

\bibitem{Hao_13}
C. Hao, T. Luo.
A priori estimates for free boundary problem of incompressible inviscid magnetohydrodynamic flows.
\emph{Arch. Ration. Mech. Anal.} \textbf{212} (2014), no.3, 805--847.

\bibitem{IP}
A. Ionescu, F. Pusateri.
Global solutions for the gravity water waves system in 2D.
\emph{Invent. Math.} \textbf{199} (2015), no. 3, 653--804.

\bibitem{IP2}
A. Ionescu, F. Pusateri.
Global regularity for 2D water waves with surface tension.
\emph{Mem. Amer. Math. Soc.}, to appear.

\bibitem{Lannes}
D. Lannes.
Well-posedness of the water-waves equations.
\emph{J. Amer. Math. Soc.} \textbf{18} (2005), no. 3, 605--654.

\bibitem{Lindblad05}
H. Lindblad.
Well-posedness for the motion of an incompressible liquid with free surface boundary.
\emph{Ann. of Math. (2)} \textbf{162} (2005), no. 1, 109--194.


\bibitem{MasRou}
N. Masmoudi, F. Rousset.
Uniform regularity and vanishing viscosity limit for the free surface Navier-Stokes equations.
\emph{Arch. Ration. Mech. Anal.} (2016), DOI: 10.1007/s00205-016-1036-5.


\bibitem{Mo_14}
A. Morando, Y. Trakhinin, P. Trebeschi.
Well-posedness of the linearized plasma-vacuum interface problem in ideal incompressible MHD.
\emph{Quart. Appl. Math.} \textbf{72} (2014), no. 3, 549--587.


 \bibitem{N}
V. I. Nalimov.
The Cauchy-Poisson problem. (Russian)
\emph{Dinamika Splo$\breve{s}$n. Sredy Vyp. 18 Dinamika $\breve{Z}$idkost. so Svobod. Granicami.} \textbf{254} (1974), 104--210.




\bibitem{Secchi_13}
P. Secchi, Y. Trakhinin.
Well-posedness of the linearized plasma-vacuum interface problem.
\emph{Interfaces Free Bound.} \textbf{15} (2013), no. 3, 323--357.


\bibitem{Secchi_14}
P. Secchi, Y. Trakhinin.
Well-posedness of the plasma-vacuum interface problem.
\emph{Nonlinearity} \textbf{27} (2014), no. 3, 105--169.


\bibitem{SZ}
J. Shatah, C. Zeng.
Geometry and a priori estimates for free boundary problems of the Euler equation.
\emph{Comm. Pure Appl. Math.} \textbf{61} (2008), no. 5, 698--744.
%\bibitem{Sol_12}
%V. Solonnikov.
%Free boundary problems of magnetohydrodynamics in multi-connected domains.
%\emph{Interfaces Free Bound.} \textbf{14} (2012), no. 4, 569--602.
%\bibitem{Sol_141}
%V. Solonnikov.
%$L^p$-theory of free boundary problems of magnetohydrodynamics in simply connected domains.
%Proceedings of the St. Petersburg Mathematical Society. Vol. XV. Advances in mathematical analysis of partial differential equations, Amer. Math. Soc. Transl. Ser. 2, 232, Amer. Math. Soc., Providence, RI, (2014), 245--270.
%\bibitem{Sol_142}
%V. Solonnikov.
%$L^p$-theory of free boundary problems of magnetohydrodynamics in multi-connected domains.
%\emph{Ann. Univ. Ferrara Sez. VII Sci. Mat.} \textbf{60} (2014), no. 1, 263--288.

\bibitem{Sun_15}
Y. Sun, W. Wang, Z. Zhang.
Nonlinear stability of current-vortex sheet to the incompressible MHD equations.
Preprint (2015), arXiv: 1510.02228.

\bibitem{Taylor}
\newblock M. Taylor,
\newblock Partial Differential Equations, Vol. I-III,
\newblock Berlin-Heidelberg-New York: Springer, (1996).


\bibitem{Trak_09}
Y. Trakhinin.
The existence of current-vortex sheets in ideal compressible magnetohydrodynamics.
\emph{Arch. Ration. Mech. Anal.} \textbf{191} (2009), no. 2, 245--310.


\bibitem{Trak_10}
Y. Trakhinin.
On the well-posedness of a linearized plasma-vacuum interface problem in ideal compressible MHD.
\emph{J. Differential Equations} \textbf{249} (2010), no. 10, 2577--2599


\bibitem{Wang_15}
Y. J. Wang, Z. Xin.
Vanishing viscosity and surface tension limits of incompressible viscous surface waves.
Preprint (2015), arXiv: 1504.00152.

\bibitem{Wu1}
S. Wu.
Well-posedness in Sobolev spaces of the full water wave problem in 2-D.
\emph{Invent. Math.} \textbf{130} (1997), no. 1, 39--72.


\bibitem{Wu2}
S. Wu.
Well-posedness in Sobolev spaces of the full water wave problem in 3-D.
\emph{J. Amer. Math. Soc.} \textbf{12} (1999), no. 2, 445--495.


\bibitem{Wu3}
S. Wu.
Almost global wellposedness of the 2-D full water wave problem.
\emph{Invent. Math.} \textbf{177} (2009), no. 1, 45--135.


\bibitem{Wu4}
S. Wu.
Global wellposedness of the 3-D full water wave problem.
\emph{Invent. Math.} \textbf{184} (2011), no. 1, 125--220.



\bibitem{ZZ}
P. Zhang, Z. Zhang.
On the free boundary problem of three-dimensional incompressible Euler equations.
\emph{Comm. Pure Appl. Math.} \textbf{61} (2008), no. 7, 877--940.

\end{thebibliography}
\end{document}